\documentclass[12pt]{article}
\usepackage[a4paper, total={7in, 9.5in}]{geometry}
\usepackage{amsmath,amsfonts,amssymb,amsthm}
\usepackage{color}
\usepackage{caption}
\usepackage{authblk}
\usepackage{graphicx}
\usepackage{hyperref}
\usepackage{ragged2e}

\newtheorem{definition}{Definition}[section]
\newtheorem{theorem}{Theorem}[section]

\usepackage[titletoc,title]{appendix}
\usepackage{booktabs}

\usepackage{caption}
\usepackage{subcaption}
\newtheorem{thm}{Theorem}[section]
\newtheorem{lemma}[thm]{Lemma}

\newcounter{example}[section]
\newenvironment{example}[1][]{\refstepcounter{example}\par\medskip
	\textbf{Test Problem~\theexample. #1} \rmfamily}{\medskip}
\providecommand{\keywords}[1]
{
	\small	
	\textbf{\textit{Keywords---}} #1
}
\title{Entropy stable discontinuous Galerkin methods for ten-moment Gaussian closure equations}
\date{}
\author[$\dagger$]{Biswarup Biswas}
\author[$\dagger$]{Harish Kumar}
\author[$\dagger$]{Anshu Yadav}
\affil[$\dagger$]{Department of Mathematics, Indian Institute of Technology Delhi, New Delhi, India-110016}
\date{}

\begin{document}
	
	\maketitle
\begin{abstract}
 In this article, we propose high order discontinuous Galerkin entropy stable schemes for ten-moment Gaussian closure equations, which is based on the suitable quadrature rules (see \cite{chen2017entropy}). The key components of the proposed method are the use of an entropy conservative numerical flux \cite{sen2018entropy} in each cell and a suitable entropy stable numerical flux at the cell edges. This is then used in the entropy stable DG framework of \cite{chen2017entropy} to obtain entropy stability of the semi-discrete scheme. We also extend these schemes to a source term that models plasma laser interaction. For the time discretization, we use strong stability preserving schemes. The proposed schemes are then tested on several test cases to demonstrate stability, accuracy, and robustness.
\end{abstract}
\section{Introduction}
	
    In several fluid and plasma flow applications assumptions of the {\em local thermodynamic  equilibrium} do not hold (see \cite{dubroca_magnetic_2004,miller_multi-species_2016,brown_numerical_nodate,meng_classical_2012,hirabayashi_new_2016,meena_robust_2019}). Due to this, the simulations of such flows using the scalar description of the pressure (used most commonly in Euler equations of compressible fluid flows) is not accurate and instead a tonsorial pressure (and hence temperature) is needed. One such model is ten-moment Gaussian closure equations (see \cite{berthon_numerical_2006,berthon2015entropy,meena2017positivity,sen2018entropy,meena2018well,meena_robust_2019,Levermore1996,levermore_gaussian_1998}). It is one of the simplest fluid model which consider pressure as tensor.

The system of the ten-moment Gaussian closure equations is a nonlinear system of hyperbolic conservation laws \cite{levermore_gaussian_1998}. So, the solutions of the corresponding Cauchy problem may contain discontinuities, even with smooth initial data. Hence, weak solutions are considered. Furthermore, to rule out the physically irrelevant solutions, an additional criterion in the form of entropy stability is imposed. Due to the presence of the nonlinear flux, the theoretical existence of the solutions is highly unlikely for most of the problems. Hence computational methods are used for most of the applications. Numerical methods for hyperbolic PDEs are often based on the finite volume methods \cite{LeVeque1992}, where higher-order accuracy is achieved using TVD, ENO, or WENO based reconstruction process. Another prevalent method is discontinuous Galerkin (DG) schemes, first developed by Reed and Hill in \cite{reed1973triangular} in the context of neutron transport problems. They were generalized for the time-dependent hyperbolic problems in \cite{Chavent1989} and improved in \cite{Cockburn1989,COCKBURN1988}. These schemes show significant improvement in accuracy in comparison to the finite volume schemes of the equivalent order. 

One of the most important theoretical estimates for the solutions of hyperbolic conservation laws is entropy stability. For the discrete solutions to be physically relevant, we need to ensure that they satisfy a discrete entropy inequality. However, this is highly nontrivial for the higher-order schemes. In \cite{Fjordholm2012}, authors have developed entropy stable higher-order numerical schemes. These are further extended to TVD \cite{Dubey2018,Biswas2018} and WENO schemes\cite{Fjordholm2016}. Entropy stable DG schemes were developed for Shallow water equations in \cite{Gassner2016}. More recently, a general framework of constructing entropy stable DG scheme for hyperbolic conservation laws is proposed by Chen and Shu in \cite{chen2017entropy} and applied to Euler's equations. This framework is also applied to several other interesting hyperbolic conservation laws \cite{Liu2018,duan2019high,biswas2019entropy}. 

For the Ten-Moment equations, several numerical schemes have been developed. In \cite{berthon_numerical_2006,berthon2015entropy} Berthon has developed first-order entropy stable and first-order positivity preserving schemes. The positivity preserving schemes are then extended to higher-order finite volume, DG, and WENO schemes by Meena et al. in \cite{meena2017positivity,meena_robust_2019,meena_positivity-preserving_2020}. Furthermore, in \cite{meena2018well}, authors have proposed a well-balanced scheme for a potential type source terms. For the entropy stability of schemes, authors in \cite{sen2018entropy} developed higher-order entropy stable finite difference schemes.

In this work, we design higher-order entropy stable DG schemes for ten-moment Gaussian closure equations. We proceed as follows:
\begin{itemize}
	\item Following \cite{sen2018entropy}, we first present the entropy framework for the ten-moment Gaussian closure equations.
	\item We then discretize the equations using the suitable quadrature rules presented in \cite{chen2017entropy}. 
	\item To achieve the entropy stability of the scheme, we use entropy conservative numerical flux from \cite{sen2018entropy} and a suitable entropy stable flux at the cell edges. 
	\item The scheme is then extended to include the potential type source terms.
\end{itemize}

The rest of the article is organized as follows: In Section \ref{sec:governing_eqs}, we present the ten-moment Gaussian closure equations and related entropy framework. In Section \ref{sec:1dscheme} we presents the one-dimensional scheme. We prove the accuracy, consistency, and the entropy stability of the schemes. These schemes are then extended to the two-dimensional case in Section \ref{sec:2dscheme}. Extensive numerical results for one and two-dimensional test cases are presented in Section \ref{sec:num_results}.  
	
\section{Ten-moment Gaussian closure equations}\label{sec:governing_eqs}

Following \cite{berthon2015entropy,sen2018entropy,meena2017positivity}, we consider the following two-dimensional ten-moment Gaussian closure model with source terms:
\begin{equation}
\label{ConLaw}
\partial_t \mathbf{u}+\partial_x \mathbf{f}(\mathbf{u})+\partial_y \mathbf{g}(\mathbf{u})=\mathbf{s}^x(\mathbf{u})+\mathbf{s}^y(\mathbf{u}),
\end{equation}
where $\mathbf{u}=\left(\rho, \rho \vec{v},\mathbf{E}\right)^T$ is the vector of the conservative variables. Here $\rho$ is the fluid density, vector $\vec{v}=\left(v^x,v^y\right)^T$ is the fluid velocity and symmetric tensor $\mathbf{E}=\left(E^{xx},E^{xy},E^{yy}\right)^T$ is the energy tensor. The fluxes $\mathbf{f}$ and $\mathbf{g}$ are given by,
\begin{equation*}
\mathbf{f}(\mathbf{u})=\begin{pmatrix}
\rho v^x\\
\rho \left( v^x\right)^2+p^{xx}\\
\rho v^x v^y+p^{xy}\\
\rho {v^x}^3+3v^x p^{xx}\\
\rho \left( v^x\right)^2 v^y+2 v^x p^{xy}+v^{y}p^{xx}\\
\rho v^x \left( v^y\right)^2+v^x p^{yy}+2v^{y}p^{xy}
\end{pmatrix}
\text{, and }
\mathbf{g}(\mathbf{u})=\begin{pmatrix}
\rho v^y\\
\rho v^x v^y+p^{xy}\\
\rho \left( v^y\right)^2+p^{yy}\\
\rho {v^y}\left( v^x\right)^2+v^y p^{xx}+2 v^x p^{xy}\\
\rho \left( v^y\right)^2 v^x+2 v^y p^{xy}+v^{x}p^{yy}\\
\rho \left( v^y\right)^3+3v^{y}p^{yy}
\end{pmatrix}.
\end{equation*}
The source terms in \eqref{ConLaw} are given by,
\begin{equation*}
\mathbf{s}_x(\mathbf{u})=\begin{pmatrix}
0\\
-\dfrac{1}{2}\rho \partial_x W\\
0\\
-\rho v^x \partial_x W\\
-\dfrac{1}{2}\rho v^y \partial_x W\\
0
\end{pmatrix}
\text{, and }
\mathbf{s}_y(\mathbf{u})=\begin{pmatrix}
0\\
0\\
-\dfrac{1}{2}\rho \partial_y W\\
0\\
-\dfrac{1}{2}\rho v^x \partial_y W\\
-\rho v^y \partial_y W
\end{pmatrix},
\end{equation*}
where $W(x,y,t)$ is the given function, which models electron quiver energy in the laser. The system \eqref{ConLaw} is closed by the using equation of state expression,
\begin{equation}
\mathbf{E}=\rho \vec{v}\otimes \vec{v}+\mathbf{p},
\end{equation}
where $\mathbf{p}=\left(p^{xx},p^{xy},p^{yy}\right)^T$ is the symmetric pressure tensor.  For the solutions to be physically admissible, we need density and symmetric pressure tensor to be positive. Hence, we consider the following set $\Omega$ of physically admissible weak solutions:
\begin{equation}
\Omega=\left\{\mathbf{u}\in\mathbb{R}^6\;\;|\,\rho>0,\, \mathbf{x}^T \mathbf{p} \mathbf{x}>0, \;\;\, \forall\, \mathbf{x}\in \mathbb{R}^2 \text{ with } \mathbf{x}\neq \mathbf{0}\right\}.
\end{equation}
For solutions, $\mathbf{u}\in\Omega$, we have the following  results from \cite{levermore_gaussian_1998,berthon_numerical_2006}:

\begin{lemma} The system \eqref{ConLaw}  is hyperbolic for $\mathbf{u}\in\Omega$. Furthermore,  the eigenvalues are given as,
	$$
	\vec{v}.{\bf n}, \vec{v}.{\bf n} \pm \sqrt{\frac{3({\bf p}.{\bf n}).{\bf n}}{\rho}}, \vec{v}.{\bf n} \pm \sqrt{\frac{({\bf p}.{\bf n}).{\bf n}}{\rho}}, 
	$$
	along the unitary vector ${\bf n}$. The multiplicity of the eigenvalue $\vec{ v}.{\bf n}$ is two and multiplicity of all other eigenvalues have is one. In addition, the eigenvalue $\vec{v}.{\bf n}$ is associated to a linearly degenerate field. The eigenvalues $\vec{v}.{\bf n} \pm \sqrt{\frac{3({\bf p}.{\bf n}).{\bf n}}{\rho}}$ are associated to a genuinely nonlinear field while eigenvalues $\vec{v}.{\bf n} \pm \sqrt{\frac{({\bf p}.{\bf n}).{\bf n}}{\rho}}$ are associated to a linearly degenerate field.
\end{lemma}

In order to choose {\em physically relevant} solutions, following \cite{berthon2015entropy,sen2018entropy}, we will now describe the entropy framework. Let us define the following:
\begin{definition}
	A convex function $\mathcal{U}$ is said to be an entropy function for conservation laws \eqref{ConLaw} if there exist smooth functions $\mathcal{F}$ and $\mathcal{G}$ such that \begin{equation}
	\mathcal{F}'=\mathcal{U}'(\mathbf{u})\mathbf{f}'(\mathbf{u}),\,\,\mathcal{G}'=\mathcal{U}'(\mathbf{u})\mathbf{g}'(\mathbf{u}).
	\end{equation}
\end{definition}

For the ten-moment Gaussian clousre equations \eqref{ConLaw}, following \cite{sen2018entropy}, we consider entropy $\mathcal{U}$ and the entropy fluxes as follows:
\begin{equation}
\mathcal{U}=-\rho s,\,\mathcal{F}=\rho v^x s,\,\mathcal{G}=\rho v^y s,
\end{equation}
where $s=\ln\left(\dfrac{\det(\textbf{p})}{\rho^4}\right)$.  For the smooth solutions, we have the following equality (see \cite{sen2018entropy}):
\begin{equation}\label{en_eq}
\dfrac{\partial \mathcal{U}}{\partial t}+\dfrac{\partial \mathcal{F}}{\partial x}+\dfrac{\partial \mathcal{G}}{\partial y}=0,
\end{equation}
which for non-smooth weak solutions become {\em entropy inequality}:
\begin{equation}\label{en_ineq}
\dfrac{\partial \mathcal{U}}{\partial t}+\dfrac{\partial \mathcal{F}}{\partial x}+\dfrac{\partial \mathcal{G}}{\partial y}\leq0.
\end{equation}
Using entropy $\mathcal{U}$, we define entropy variable $\textbf{v}=\partial_{\textbf{u}}\mathcal{U}$. A lengthy calculation results in the following expression for $\mathbf{v}$:
\begin{equation*}
\textbf{v}=\partial_{\textbf{u}}\mathcal{U}=\begin{pmatrix}
4-s- \dfrac{\rho}{\det(\textbf{p})}\left(p^{xx}\left(v^y\right)^2+p^{yy}\left(v^x\right)^2-2p^{xy}v^x v^y\right)\\
\dfrac{2\rho v^y}{\det(\textbf{p})}\\
\dfrac{2\rho v^x}{\det(\textbf{p})}\\
-\dfrac{\rho p^{yy}}{\det(\textbf{p})}\\
\dfrac{2\rho p^{xy}}{\det(\textbf{p})}\\
-\dfrac{\rho p^{xx}}{\det(\textbf{p})}
\end{pmatrix}
\end{equation*}
In the following sections, we aim to design DG schemes which satisfy \eqref{en_ineq} at semi-discrete level. 

\section{Entropy Stable DG Schemes: One dimensional scheme}\label{sec:1dscheme}
To simplify the presentation, we will first present the numerical schemes for the one-dimensional model:
\begin{equation}
\label{1DConLaw}
\partial_t \mathbf{u}+\partial_x \mathbf{f}(\mathbf{u})=\mathbf{s}^x(\mathbf{u}).
\end{equation} 
We discretize the $x$-spatial domain into $N$ elements $I_i=\left[x_{i-\frac{1}{2}},\,x_{i+\frac{1}{2}}\right]\, (1\leq i\leq N)$. Then we seek a solution, \begin{equation*}
\mathbf{w}_h\in \mathbf{V}^{k}_h:=\displaystyle \left\{\mathbf{v}_h:\mathbf{v}_h|_{I_i}\in \left[\mathcal{P}^k(I_i)\right]^6,\,1\leq i\leq N \right\},
\end{equation*} such that for all $\mathbf{v}_h\in \mathbf{V}^{k}_h$ and for all $1\leq i\leq N$,
\begin{equation}\label{dgntform}
\begin{split}
\int_{I_i} \dfrac{\partial \mathbf{w}_h^T}{\partial t} \mathbf{v}_h dx-&\int_{I_i} \mathbf{f}(\mathbf{w}_h)^T\dfrac{d \mathbf{v}_h}{d x}  dx\\ &+ \hat{\mathbf{f}}^T_{i+1/2} \mathbf{v}_h(x_{i+1/2}^-)-\hat{\mathbf{f}}^T_{i-1/2} \mathbf{v}_h(x_{i-1/2}^+)=\int_{I_i} \mathbf{s}^x(\mathbf{w}_h)^T \mathbf{v}_h dx.
\end{split}
\end{equation} Here, $\hat{\mathbf{f}}_{i+1/2}$ is a numerical flux depends on the numerical solutions at element interface, that is, $\hat{\mathbf{f}}_{i+1/2}=\hat{\mathbf{f}}\left(\mathbf{w}_h(x_{i+1/2}^-), \mathbf{w}_h(x_{i+1/2}^+) \right)$. We intend to apply Gauss-Lobatto quadrature to the integrals in \eqref{dgntform}. To present the quadrature rules, we first consider the scalar case. This can be easily extended to  \eqref{dgntform}. Now applying the change of variable 
\begin{equation*}
x=\frac{x_{i+1/2}+x_{i-1/2}}{2}+\frac{1}{2}\xi \Delta x_{i},
\end{equation*} 
we have the following 
\begin{equation*}
\begin{split}
\frac{\Delta x_{i}}{2}\int_{I} \frac{\partial w_h}{\partial t} v_h d\xi&-\int_{I} f(w_h)\frac{d v_h}{d \xi}  d\xi\\
&+{\hat{f}_{i+1/2}} v_h(1)-{\hat{f}_{i-1/2}} v_h(-1)=\frac{\Delta x_{i}}{2}\int_{I} s_x(w_h) v_h d\xi,
\end{split}
\end{equation*}
with $I=[-1,1]$. Given the Gauss–Lobatto quadrature points,
\begin{equation*}
{ -1=\xi_0<\xi_1<\xi_2...<\xi_k=1},
\end{equation*}
and the corresponding weights $\omega_j,\, 0\leq j\leq k$,
we consider the following nodal basis,
\begin{equation*}
{L_j(\xi)=\prod_{l=0,l\neq j}^{N}\frac{\xi-\xi_l}{\xi_j-\xi_l}}.
\end{equation*}
Then $w_h$ is given by, $w_h=\sum_{j=0}^{k}u^i_jL_j(\xi)$. Furthermore, we approximate $f(w_h)$ and $s_x(w_h)$ as,
\begin{equation}\label{approxfh}
f(w_h)\approx f_h(\xi):=\sum_{j=0}^{k}f(u^i_j)L_j(\xi),\;\; s_x(w_h)\approx s_{x,h}(\xi):=\sum_{j=0}^{k}s_x(u^i_j)L_j(\xi).
\end{equation}
Finally, choosing the test function $v_h=L_j$, we have,
\begin{equation}\label{afterInt1}
\frac{\Delta x_{i}}{2}\frac{d}{dt}\langle w_h,\, L_j\rangle_{h}- \langle f_h,\, L'_j\rangle_{h}+{\hat{f}_{i+1/2}} L_j(1)-{\hat{f}_{i-1/2}} L_j(-1)=\frac{\Delta x_{i}}{2}\langle s_{x,h},\, L_j\rangle_{h}. 
\end{equation}
where we use the inner product notation, 
\begin{equation*}
\langle u,\,v\rangle:=\int_{I}uv d\xi,\;\;
\langle u,\,v\rangle _{h}:=\sum_{j=0}^{k}\omega_j u(\xi_j)v(\xi_j).
\end{equation*}
By using \ref{approxfh}, Equation \eqref{afterInt1} can be further written as,
\begin{equation}
\label{afterInt}
\frac{\Delta x_{i}}{2}\sum_{l=0}^{k}\frac{d u_l^i}{dt}\langle L_l,\, L_j\rangle_{h}- \sum_{l=0}^{k}f_l^i\langle L_l,\, L'_j\rangle+{\hat{f}_{i+1/2}} L_j(1)-{\hat{f}_{i-1/2}} L_j(-1)=\frac{\Delta x_{i}}{2}\sum_{l=0}^{k}s_x( u_l^i)\langle L_l,\, L_j\rangle_{h}.
\end{equation}
For a compact from of the scheme, let us define the matrices $D$, $M$ and $S$ as,
\begin{equation*}
\begin{cases}
D_{jl}=L'_l(\xi_j),\\
M_{jl}=\langle L_j,\,L_l\rangle _{h}=\omega_j\delta_{jl},\\
S{jl}=\langle L_j,\,L'_l\rangle _{h}=\langle L_j,\,L'_l\rangle.
\end{cases}
\end{equation*}
Also, define the boundary matrix as, 
\begin{equation}
B=\text{diag}[\tau_0,\tau_1,\dots,\, \tau_k], \text{ where }
\tau_j:=\begin{cases}
-1 & j=0\\
0 & 1\leq j\leq k-1\\
1 & j=k
\end{cases}.
\end{equation}
These matrices are also known as summation-by-parts (SBP) matrices and has the following properties:
	\begin{enumerate}
		\item SBP property (\cite{carpenter2014entropy}):
		\begin{equation}\label{sbp}
		\begin{cases}
		S=MD,\\
		MD+D^TM=S+S^T=B.
		\end{cases}
		\end{equation}
		\item For $0\leq j \leq k$ we have (see \cite{chen2017entropy}), \begin{equation}\label{prop2}
		\sum_{l=0}^{k}D_{jl}=\sum_{l=0}^{k}S_{jl}=0,\,\sum_{l=0}^{k}S_{lj}=\tau_j.
		\end{equation}
	\end{enumerate}

With the SBP matrices defined above, \eqref{afterInt} can be simplified as,
\begin{equation}\label{weaknodal}
\frac{\Delta x_{i}}{2}M\frac{d u^i}{dt}- S^Tf^i+B\hat{f}^i=\frac{\Delta x_{i}}{2}Ms_x^i.
\end{equation}
where, we have used the following notations,
\begin{align*}
u^i&=[u^i_0,\dots, u^i_k]^T\\
f^i&=[f^i_0,\dots, f^i_k]^T\\
s_x^i&=[s_x(u^i_0),\dots, s_x(u^i_k)]^T,\\
\hat{f}^i&=[f_{i-1/2},0,\dots,0, f_{i+1/2}]^T.
\end{align*}
The scheme \eqref{weaknodal} can now be written in the following form,	
\begin{equation}
\frac{d u^i}{dt}+ \frac{2}{\Delta x_{i}}Df^i=\frac{2}{\Delta x_{i}}M^{-1}B(f^i-\hat{f}^i)+s_x^i.
\end{equation}
For a single element, this scheme can be written as (ignoring the element index i),
\begin{equation}\label{OneElement}
\frac{d u_j}{dt}+ \frac{2}{\Delta x}\sum_{l=0}^{k} D_{jl}f_l=\frac{2}{\Delta x}\frac{\tau_j}{\omega_j}(f_j-\hat{f}_j)+(s_x)_j.
\end{equation}
A similar analysis works for the system case (see \cite{chen2017entropy}), and without loss of generality, the scheme \eqref{OneElement} can be extended to \eqref{1DConLaw} as,
\begin{equation}\label{SysOneElement}
\frac{d \mathbf{u}_j}{dt}+ \frac{2}{\Delta x}\sum_{l=0}^{k} D_{jl}\mathbf{f}_l=\frac{2}{\Delta x}\frac{\tau_j}{\omega_j}(\mathbf{f}_j-\hat{\mathbf{f}}_j)+(\mathbf{s}_x)_j.
\end{equation} 
In general, we do not have an entropy stability proof of the scheme \eqref{SysOneElement}. However, a modification to the scheme \eqref{SysOneElement} provides an entropy estimation. First, let us consider the following definitions:
\begin{definition}
	A two point symmetric, consistent numerical flux $\mathbf{f}^*$ is said to be {entropy conservative} flux for an entropy function $\mathcal{U}$ if \begin{equation}
	(\mathbf{v}_R-\mathbf{v}_L)^\text{T}\mathbf{f}^*(\mathbf{u}_R,\,\mathbf{u}_L)=\psi_R-\psi_L,
	\end{equation}
	where $\mathbf{v}=\mathcal{U}'(\mathbf{u})$ is known as entropy variable, and $\psi=\mathbf{v}^\text{T}\cdot\mathbf{f}-\mathcal{F}$ is the entropy potential. $L$ and $R$ in suffix denote the left and right state.
\end{definition}
\begin{definition}
	A two-point symmetric, consistent numerical flux $\mathbf{f}$is said to be {entropy stable} flux for the entropy function $\mathcal{U}$ if \begin{equation}
	(\mathbf{v}_R-\mathbf{v}_L)^\text{T}\mathbf{f}(\mathbf{u}_R,\,\mathbf{u}_L)\leq\psi_R-\psi_L.
	\end{equation}
\end{definition}

We modify \eqref{SysOneElement} as,

\begin{equation}\label{ESDGOneElement}
\frac{d \mathbf{u}_j}{dt}+ \frac{4}{\Delta x}\sum_{l=0}^{k} D_{jl}{\mathbf f}^*(\mathbf{u}_j,\mathbf{u}_l)=\frac{2}{\Delta x}\frac{\tau_j}{\omega_j}(\mathbf{f}_j-\hat{\mathbf{f}}_j)+(\mathbf{s}_x)_j.
\end{equation} 
where ${\mathbf f}^*$ is taken to be entropy conservative.
Then we have the following result.
\begin{theorem}[\cite{chen2017entropy}]
	If ${\mathbf f}^*(\mathbf{u}_j,\mathbf{u}_l)$ is consistent and symmetric, then \eqref{ESDGOneElement} is conservative and atleast k-th order accurate. If we further assume that ${\mathbf f}^*(\mathbf{u}_j,\mathbf{u}_l)$ is entropy conservative, then \eqref{ESDGOneElement} is also  locally entropy conservative within a single element.
\end{theorem}
\begin{proof}
	{\setlength{\parindent}{0cm}\textbf{Conservation:}}
	Ignoring the source term we have,
	\begin{flalign*}
	\frac{d}{dt}\left({\sum_{j=0}^{k}\frac{\Delta x}{2}\omega_j\mathbf{u}_j}\right)&=\sum_{j=0}^{k}\tau_j(\mathbf{\mathbf{f}}_j-\hat{\mathbf{f}}_j)-2\sum_{j=0}^{k}\sum_{l=0}^{k}S_{jl}\mathbf{f}^*(\mathbf{u}_j,\mathbf{u}_l)\\
	&=\sum_{j=0}^{k}\tau_j(\mathbf{f}_j-\hat{\mathbf{f}}_j)-\sum_{j=0}^{k}\sum_{l=0}^{k}(S_{jl}+S_{lj})\mathbf{f}^*(\mathbf{u}_j,\mathbf{u}_l)\\
	&=\sum_{j=0}^{k}\tau_j(\mathbf{f}_j-\hat{\mathbf{f}}_j)-\sum_{j=0}^{k}\sum_{l=0}^{k}B_{jl}\mathbf{f}^*(\mathbf{u}_j,\mathbf{u}_l),\,\,(S+S^T=B)\\
	&=-(\mathbf{f}_{i+1/2}^*-\mathbf{f}_{i-1/2}^*).
	\end{flalign*}
	The above expression shows that the scheme is conservative within an element.
	\par {\setlength{\parindent}{0cm} \textbf{Accuracy:}}
	The accuracy analysis is easier to see in scalar case. Also,  it can be easily extended for the systems case. Let $f^*(x,y)=f^*(u(x),u(y))$ and $f(x)=f(u(x))$. Then,
	\begin{equation*}
	\frac{\partial f}{\partial x}(x)=\frac{\partial f^*}{\partial x}(x,x)+\frac{\partial f^*}{\partial y}(x,x)=2\frac{\partial f^*}{\partial y}(x,x).
	\end{equation*}
	Since the matrix $D$ is exact for polynomials of degree up to $k$,
	\begin{flalign*}
	\frac{4}{\Delta x} \sum_{l=0}^{k}D_{jl}f^*(x(\xi_j),x(\xi_l))&=2\frac{\partial f^*}{\partial y}(x(\xi_j),x(\xi_j))+O({\Delta x}^k)\\&=\frac{\partial f}{\partial x}(x(\xi_j))+O({\Delta x}^k).
	\end{flalign*}
	{\setlength{\parindent}{0cm}\textbf{Entropy conservation:}}
	Entropy production within an single element can be calculated as, 
	\begin{flalign}\label{en_exp}
	\frac{d}{dt}\left(\sum_{j=0}^{k}\frac{\Delta x}{2}\omega_j\mathcal{U}_j\right)&=\sum_{j=0}^{k}\frac{\Delta x}{2}\omega_j\mathbf{v}_j^T\frac{d\mathbf{u}_j}{dt}\nonumber\\
	&=\sum_{j=0}^{k}\tau_j\mathbf{v}_j^T(\mathbf{f}_j-\hat{\mathbf{f}}_j)-2{\sum_{j=0}^{k}\sum_{l=0}^{k}S_{jl}\mathbf{v}_j^T\mathbf{f}^*(\mathbf{u}_j,\mathbf{u}_l)}+\sum_{j=0}^{k}\frac{\Delta x}{2}\omega_j\mathbf{v}_j^T(\mathbf{s}_x)_j\nonumber\\
	&=\sum_{j=0}^{k}\tau_j\mathbf{v}_j^T(\mathbf{f}_j-\hat{\mathbf{f}}_j)-\sum_{j=0}^{k}\sum_{l=0}^{k}(B_{jl}+S_{jl}-S_{lj})\mathbf{v}_j^T\mathbf{f}^*(\mathbf{u}_j,\mathbf{u}_l)\nonumber\\
	& \hspace{150pt}(\text{We have used $S+S^T=B$ and $\mathbf{v}^T\mathbf{s}_x=0$.})\\
	&=-\sum_{j=0}^{k}\tau_j\mathbf{v}_j^T\hat{\mathbf{f}}_j-\sum_{j=0}^{k}\sum_{l=0}^{k}S_{jl}(\mathbf{v}_j-\mathbf{v}_l)^T\mathbf{f}^*(\mathbf{u}_j,\mathbf{u}_l)\nonumber\\
	&=-\sum_{j=0}^{k}\tau_j\mathbf{v}_j^T\hat{\mathbf{f}}_j-\sum_{j=0}^{k}\sum_{l=0}^{k}S_{jl}(\psi_j-\psi_l)\nonumber\\
	&=-\sum_{j=0}^{k}\tau_j\mathbf{v}_j^T\hat{\mathbf{f}}_j+\sum_{j=0}^{k}\tau_j\psi_j,\;\;\;(\text{ We have used \eqref{prop2}.})\nonumber\\
	&=(\psi_k-\mathbf{v}_k^T\hat{\mathbf{f}}_{i+1/2})-(\psi_0-\mathbf{v}_0^T\hat{\mathbf{f}}_{i-1/2}).
	\end{flalign}
	Equation \eqref{en_exp} shows that the scheme is entropy conservative within a single element.
\end{proof}
\begin{theorem}
	If the numerical flux $\hat{\mathbf{f}}$ at the element interface is entropy stable, then the scheme \eqref{ESDGOneElement} is entropy stable.
\end{theorem}
\begin{proof}
	The entropy production rate at the interface is
	\begin{flalign}\label{enstbl_exp}
	(\psi_k^i-(\mathbf{v}_k^i)^T\hat{\mathbf{f}}(\mathbf{u}^{i}_k,\mathbf{u}^{i+1}_0)-(\psi_0^{i+1}-(\mathbf{v}_0^{i+1})^T\hat{\mathbf{f}}(\mathbf{u}^{i}_k,\mathbf{u}^{i+1}_0)\nonumber\\
	=(\mathbf{v}_0^{i+1}-\mathbf{v}_k^i)^T\hat{\mathbf{f}}(\mathbf{u}^{i}_k,\mathbf{u}^{i+1}_0)-(\psi_0^{i+1}-\psi_k^i).
	\end{flalign}
	Expression \eqref{enstbl_exp} clearly shows that if $\hat{\mathbf{f}}$ is entropy stable then together with the compact (or periodic) boundary condition the scheme is globally entropy stable.
\end{proof}

\section{Entropy Stable DG Schemes: Two dimensional scheme}\label{sec:2dscheme}
In this section, we will describe the two dimensional schemes.  We use the rectangular mesh $I_{i,j}=\left[x_{i-\frac{1}{2}},\,x_{i+\frac{1}{2}}\right]\times\left[y_{j-\frac{1}{2}},\,y_{j+\frac{1}{2}}\right]\, (1\leq i\leq N_x), (1\leq j\leq N_y)\,$ with mesh size $\Delta x_i$ and $\Delta y_j$ in $x$ and $y$ direction, respectively. For simplicity, we consider the same number of Gauss-Lobatto points ($k+1$) in the both directions. We use the following change of variables,
\begin{align*}
x_i(\xi)&=\dfrac{1}{2}(x_{i-1/2}+x_{i+1/2})+\dfrac{\xi}{2}\Delta x_i,\\
y_j(\xi)&=\dfrac{1}{2}(y_{j-1/2}+y_{j+1/2})+\dfrac{\xi}{2}\Delta y_j.
\end{align*}
We denote the nodal values as, $\mathbf{u}^{p,q}=\mathbf{w}_h(x_i(\xi_p),y_j(\xi_q))$.
Then for a single element $I_{i,j}$, the scheme is given by,
\begin{eqnarray}\label{2dscheme}
\dfrac{d \mathbf{u}^{p,q}}{dt}&=&-\dfrac{2}{\Delta x} \left(\sum_{l=0}^{k} 2 D_{pl}\mathbf{f}^*(\mathbf{u}^p,\mathbf{u}^l)-\dfrac{\tau_p}{\omega_p}(\mathbf{f}^{p,q}-\hat{\mathbf{f}}^{p,q})\right)\\
&&\qquad-\dfrac{2}{\Delta y}\left(\sum_{l=0}^{k} 2 D_{ql}\mathbf{g}^*(\mathbf{u}^q,\mathbf{u}^l)-\dfrac{\tau_q}{\omega_q}(\mathbf{g}^{p,q}-\hat{\mathbf{g}}^{p,q})\right)+\mathbf{s}_x^{p,q}+\mathbf{s}_y^{p,q},\,\,\,(p,q=0,1,\dots,k)\nonumber
\end{eqnarray}
where we have used the following notations by dropping the indices $i$ and $j$,
\begin{align*}
[\hat{\mathbf{f}}^{0,q},\hat{\mathbf{f}}^{1,q}\dots,\hat{\mathbf{f}}^{k,q}]&=[\hat{\mathbf{f}}_{i-1/2,q},0,\dots,0, \hat{\mathbf{f}}_{i+1/2,q}],\\
[\hat{\mathbf{g}}^{p,0},\hat{\mathbf{g}}^{p,1}\dots,\hat{\mathbf{g}}^{p,k}]&=[\hat{\mathbf{g}}_{p,j-1/2},0,\dots,0, \hat{\mathbf{g}}_{p,j+1/2}],\\
\hat{\mathbf{f}}_{i+1/2,q}&=\hat{\mathbf{f}}\left(\mathbf{w}_h(x_{i+1/2}^-,y_j(\xi_q)), \mathbf{w}_h(x_{i+1/2}^+,y_j(\xi_q)) \right),\\
\hat{\mathbf{g}}_{p,j+1/2}&=\hat{\mathbf{g}}\left(\mathbf{w}_h(x_i(\xi_p),y_{j+1/2}^-), \mathbf{w}_h(x_i(\xi_p),y_{j+1/2}^+) \right).
\end{align*}
Here $\mathbf{f}^*$ and $\mathbf{g}^*$  are entropy conservative fluxes and the interface fluxes $\hat{\mathbf{f}}$ and $\hat{\mathbf{g}}$ are entropy stable. The proof of consistency, accuracy and entropy stability is similar to the one dimensional case (see \cite{chen2017entropy}).

\section{Numerical results}\label{sec:num_results}
For the numerical test cases, we consider the cases with $k=1$, which results in a second-order scheme denoted by ESDG-O2 and $k=2$, which results in a third-order scheme ESDG-O3. For the time integration, we use SSP Runge-Kutta (\cite{gottlieb2001strong}) second and third-order time discretization for ESDG-O2 and ESDG-O3 schemes, respectively.
We make following choices to complete the description of the scheme:
\begin{itemize}
	\item \textbf{Entropy conservative flux}: We use the entropy conservative fluxes $\mathbf{f}^*$ and $\mathbf{g}^*$ presented in \cite{sen2018entropy}, and described in Appendix \ref{ecfluxes}.
	\item \textbf{Entropy stable flux}: For the entropy stable fluxes $\hat{\mathbf{f}},$ and $\hat{\mathbf{g}}$, we use Lax-Friedrich numerical fluxes.
\end{itemize}
For one-dimensional problems, we use the TVDM limiter \cite{cockburn2001runge} in order to avoid the spurious oscillations. The TVDM limiter uses $M=10$ unless stated explicitly. Note that, use of such a limiter at the post-processing does not guaranty the entropy stability. We also use the bound preserving limiter presented in \cite{meena2017positivity}, wherever needed. Bound preserving limiter does not increase the entropy (see  \cite{chen2017entropy}). We use CFL 0.2 for all test cases, where CFL condition is implemented following \cite{chen2017entropy}.

\subsection{One dimensional numerical tests} 
We will now present the results for the one dimensional test cases. 
\begin{example}[Accuracy test (without source):]
	In this test problem, we demonstrate the formal order of accuracy of the schemes without the source terms. Ignoring the source terms by taking $W=0$, we consider the following exact solution $[-0.5,\,0.5]$,
	\begin{equation*}
	\rho(x,t)=2+\sin(2\pi (x-t)),\, v^x(x,0)=1,\,v^y(x,0)=0,\,p^{xx}(x,0)=p^{yy}(x,0)=1,\,p^{xy}(x,0)=0.
	\end{equation*}
	A periodic boundary is used for the computations.
	Error is computed using the exact solution at $T=0.5$. 
	\begin{table}[htb!]
		\centering
		\begin{tabular}{ccccc}
			\hline N  & $L^1$ error  &  Order &  $L^\infty$ error      & Order \\ 
			\hline 32 & 2.58e-02 &  ...  & 2.34e-02 & ... \\ 
			64 & 6.50e-03 &  1.99 & 5.61e-03 & 2.06 \\ 
			128 & 1.63e-03 &  2.00 & 1.39e-03 & 2.02 \\ 
			256 & 4.07e-04 &  2.00 & 3.45e-04 & 2.00 \\ 
			512 & 1.02e-04 &  2.00 & 8.63e-05 & 2.00 \\ 
			\hline 
		\end{tabular} 
		\caption{$L^1$ and $L^{\infty}$ errors and order of accuracy for density using ESDG-O2}
		\label{tab:acc2}
	\end{table}
	\begin{table}[htb!]
		\centering
		\begin{tabular}{ccccc}
			\hline N  & $L^1$ error  &  Order &  $L^\infty$ error      & Order \\ 
			\hline 32 & 6.27e-04 &  ...  & 5.64e-04 & ... \\ 
			64 & 8.34e-05 &  2.91 & 7.46e-05 & 2.92 \\ 
			128 & 1.06e-05 &  2.97 & 9.40e-06 & 2.99 \\ 
			256 & 1.34e-06 &  2.99 & 1.17e-06 & 3.00 \\ 
			512 & 1.67e-07 &  3.00 & 1.46e-07 & 3.00 \\ 
			\hline 
		\end{tabular} 
		\caption{$L^1$ and $L^{\infty}$ errors and order of accuracy for density using ESDG-O3}
		\label{tab:acc3}
	\end{table}
	We present the $L^1$ and $L^{\infty}$ errors and order of accuracy for the density in Tables\ref{tab:acc2} using the ESDG-O2 scheme. We observe that the schemes have reached the desired second order of accuracy. Similarly, in Table \ref{tab:acc3}, we have presented errors for ESDG-O3. Here again, we observe that the proposed scheme has a third-order accuracy.
\end{example}	

\begin{example}[Accuracy test (with source):]
	In this test problem, we check the order of accuracy of the schemes by considering a time dependent source corresponds to the potential, $W=\sin(2\pi (x-t))$. We consider the following smooth exact solution in domain $x\in[-0.5,\,0.5]$ for this test case:
	\begin{flalign*}
	&\rho(x,t)=2+\sin(2\pi (x-t)),\, v^x(x,0)=1,\,v^y(x,0)=0,
	\\
	&p^{xx}(x,t)=1.5+\dfrac{1}{8}\left(\cos(4 \pi (x-t))-8\sin(2 \pi (x-t))\right),\,p^{xy}(x,0)=0,\,p^{yy}(x,0)=1.
	\end{flalign*}
		\begin{table}[htb!]
		\centering
		\begin{tabular}{ccccc}
			\hline N  & $L^1$ error  &  Order &  $L^\infty$ error      & Order \\ 
			\hline 32 & 9.48e-02 &  ...  & 1.71e-01 & ... \\ 
			64 & 2.50e-02 &  1.93 & 4.76e-02 & 1.84 \\ 
			128 & 6.32e-03 &  1.98 & 1.31e-02 & 1.86 \\ 
			256 & 1.58e-03 &  2.00 & 3.41e-03 & 1.95 \\ 
			512 & 3.96e-04 &  2.00 & 8.62e-04 & 1.98 \\ 
			\hline 
		\end{tabular} 
		\caption{Accuracy table of ESDG-O2}
		\label{tab:acc02}
	\end{table}
	\begin{table}[htb!]
		\centering
		\begin{tabular}{ccccc}
			\hline N  & $L^1$ error  &  Order &  $L^\infty$ error      & Order \\ 
			\hline 32 & 1.51e-03 &  ...  & 2.45e-03 & ... \\ 
			64 & 2.30e-04 &  2.71 & 4.05e-04 & 2.59 \\ 
			128 & 3.17e-05 &  2.86 & 5.07e-05 & 3.00 \\ 
			256 & 4.23e-06 &  2.91 & 5.91e-06 & 3.10 \\ 
			512 & 5.83e-07 &  2.86 & 1.00e-06 & 2.56 \\
			\hline 
		\end{tabular} 
		\caption{Accuracy table of ESDG-O3}
		\label{tab:acc03}
	\end{table}
	We present  the results at $T=0.5$ using periodic boundary. The errors and order of accuracy for density are presented in Table \ref{tab:acc02} and \ref{tab:acc03} for ESDG-O2 and ESDG-O3 schemes, respectively. In both cases, we observe that the schemes have the desired order of accuracy.
\end{example}

	\begin{figure}[htb!]
	\centering
	\begin{subfigure}[b]{0.45\textwidth}
		\includegraphics[width=\textwidth]{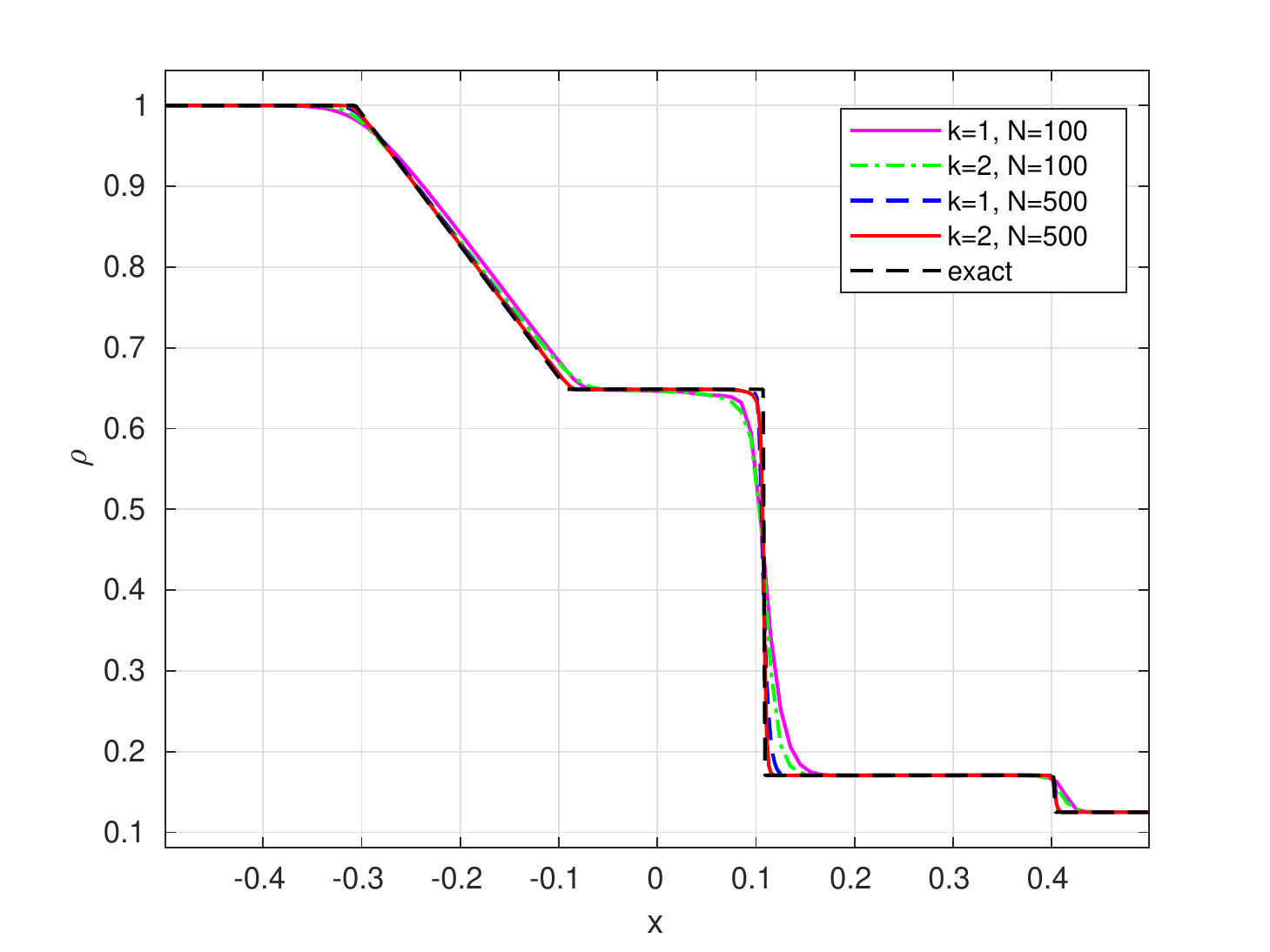}
		\caption{$\rho$}
	\end{subfigure}	
	\begin{subfigure}[b]{0.45\textwidth}
		\includegraphics[width=\textwidth]{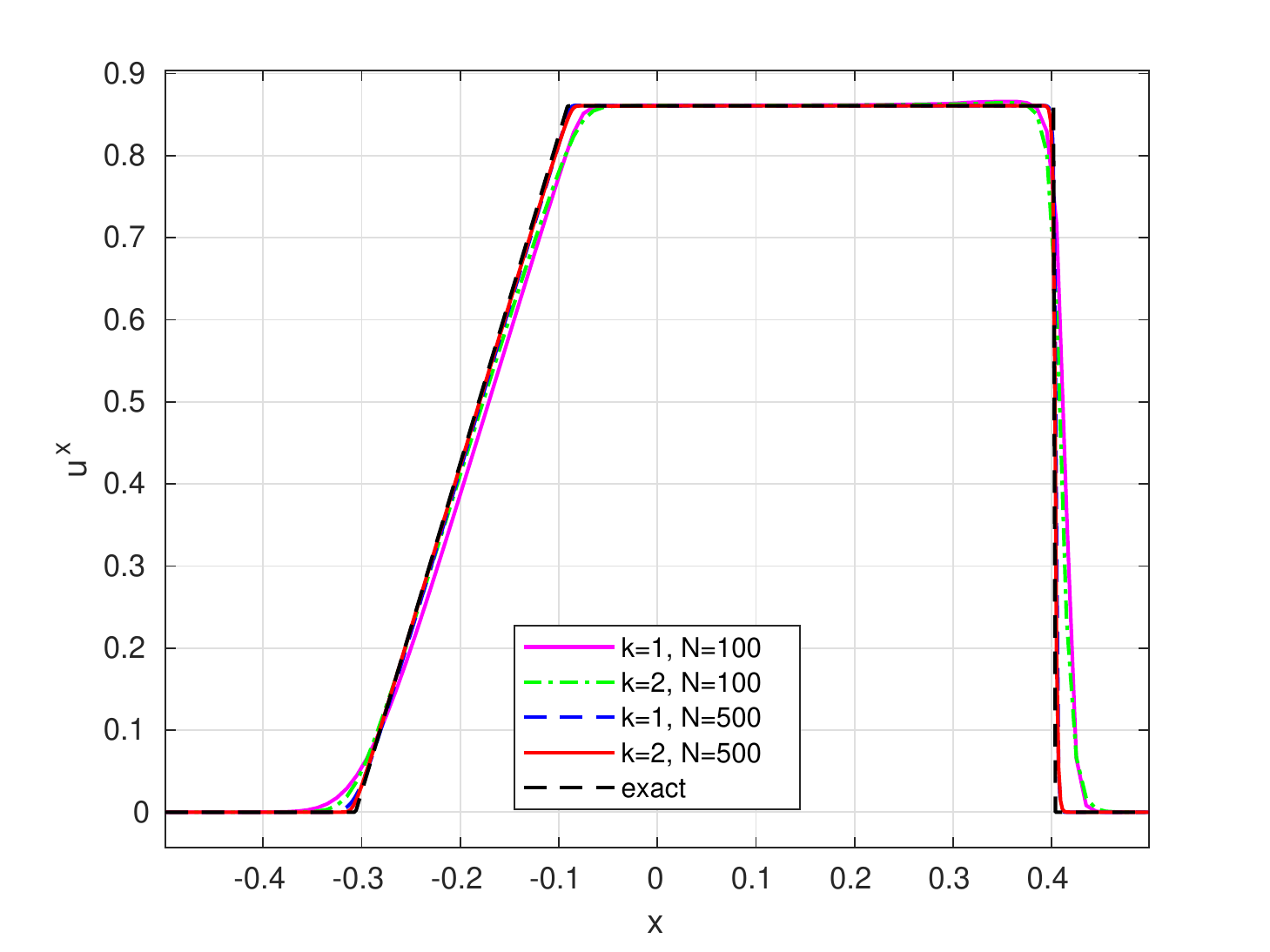}
		\caption{$v^x$}
	\end{subfigure}	
	\begin{subfigure}[b]{0.45\textwidth}
		\includegraphics[width=\textwidth]{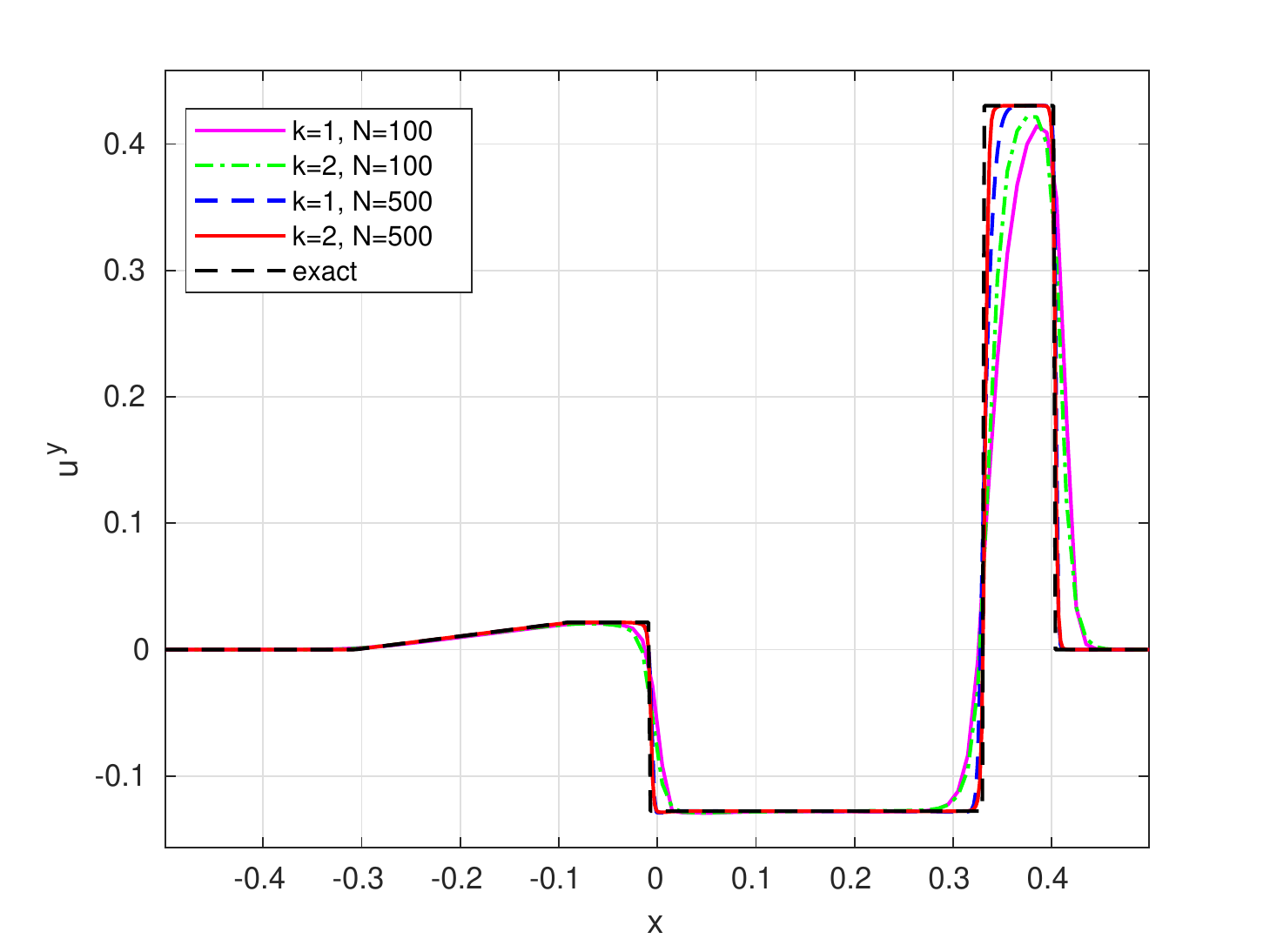}
		\caption{$v^y$}
	\end{subfigure}	
	\begin{subfigure}[b]{0.45\textwidth}
		\includegraphics[width=\textwidth]{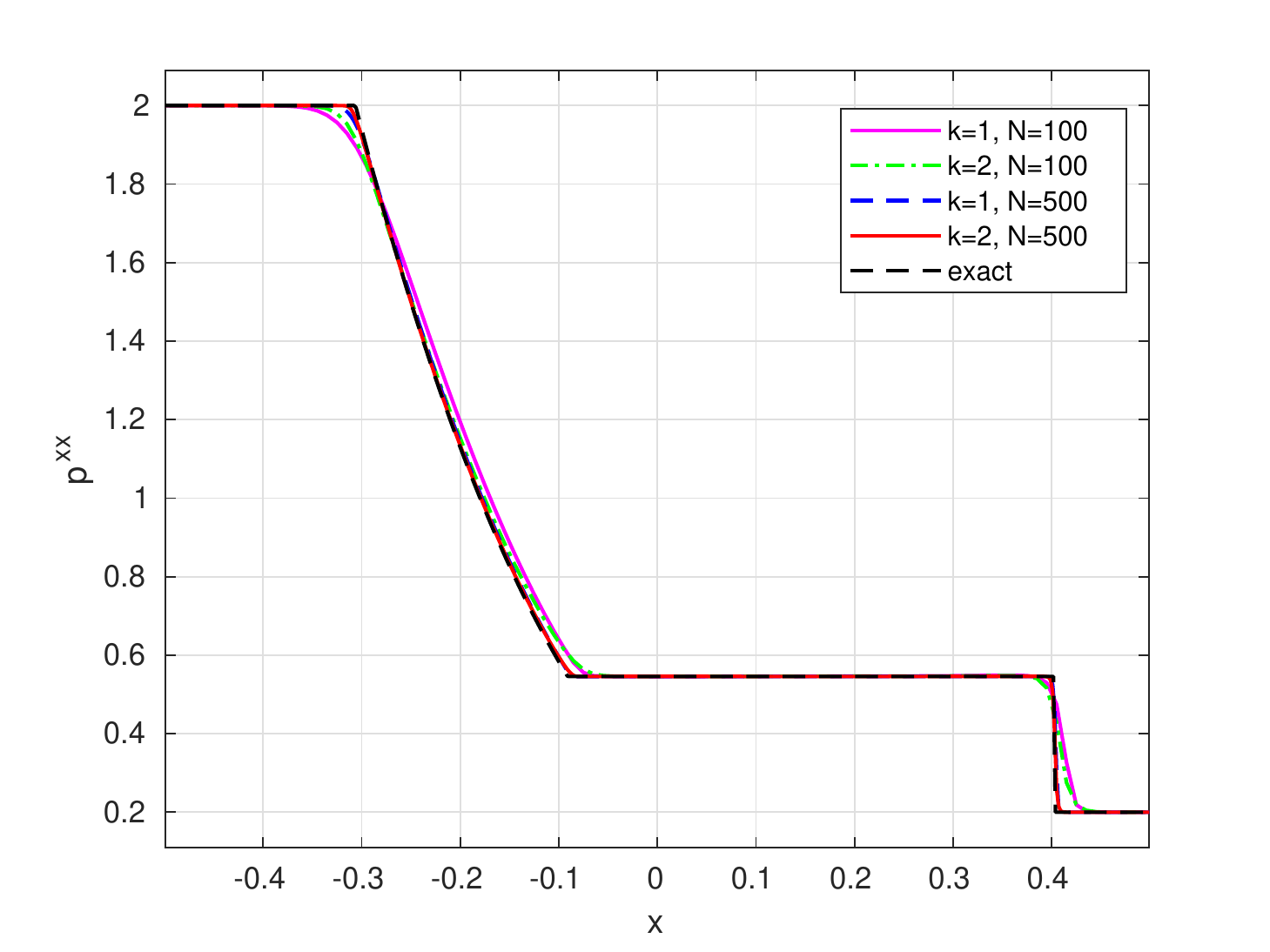}
		\caption{$p^{xx}$}
	\end{subfigure}	
	\begin{subfigure}[b]{0.45\textwidth}
		\includegraphics[width=\textwidth]{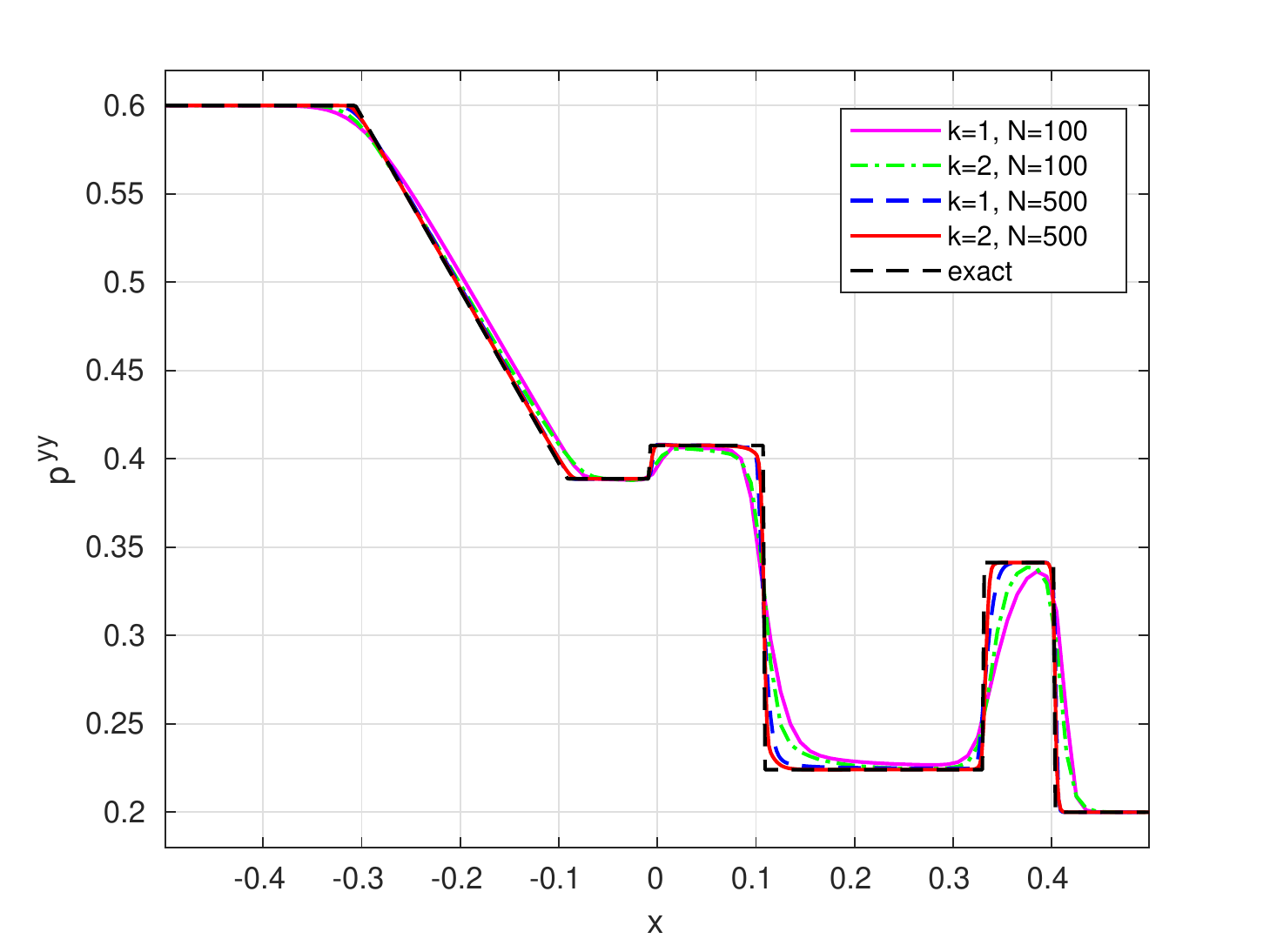}
		\caption{$p^{yy}$}
	\end{subfigure}	
	\begin{subfigure}[b]{0.45\textwidth}
	\includegraphics[width=\textwidth]{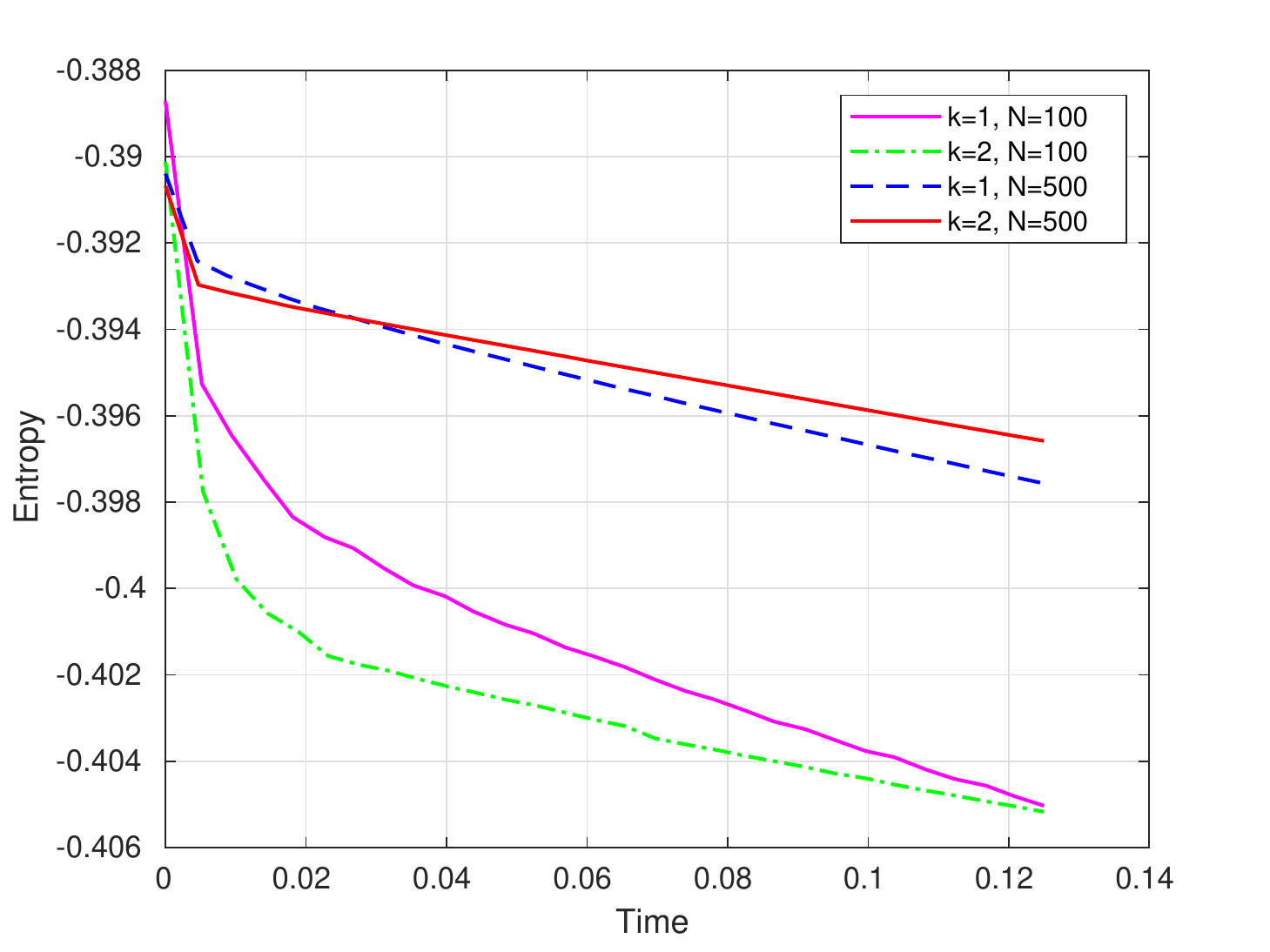}
	\caption{Evolution of total entropy}
\end{subfigure}	
	\caption{Test Problem 3 (Sod shock tube problem): Plot of density, velocity, pressure components and total entropy evolution for ESDG-O2(k=1) and ESDG-O3(k=2) using 100 and 500 cells.}
	\label{fig:pb2}
\end{figure}
\begin{figure}[htb!]
	\centering
	\begin{subfigure}[b]{0.45\textwidth}
		\includegraphics[width=\textwidth]{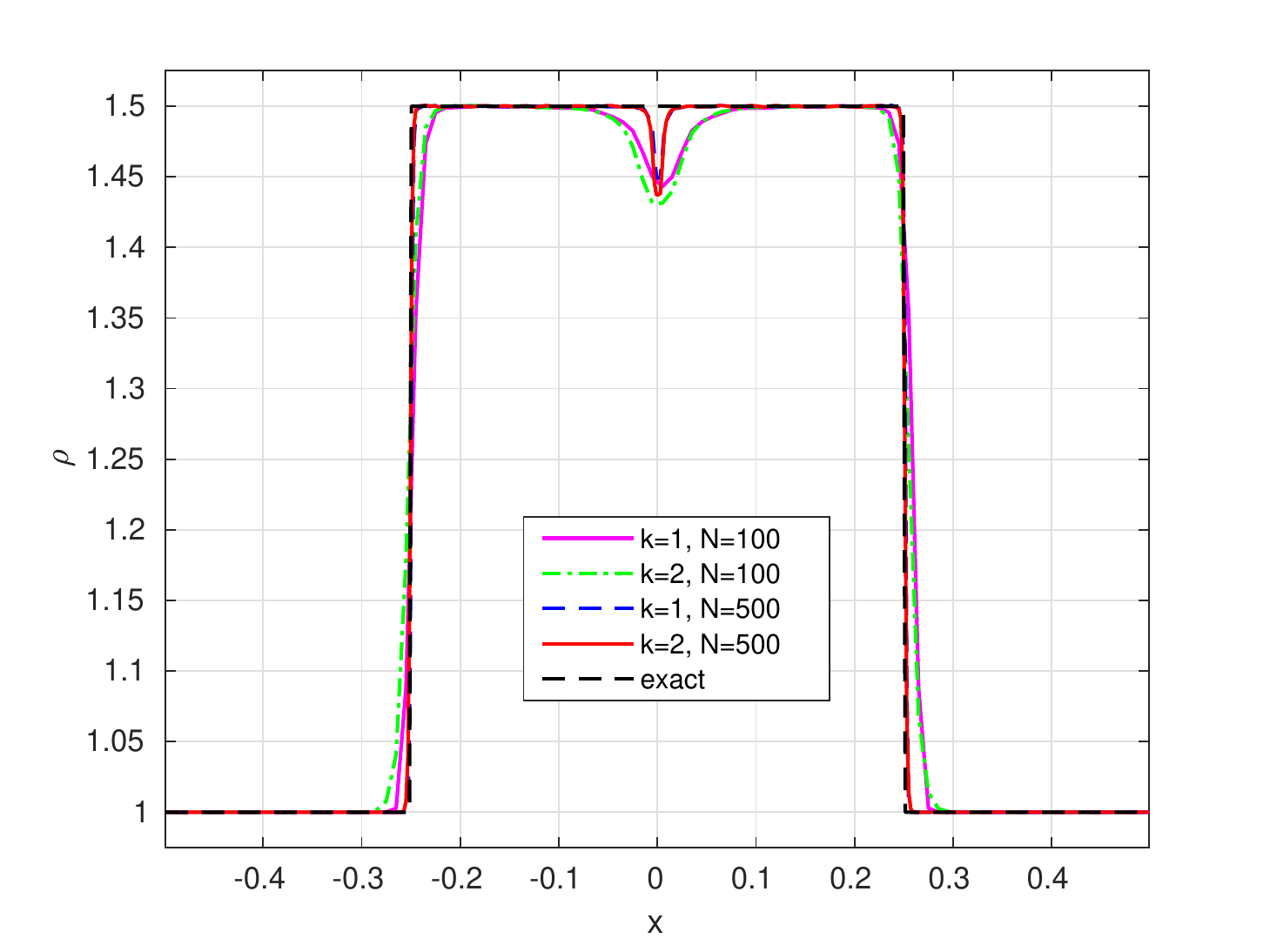}
		\caption{$\rho$}
	\end{subfigure}	
	\begin{subfigure}[b]{0.45\textwidth}
		\includegraphics[width=\textwidth]{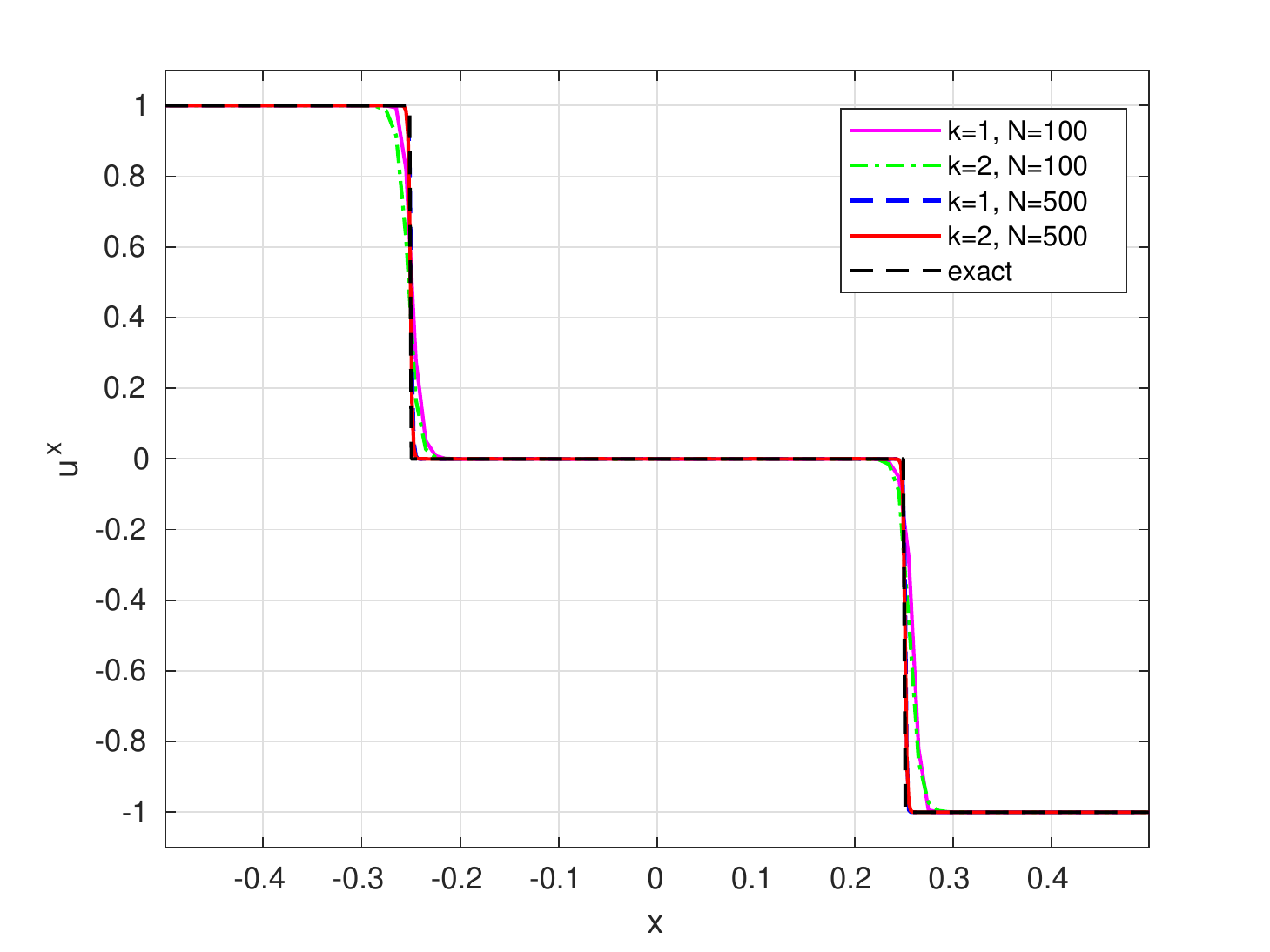}
		\caption{$v^x$}
	\end{subfigure}	
	\begin{subfigure}[b]{0.45\textwidth}
		\includegraphics[width=\textwidth]{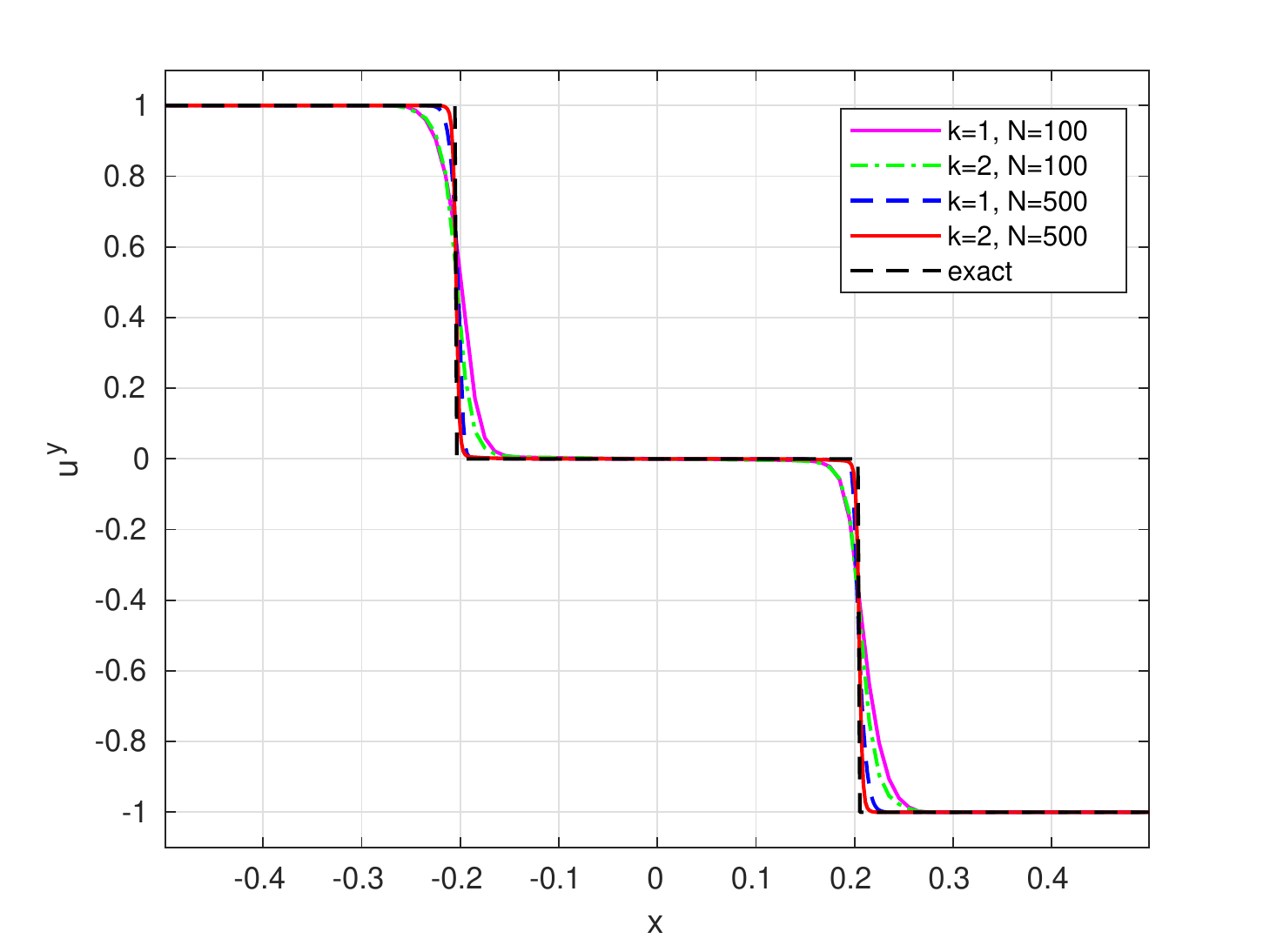}
		\caption{$v^y$}
	\end{subfigure}	
	\begin{subfigure}[b]{0.45\textwidth}
		\includegraphics[width=\textwidth]{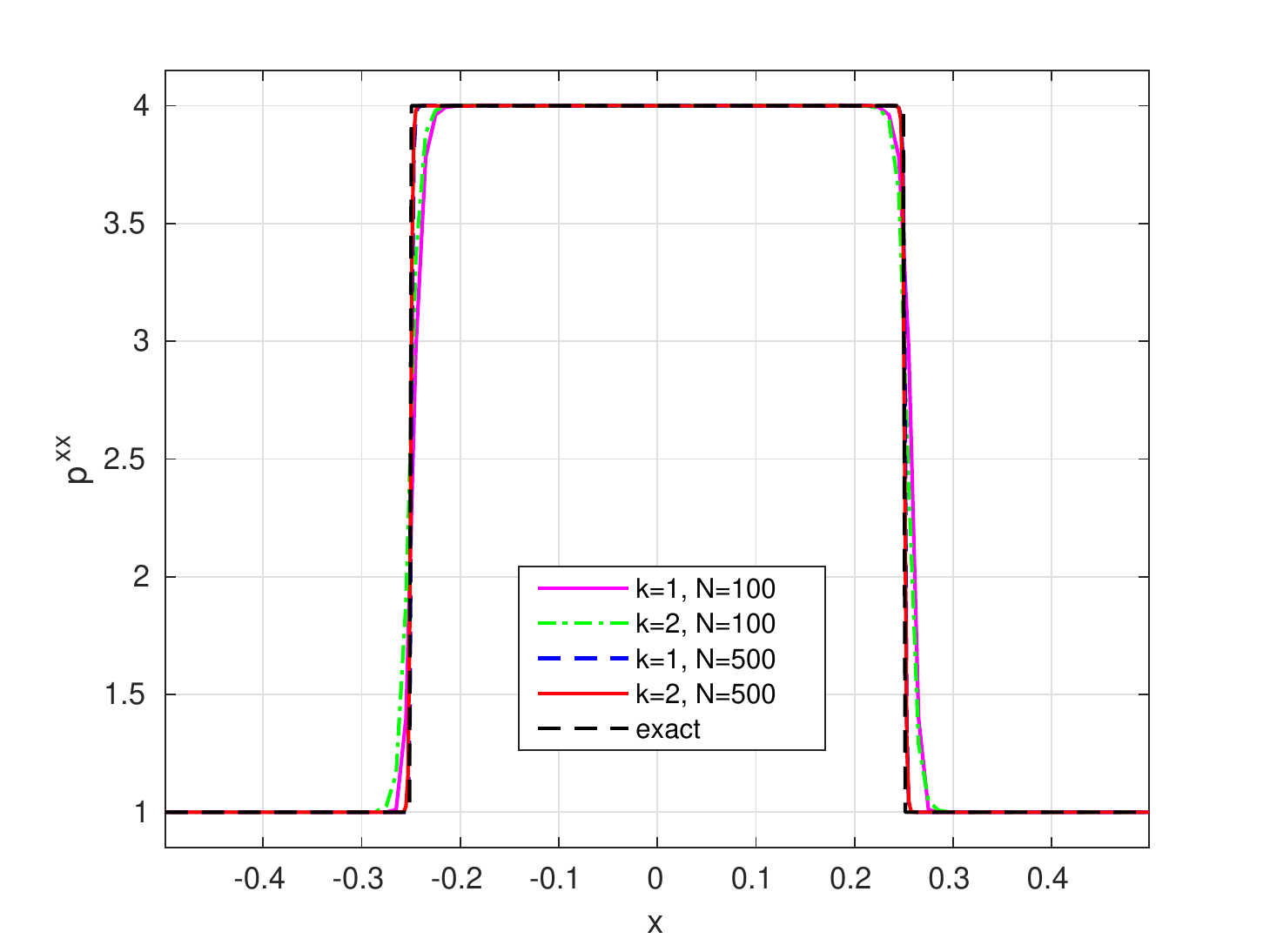}
		\caption{$p^{xx}$}
	\end{subfigure}	
	\begin{subfigure}[b]{0.45\textwidth}
		\includegraphics[width=\textwidth]{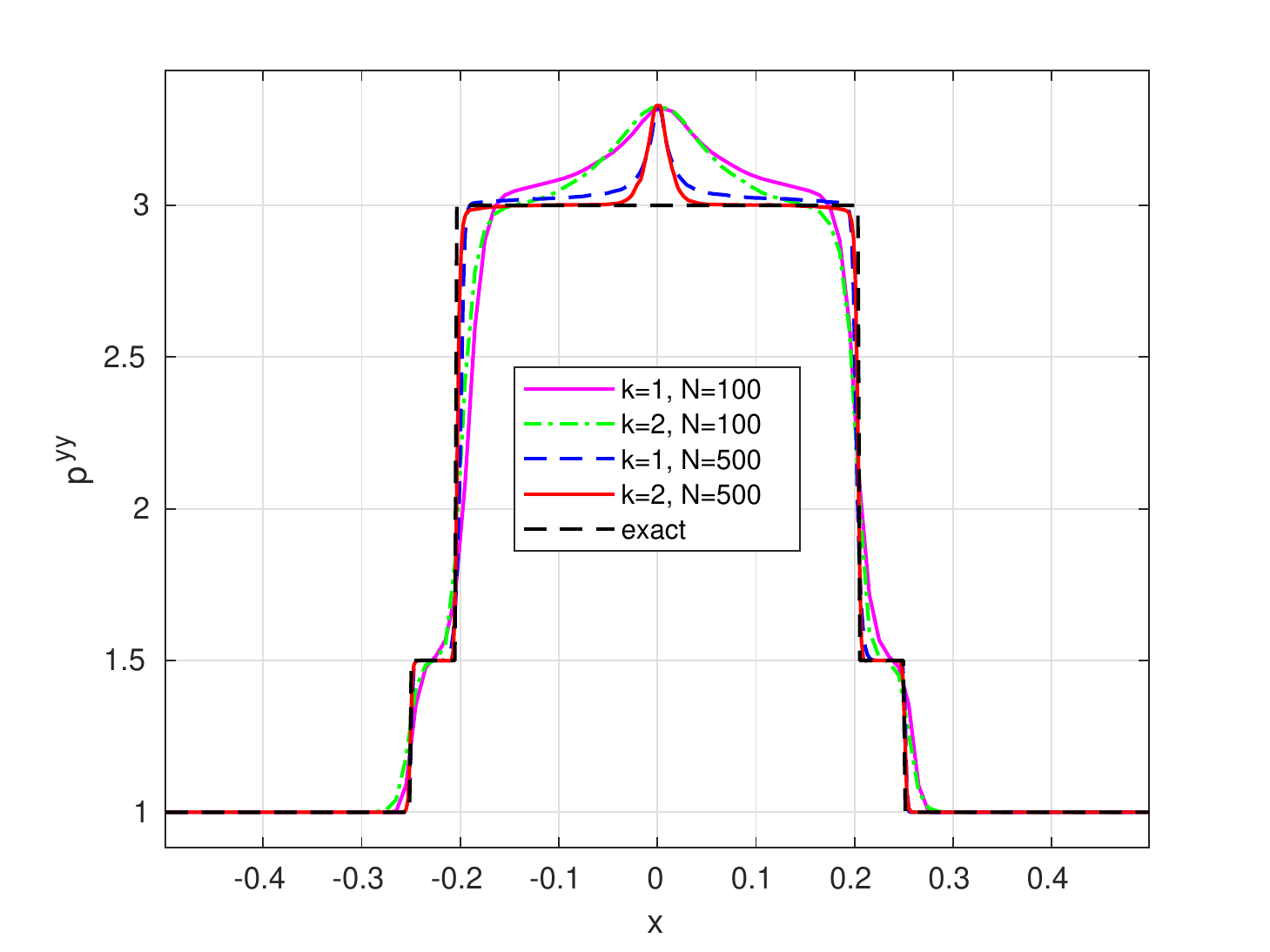}
		\caption{$p^{yy}$}
			\end{subfigure}	
	\begin{subfigure}[b]{0.45\textwidth}
			\includegraphics[width=\textwidth]{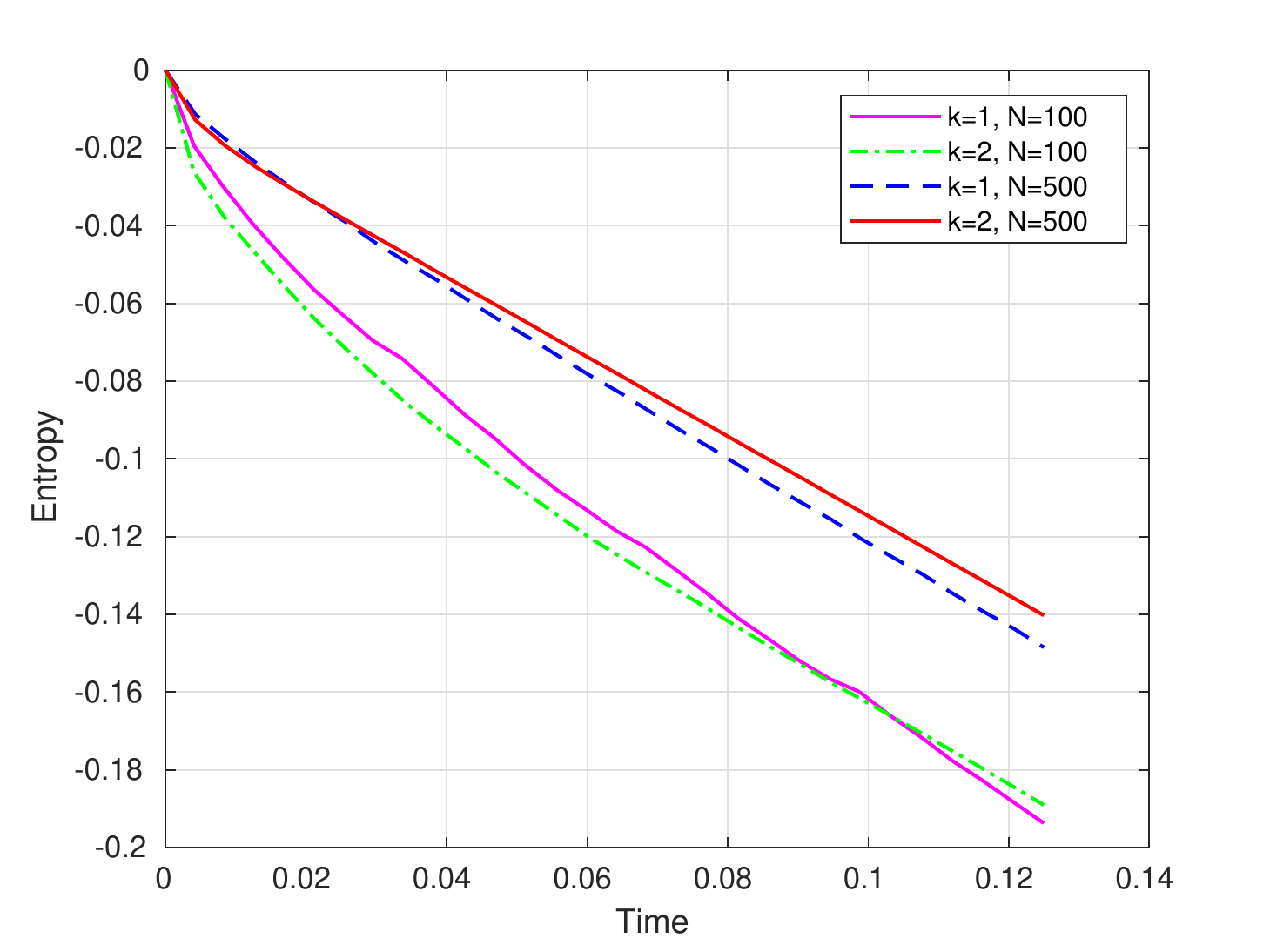}
			\caption{Evolution of total entropy}
	\end{subfigure}	
	\caption{Test Problem 4 (Two shock waves): Plot of density, velocity, pressure components and total entropy evolution for ESDG-O2(k=1) and ESDG-O3(k=2) using 100 and 500 cells.}
	\label{fig:pb1}
\end{figure}
\begin{figure}[htb!]
	\centering
	\begin{subfigure}[b]{0.45\textwidth}
		\includegraphics[width=\textwidth]{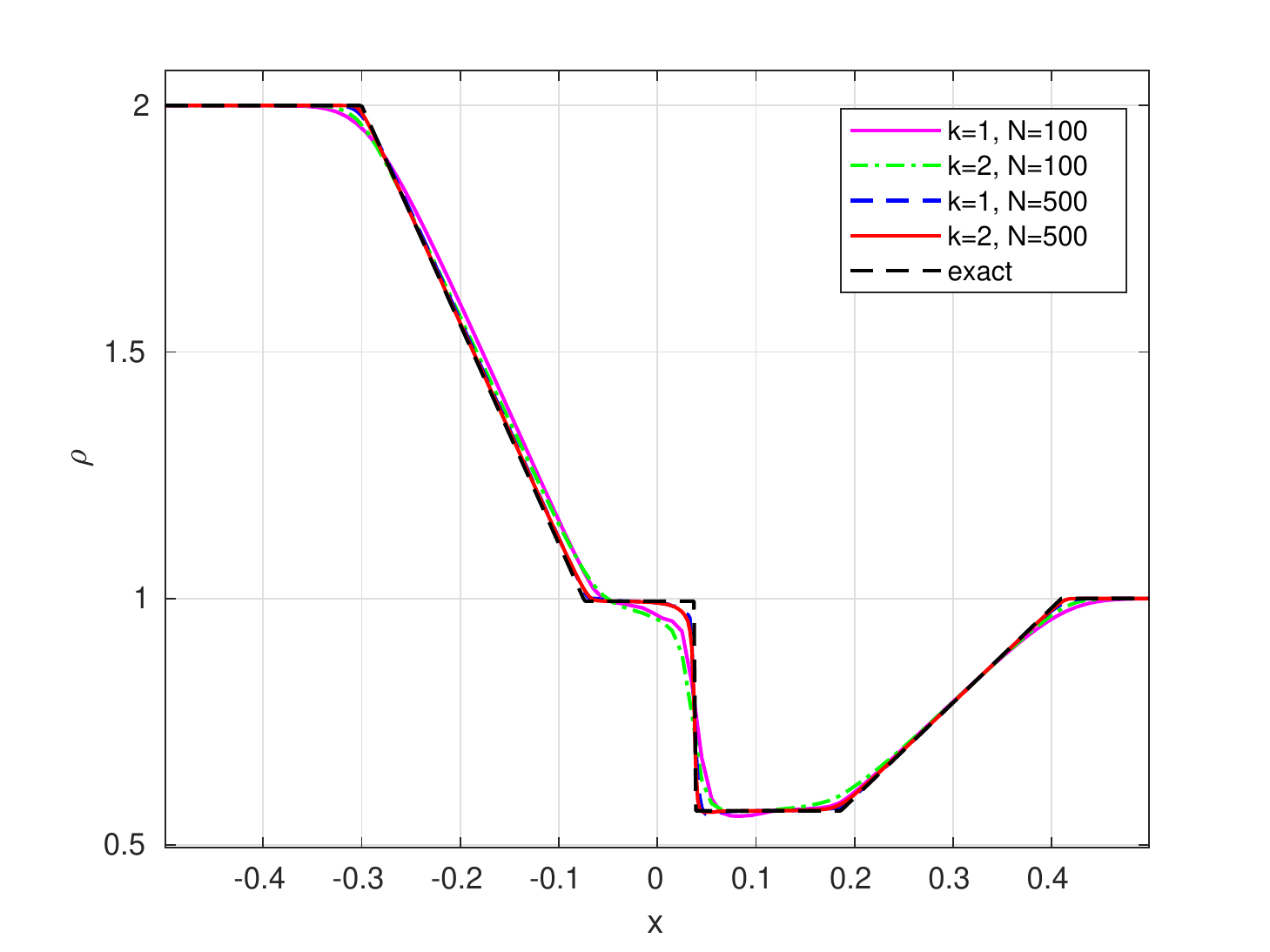}
		\caption{$\rho$}
	\end{subfigure}	
	\begin{subfigure}[b]{0.45\textwidth}
		\includegraphics[width=\textwidth]{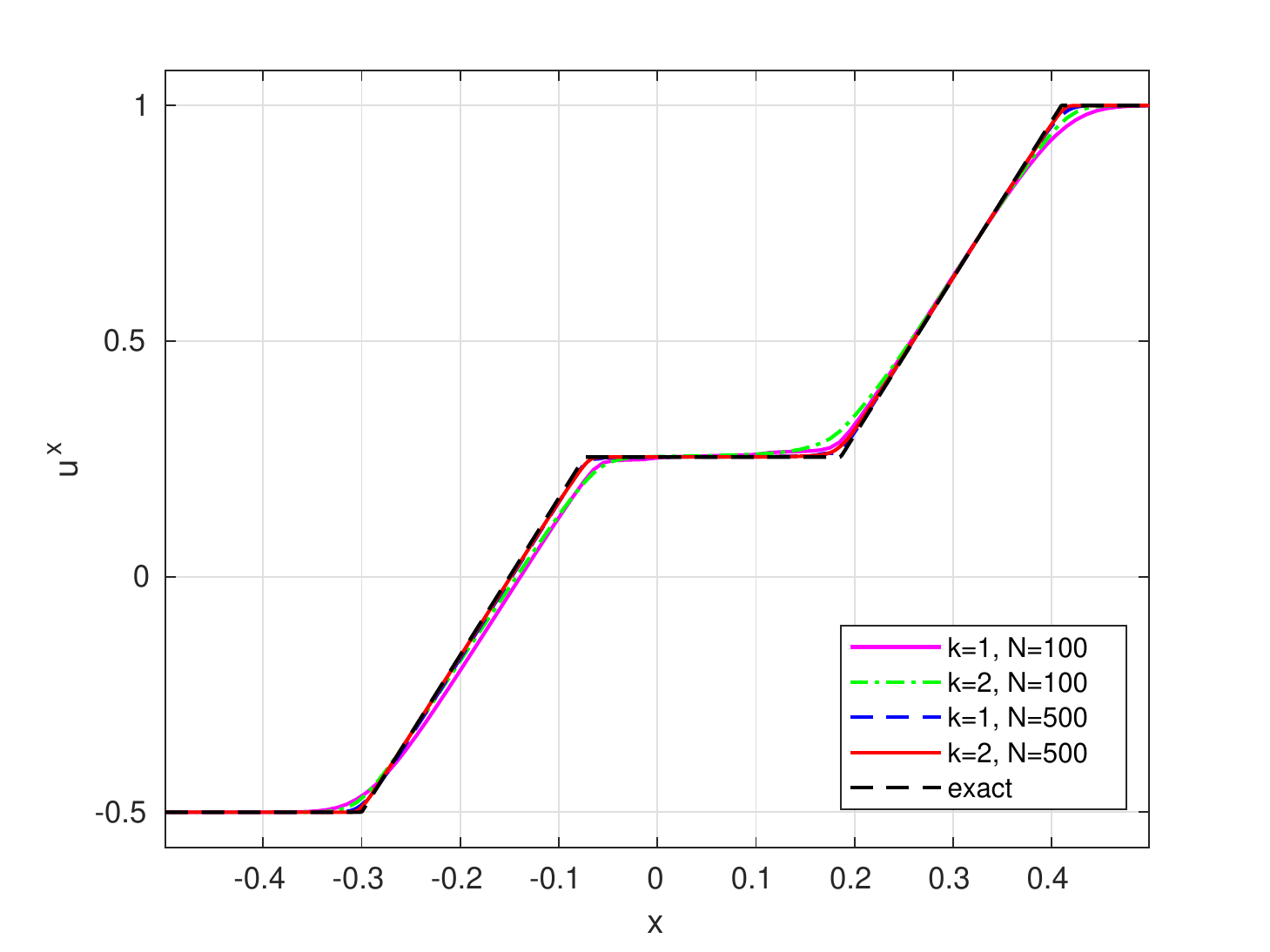}
		\caption{$v^x$}
	\end{subfigure}	
	\begin{subfigure}[b]{0.45\textwidth}
		\includegraphics[width=\textwidth]{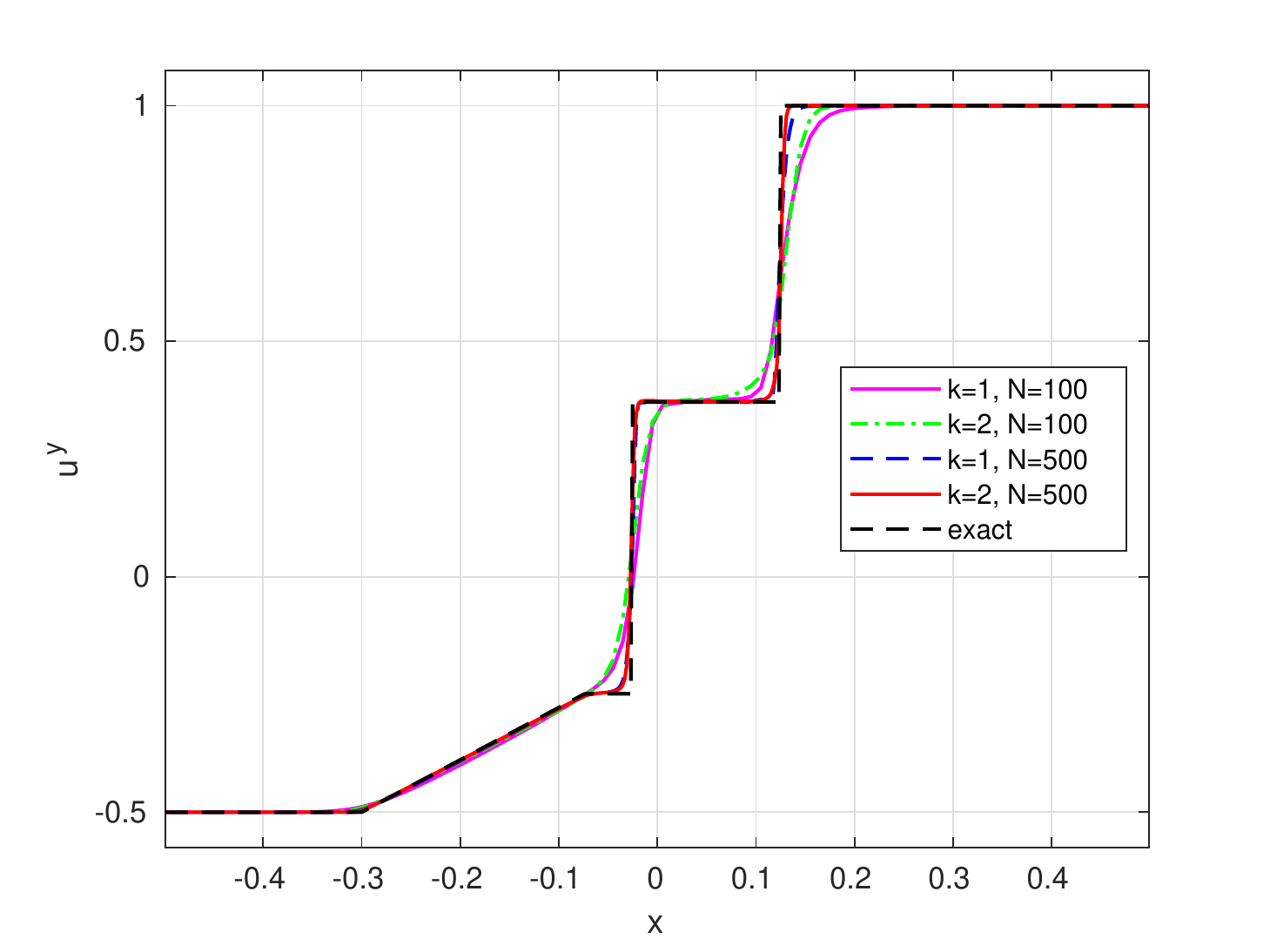}
		\caption{$v^y$}
	\end{subfigure}	
	\begin{subfigure}[b]{0.45\textwidth}
		\includegraphics[width=\textwidth]{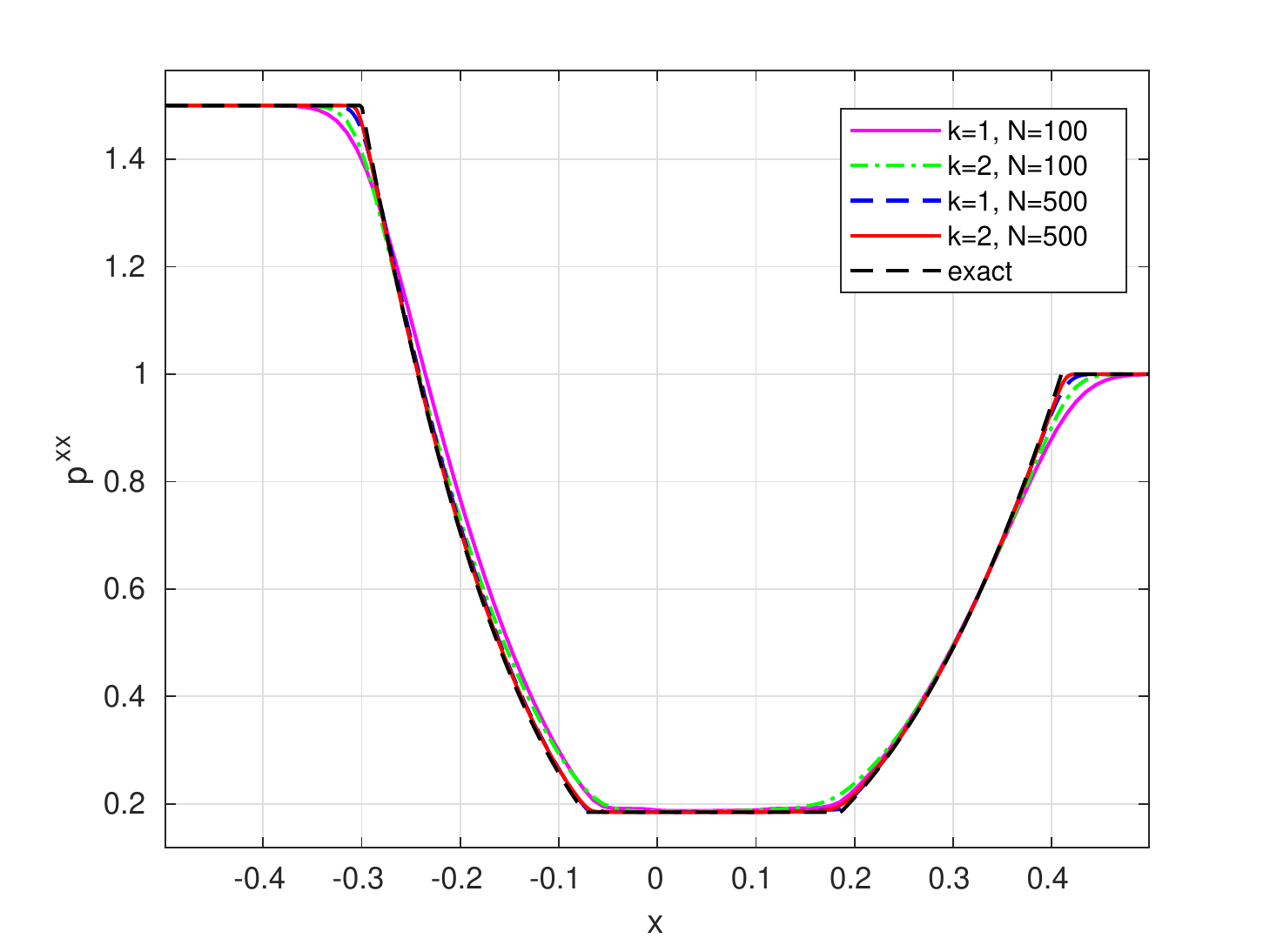}
		\caption{$p^{xx}$}
	\end{subfigure}	
	\begin{subfigure}[b]{0.45\textwidth}
		\includegraphics[width=\textwidth]{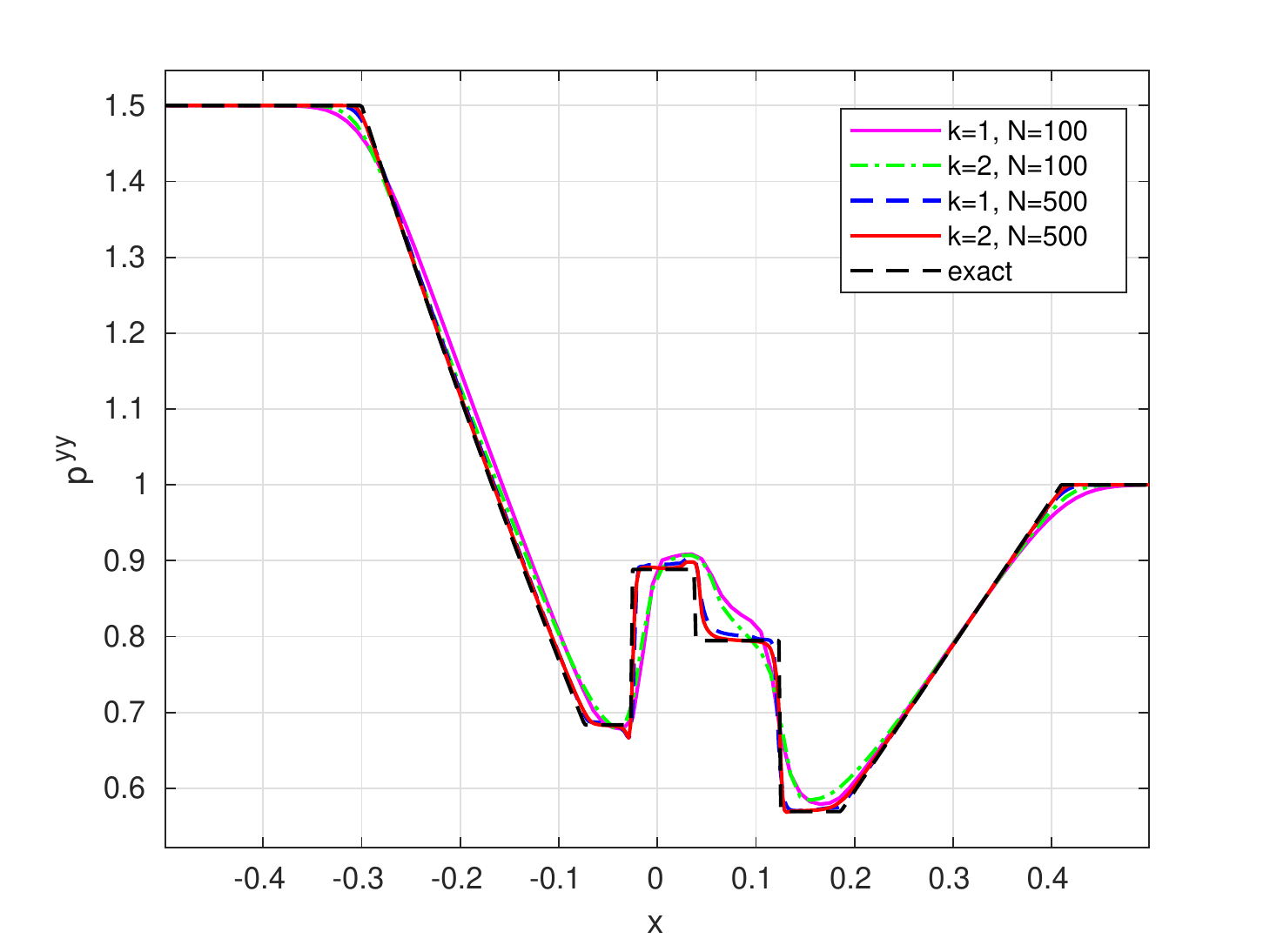}
		\caption{$p^{yy}$}
	\end{subfigure}	
	\begin{subfigure}[b]{0.45\textwidth}
	\includegraphics[width=\textwidth]{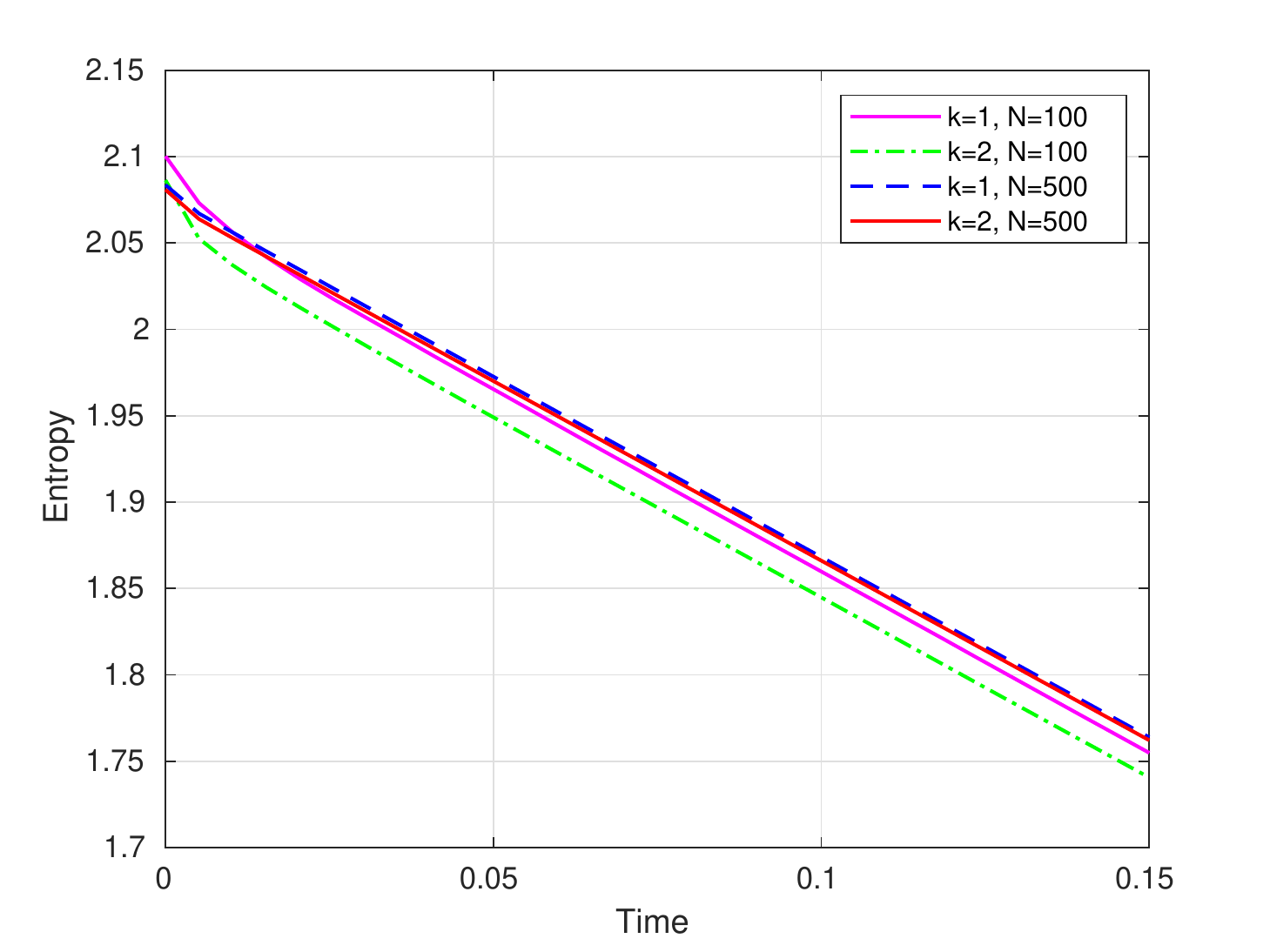}
	\caption{Evolution of total entropy}
\end{subfigure}	
	\caption{Test Problem 5 (Two rarefaction waves): Plot of density, velocity, pressure components and total entropy evolution for ESDG-O2(k=1) and ESDG-O3(k=2) using 100 and 500 cells.}
	\label{fig:pb3}
\end{figure}
\begin{example}[Sod shock tube Riemann problem:] We consider the Riemann problem on the domain n $[-0.5,\,0.5]$ with initial discontinuity at $x=0$. The left and right states are given in Table \ref{tab:sod}. We use outflow boundary conditions, and the solutions are presented at time of $T=0.125$. The exact solution of the problem contains a shock, contact wave, and rarefaction wave. We do not need bound preserving limiter for this test case. TVBM limiter was used to control the oscillations.
	\begin{table}[htb!]
		\centering
		\begin{tabular}{l|cccccc}
			\hline
			States & $\rho$ & $v^x$ & $v^y$ & $p^{xx}$ & $p^{xy}$ & $p^{yy}$ \\ 
			\hline
			Left & 1 & 0 & 0 & 2 & 0.05 & 0.6 \\
			Right &0.125& 0 &0&0.2&0.1&0.2\\
			\hline
		\end{tabular}
		\caption{Test Problem 3: Initial conditions for Sod shock tube problem}
		\label{tab:sod}
	\end{table}
	The numerical solutions are presented in Figure    \ref{fig:pb2}, using 100 and 500 cells for ESDG-O2(k=1) and ESDG-O3(k=2) schemes. We have plotted density, velocity and the pressure components. We observe that both schemes are able to resolve all the waves and as expected ESDG-O3 is more accurate than the ESDG-O2 scheme. Furthermore, the use of  finer grids $N=500$ significantly improves the results. We also note that the results are comparable with \cite{meena2017positivity,sen2018entropy}.
	
	We have also plotted the time evolution of total entropy. We note that at $500$, entropy decay is lower than at the $100$ cells. Furthermore, both ESDG-O2 and ESDG-O3 having similar decay, with ESDG-O2 having slightly lower decay at 100 cells when compared to ESDG-O3. However, at 500 cells, ESDG-O3 having lower entropy decay than the ESDG-O2.
\end{example}

\begin{example}[Two shock waves]
	This is another Riemann problem (see\cite{meena2017positivity,meena_positivity-preserving_2020,berthon_numerical_2006}) where the exact solution contains two shock waves moving away from each other. We consider the domain $[-0.5,\,0.5]$ with outflow boundary conditions. The initial discontinuity is centered at $x=0$ separating the left and right states given in Table \ref{tab:twoshock}. Bound preserving limiter is not required for this test. Computational results are plotted at final time $T=0.125$. 
	\begin{table}[htb!]
		\centering
		\begin{tabular}{l|cccccc}
			\hline
			States & $\rho$ & $v^x$ & $v^y$ & $p^{xx}$ & $p^{xy}$ & $p^{yy}$ \\ 
			\hline
			Left & 1 & 1  & 1 & 1 & 0 & 1 \\
			Right& 1 & -1 &-1 & 1 & 0 & 1\\
			\hline
		\end{tabular}
		\caption{Test Problem 4: Initial conditions for the two shock waves problem}
		\label{tab:twoshock}
	\end{table}
	Numerical results are plotted in Figure \ref{fig:pb1} for ESDG-O2(k=1) and ESDG-O3(k=2) at resolutions of 100 and 500 cells. We have again plotted density, velocity and pressure components. We note that both the schemes are able to capture the shocks and results improve significantly when we use finer mesh of 500 cells. In addition, ESDG-O3 is more accurate than the ESDG-O2. Also, results are comparable to those presented in \cite{meena2017positivity,meena_positivity-preserving_2020}. 
	
	From the total entropy decay plot, we observe that entropy decay decreases significantly when the resolution is increased from 100 to 500 cells. Furthermore, entropy decay of the ESDG-O3 scheme is lower than the ESDG-O2 scheme.
\end{example}

\begin{example}[Two rarefaction waves]
	In this test case, we consider a Riemann problem, where the exact solution contains two rarefaction waves. The test case is set in the domain $[-0.5,0.5]$ with the initial jump at $x=0$ is separating two states given in Table \ref{tab:tworare}. Similar to the last two cases, we use outflow boundary conditions.  The solutions are computed until time $T=0.15$ using 100 and 500 cells.
	\begin{table}[htb!]
		\centering
		\begin{tabular}{l|cccccc}
			\hline
			States & $\rho$ & $v^x$ & $v^y$ & $p^{xx}$ & $p^{xy}$ & $p^{yy}$ \\ 
			\hline
			Left & 2 & -0.5  & -0.5 & 1.5 & 0.5 & 1.5 \\
			Right& 1 & 1 &1 & 1 & 0 & 1\\
			\hline
		\end{tabular}
		\caption{Test Problem 5: Initial conditions for two rarefaction waves problem}
		\label{tab:tworare}
	\end{table}

	The numerical solutions are presented in Figure \ref{fig:pb3} for both schemes. We observe the similar performance of the schemes as in the last two test cases, with ESDG-O3 more accurate than the ESDG-O2 scheme. Furthermore, both of the schemes are able to capture all the waves. In addition, numerical results are consistent with results in \cite{berthon_numerical_2006,meena2017positivity,meena_positivity-preserving_2020}. We also note from the total entropy decay plot that the schemes have similar entropy decay.
\end{example}    

\begin{example}[Near vacuum state] This test case is designed to test robustness of the schemes at low density and pressure (see \cite{meena2017positivity,meena_positivity-preserving_2020}). The domain  $[-0.5,\,0.5]$ consists of left and right states, given in Table \ref{tab:tworare}, separated at $x=0$.  We have simulated the solution till time $T=0.05$  using outflow boundary conditions. Even though we have low density and pressure areas, we have not used bound preserving limiter.
	\begin{table}[htb!]
		\centering
		\begin{tabular}{l|cccccc}
			\hline
			States & $\rho$ & $v^x$ & $v^y$ & $p^{xx}$ & $p^{xy}$ & $p^{yy}$ \\ 
			\hline
			Left & 1 & -5  & 0 & 2 & 0 & 2 \\
			Right& 1 &  5  & 0 & 2 & 0 & 2\\
			\hline
		\end{tabular}
		\caption{Test Problem 6: Initial conditions for the near vacuum state problem}
		\label{tab:nearvacuum}
	\end{table}
	\begin{figure}[htb!]
		\centering
		\begin{subfigure}[b]{0.45\textwidth}
			\includegraphics[width=\textwidth]{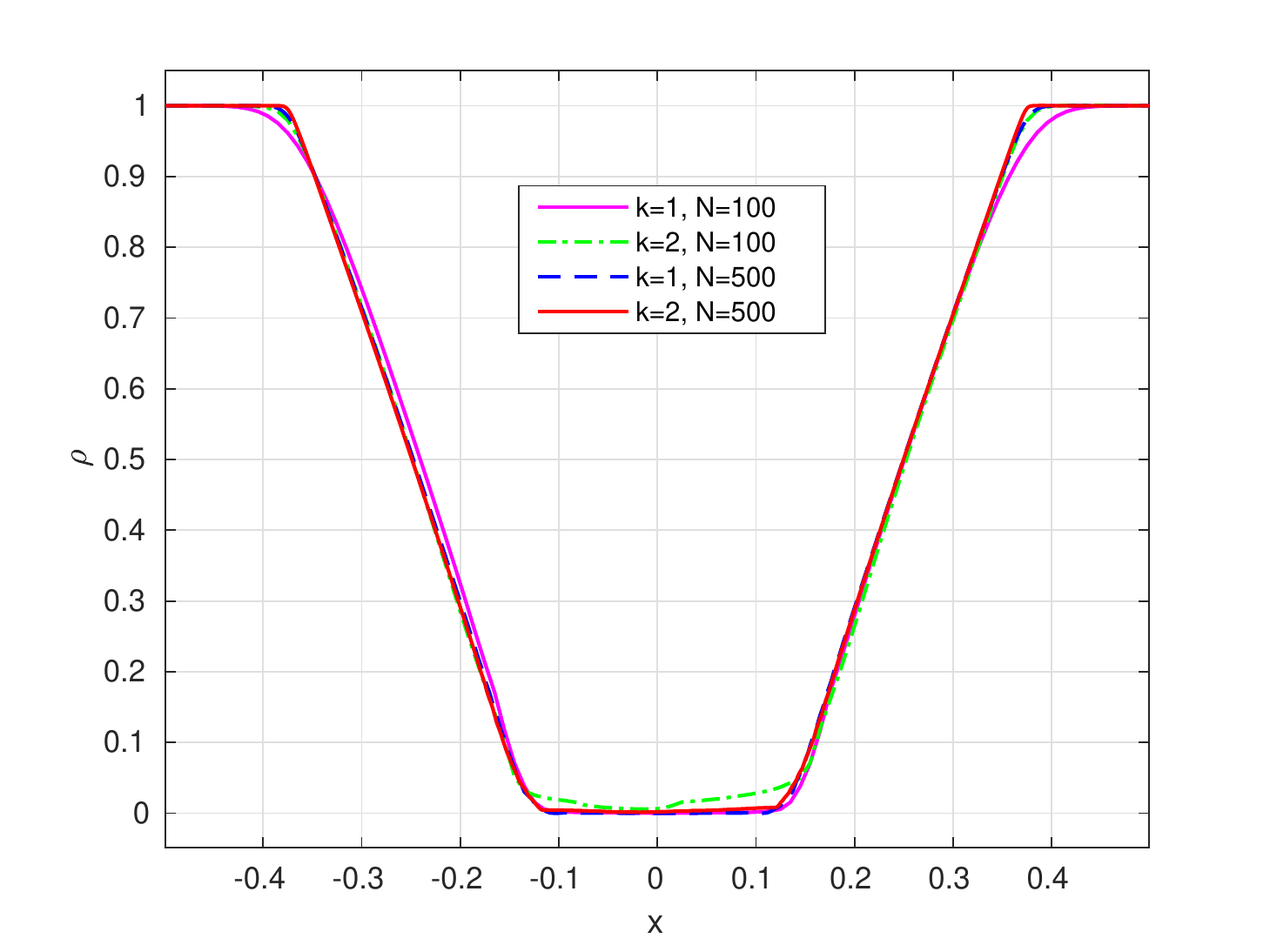}
			\caption{$\rho$}
		\end{subfigure}    
		\begin{subfigure}[b]{0.45\textwidth}
			\includegraphics[width=\textwidth]{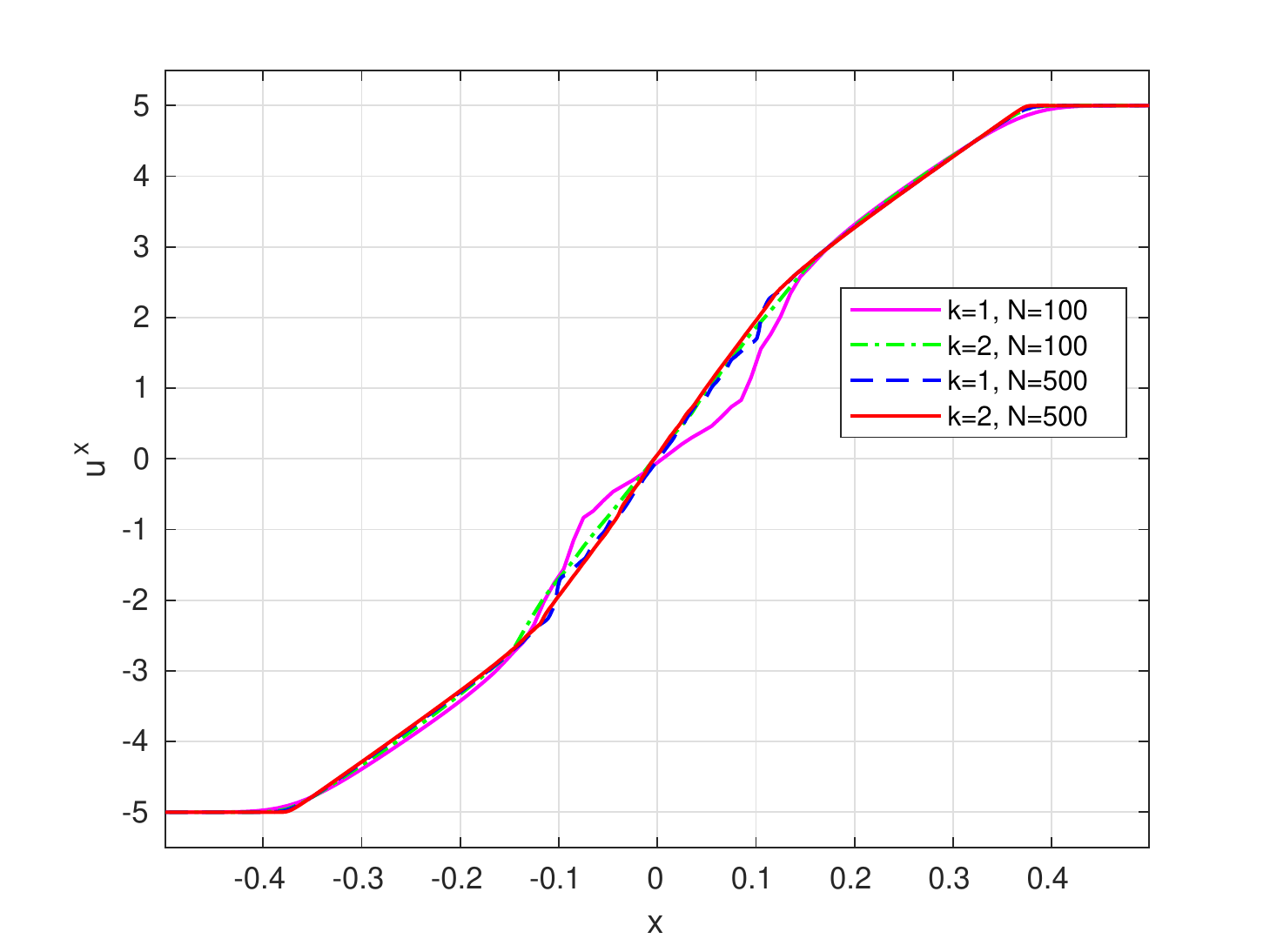}
			\caption{$v^x$}
		\end{subfigure}    
		\begin{subfigure}[b]{0.45\textwidth}
			\includegraphics[width=\textwidth]{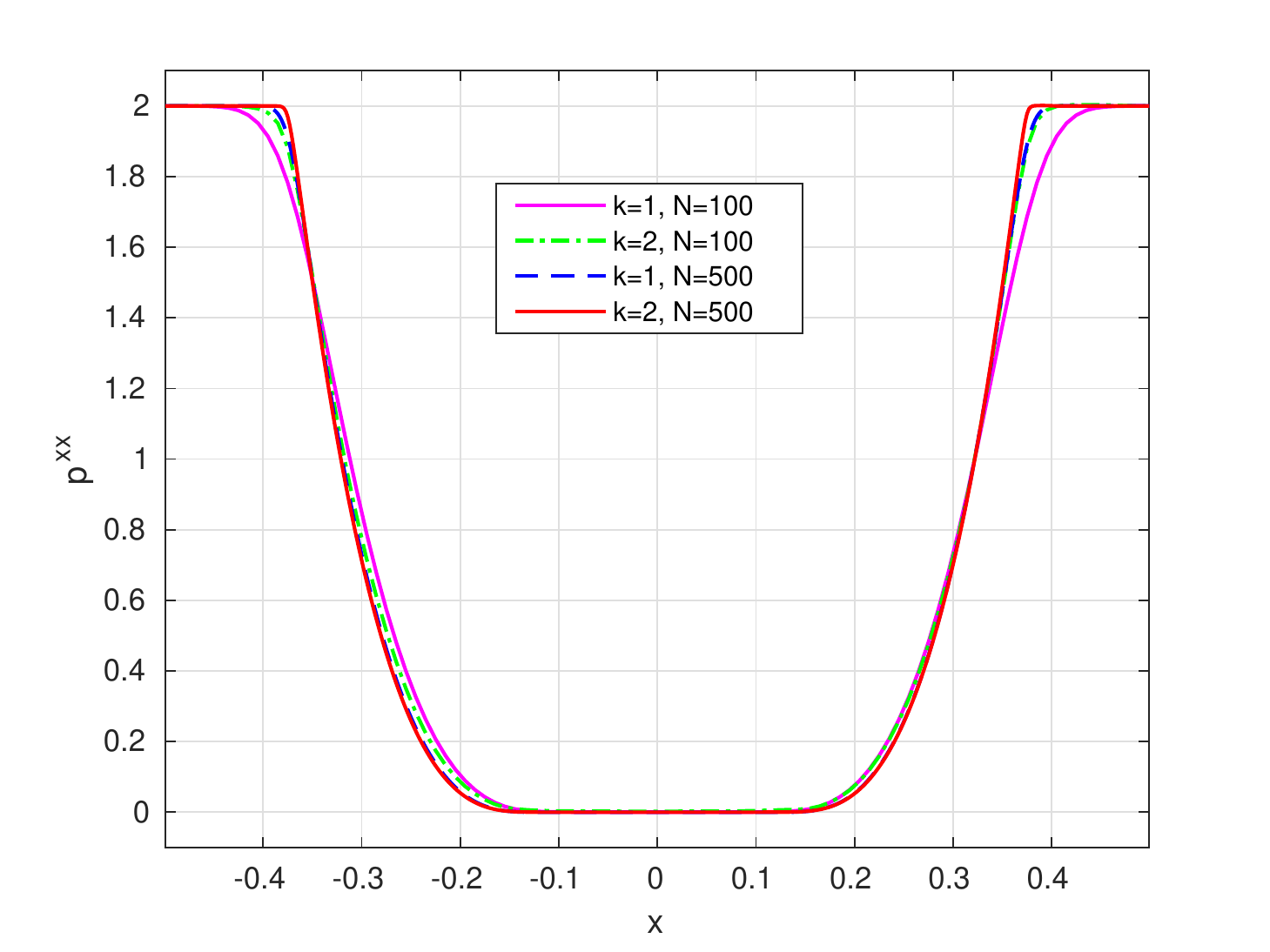}
			\caption{$p^{xx}$}
		\end{subfigure}    
		\begin{subfigure}[b]{0.45\textwidth}
			\includegraphics[width=\textwidth]{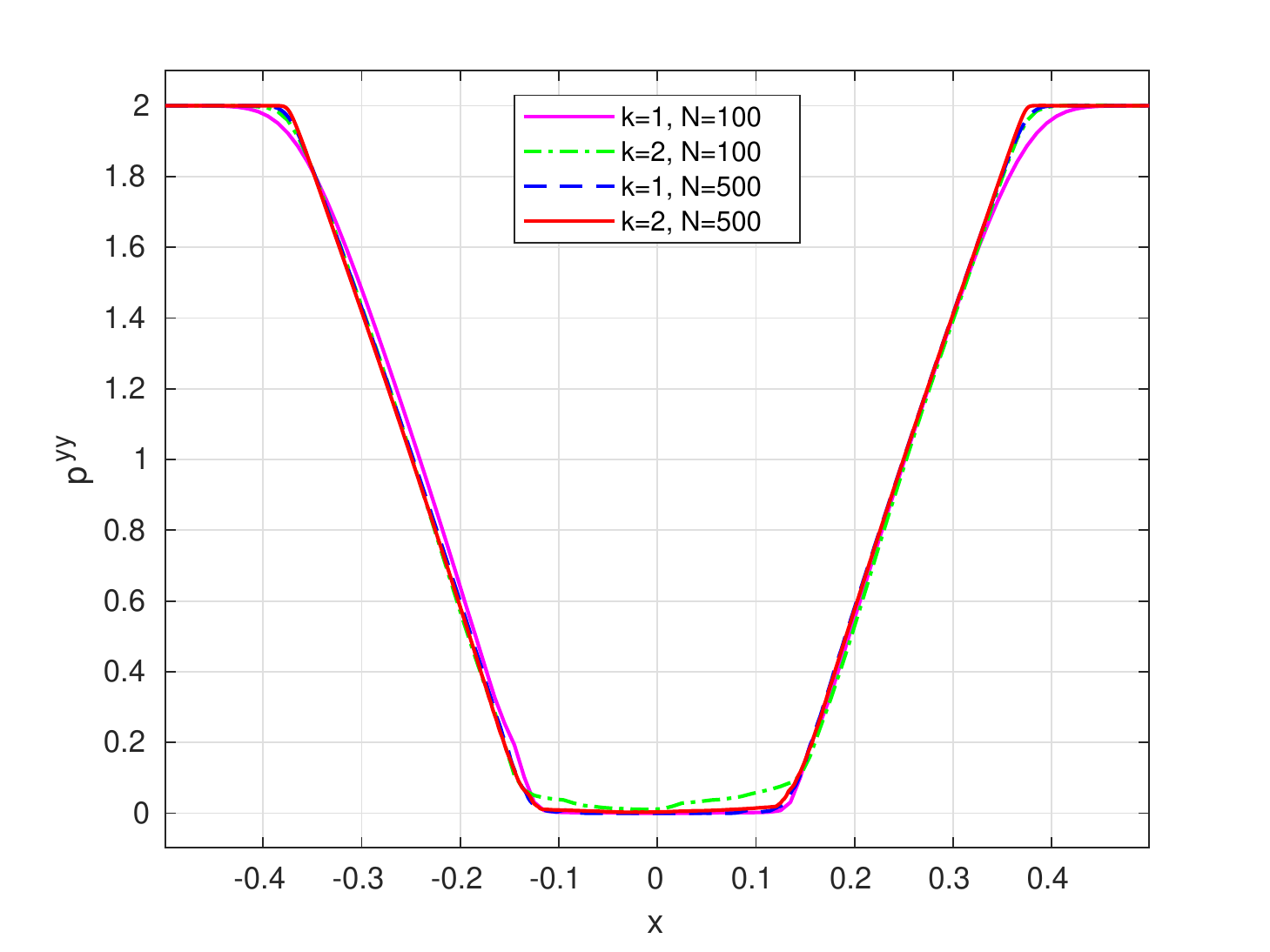}
			\caption{$p^{yy}$}
		\end{subfigure}    
		\caption{Test Problem 6 (Near vacuum state): Plot of density, velocity and pressure components for ESDG-O2(k=1) and ESDG-O3(k=2) using 100 and 500 cells.}
		\label{fig:pb4}
	\end{figure}
	
	Numerical results are plotted in Figure \ref{fig:pb4} for both schemes. We observe that both the schemes able to capture two outgoing rarefaction waves, and both are stable. Also, at the finer resolution of $500$ cells, the results are highly accurate and comparable to those presented in \cite{meena2017positivity,meena_positivity-preserving_2020}.
\end{example}

\begin{figure}[htb!]
	\centering
	\begin{subfigure}[b]{0.45\textwidth}
		\includegraphics[width=\textwidth]{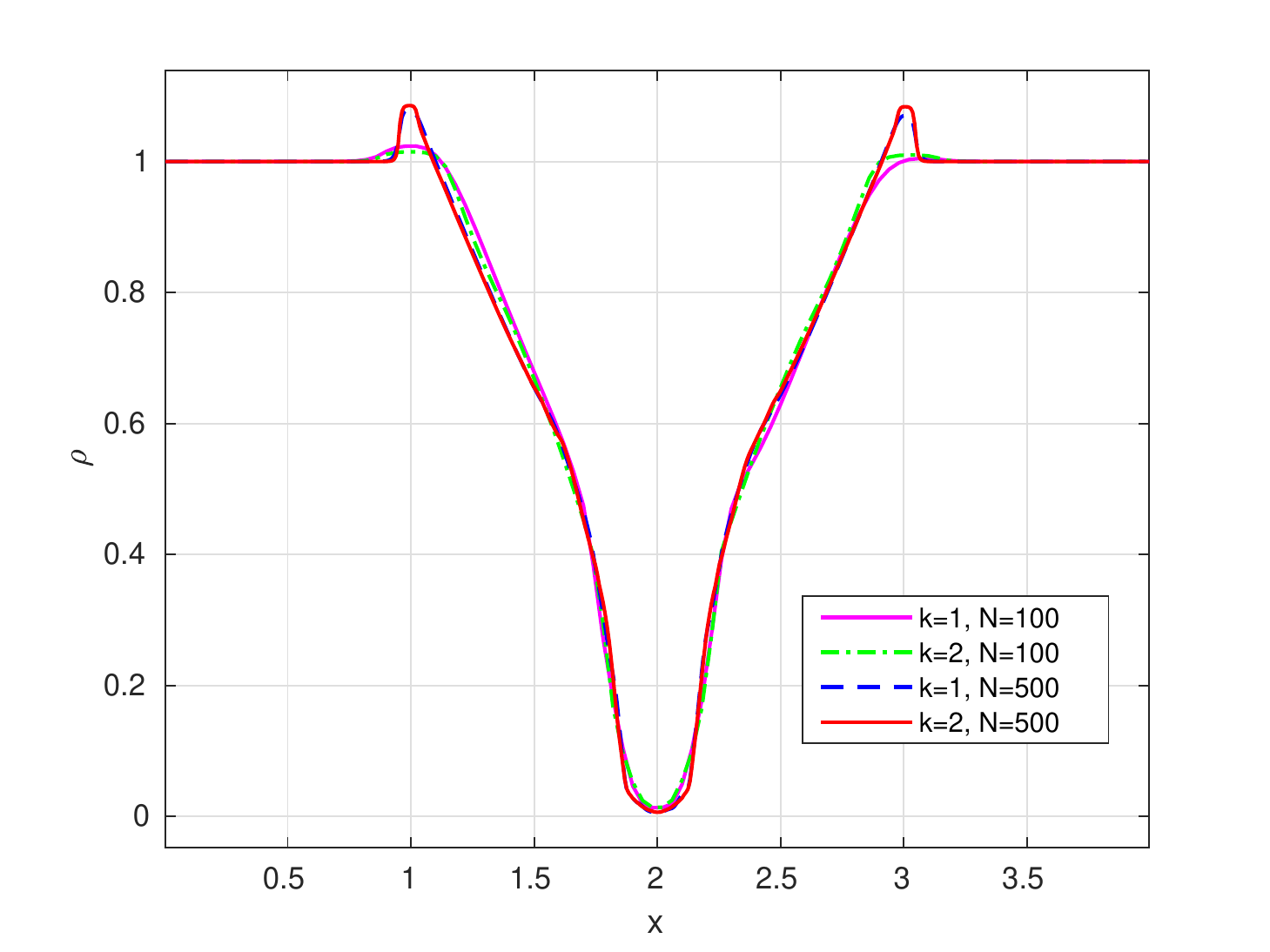}
		\caption{$\rho$}
	\end{subfigure}	
	\begin{subfigure}[b]{0.45\textwidth}
		\includegraphics[width=\textwidth]{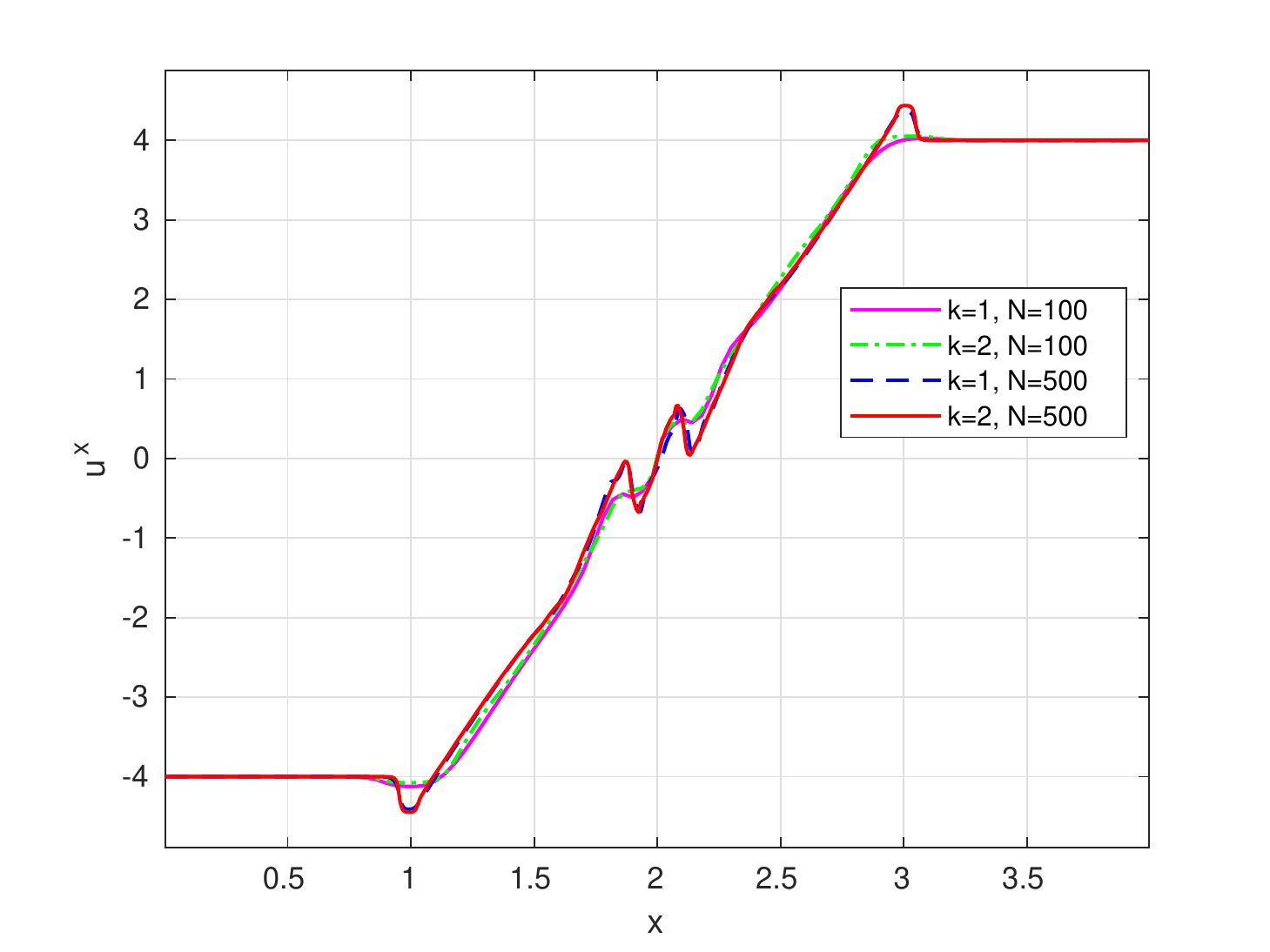}
		\caption{$v^x$}
	\end{subfigure}	
	\begin{subfigure}[b]{0.45\textwidth}
		\includegraphics[width=\textwidth]{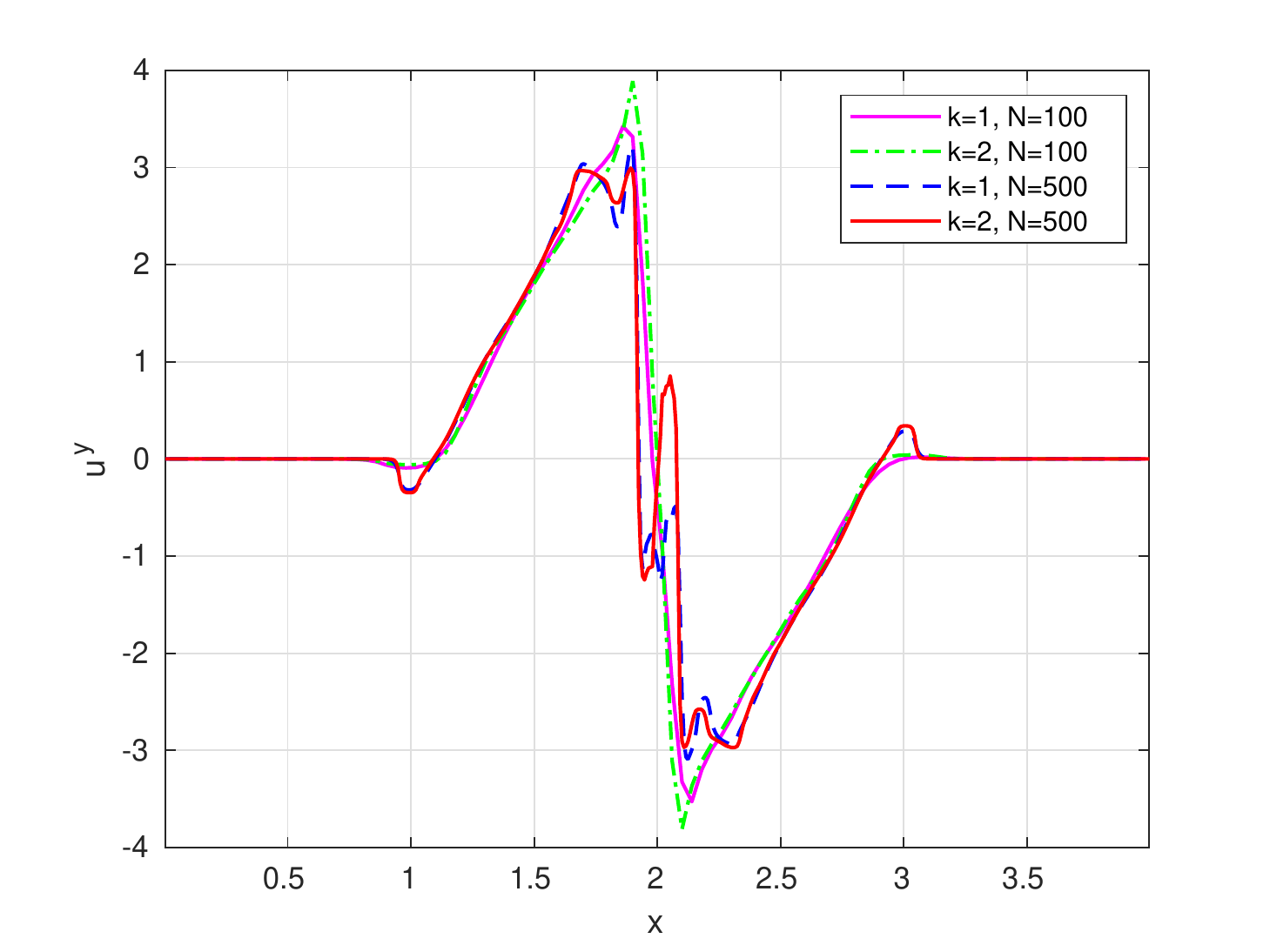}
		\caption{$v^y$}
	\end{subfigure}	
	\begin{subfigure}[b]{0.45\textwidth}
		\includegraphics[width=\textwidth]{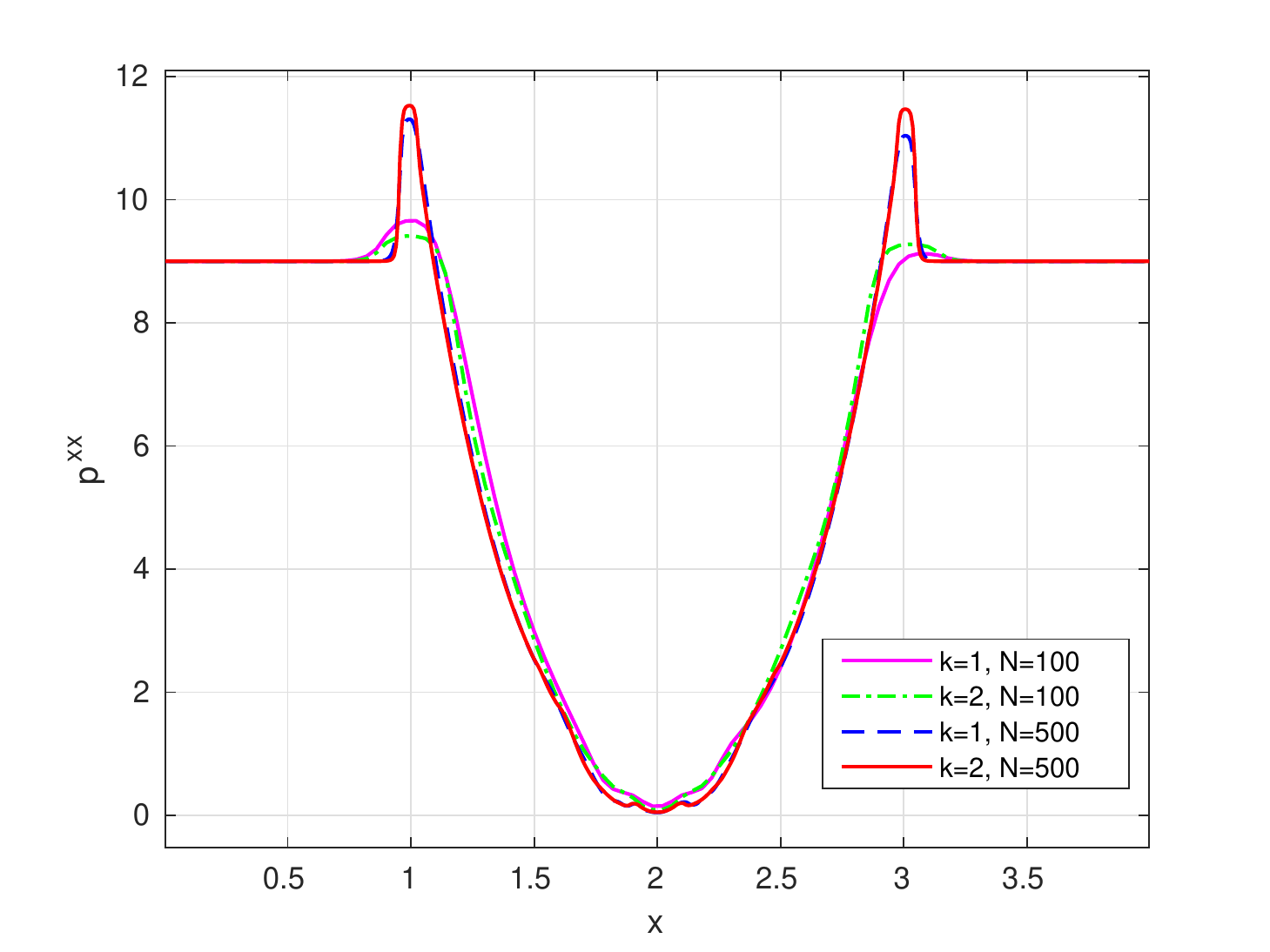}
		\caption{$p^{xx}$}
	\end{subfigure}	
	\begin{subfigure}[b]{0.45\textwidth}
		\includegraphics[width=\textwidth]{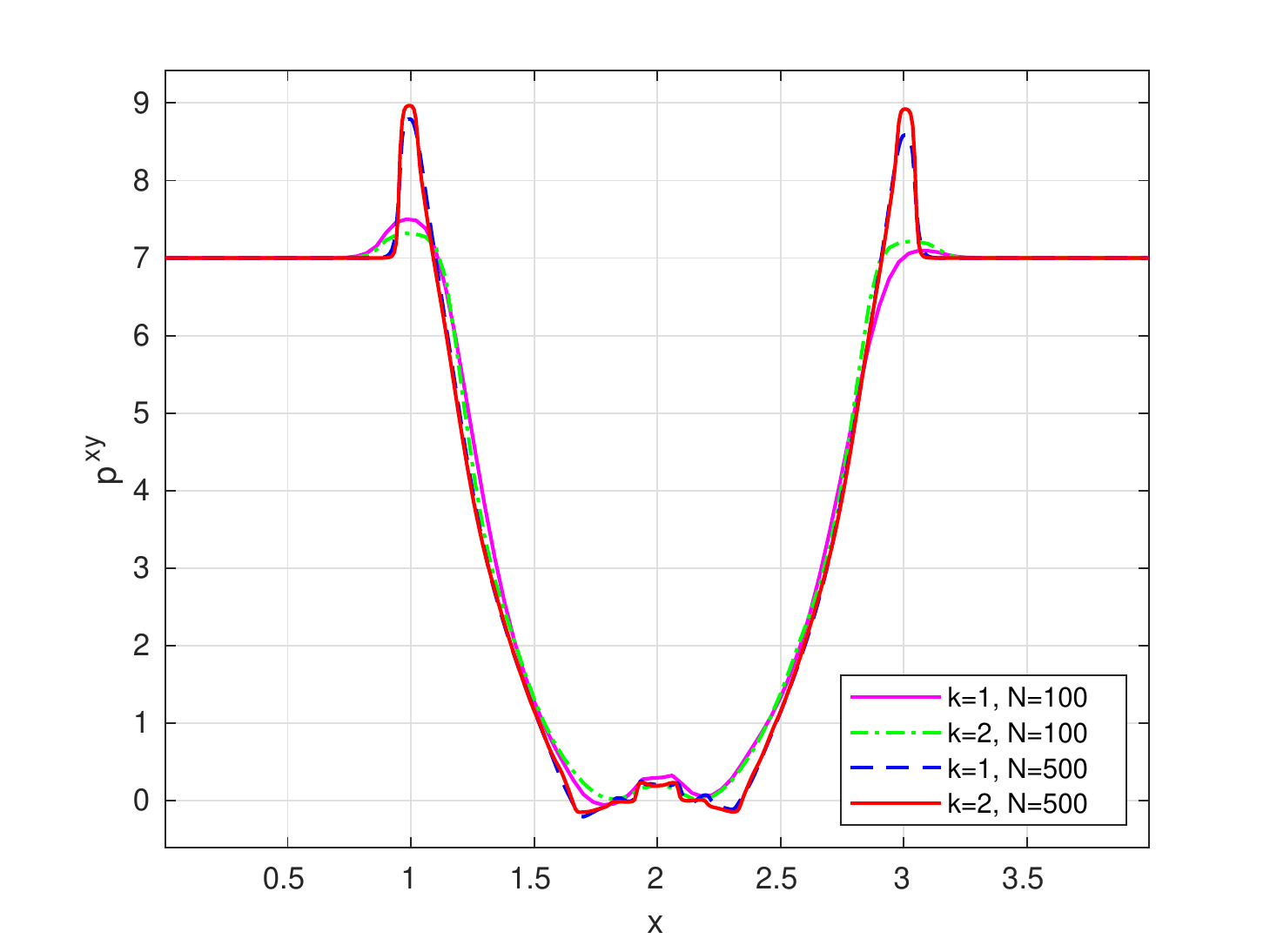}
		\caption{$p^{xy}$}
	\end{subfigure}	
	\begin{subfigure}[b]{0.45\textwidth}
		\includegraphics[width=\textwidth]{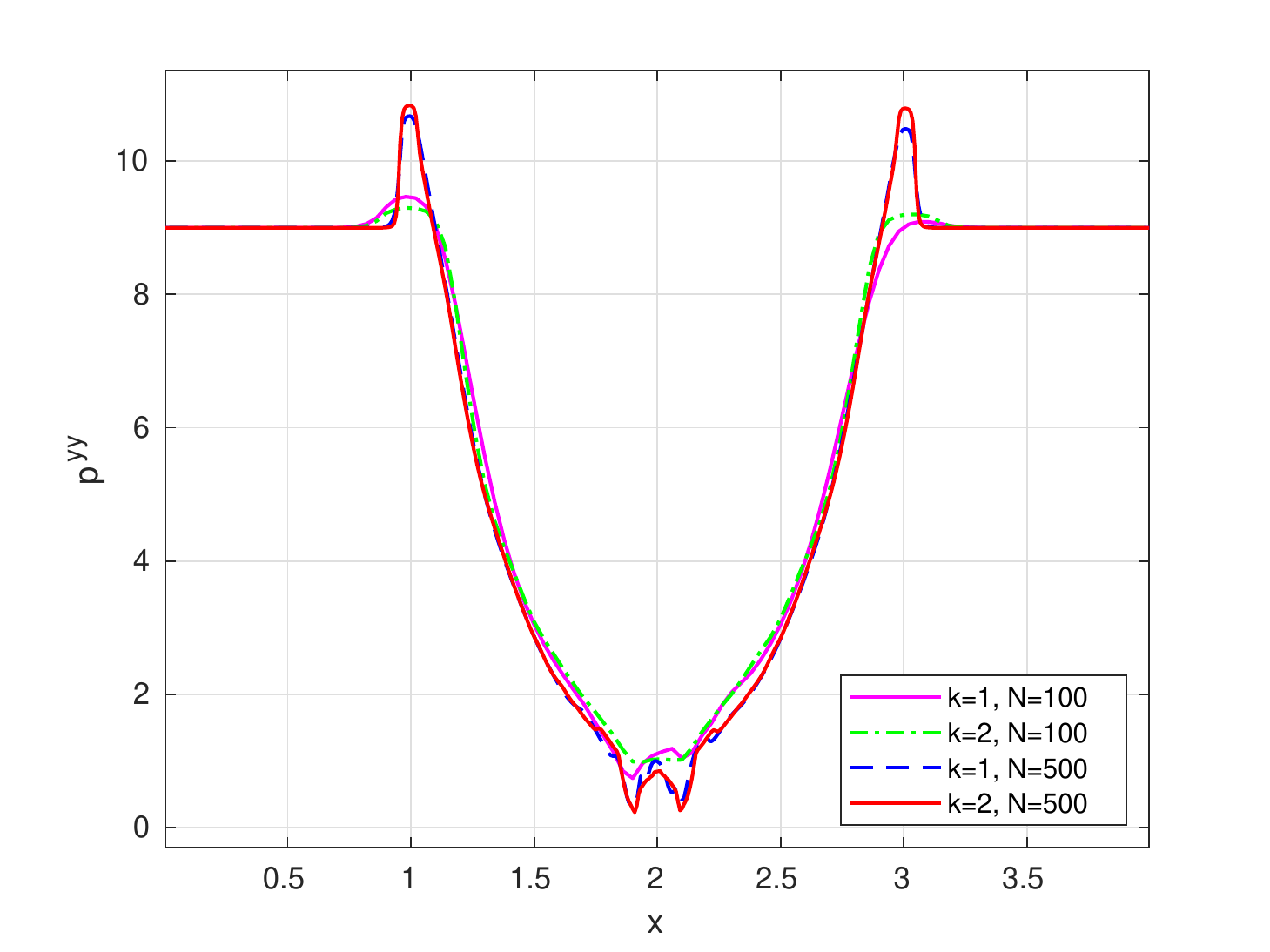}
		\caption{$p^{yy}$}
	\end{subfigure}	
	\caption{Test Problem 7 (Two rarefaction waves problem with Gaussian source): Plot of density, velocity and pressure components  for ESDG-O2 and ESDG-O3 using 100 and 500 cells.}
	\label{fig:pb5}
\end{figure}
\begin{figure}[htb!]
	\centering
	\begin{subfigure}[b]{0.45\textwidth}
		\includegraphics[width=\textwidth]{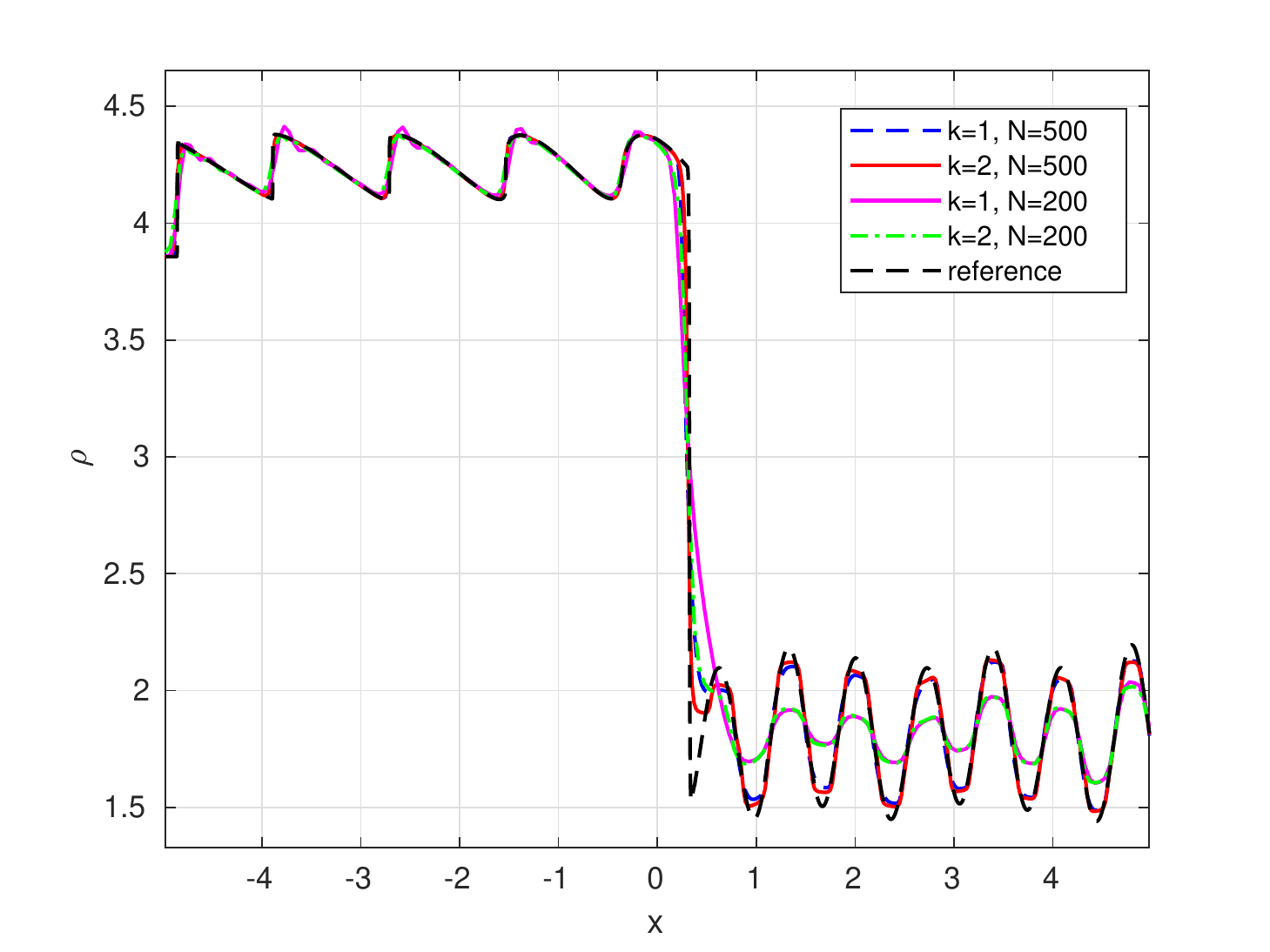}
		\caption{$\rho$}
	\end{subfigure}	
	\begin{subfigure}[b]{0.45\textwidth}
		\includegraphics[width=\textwidth]{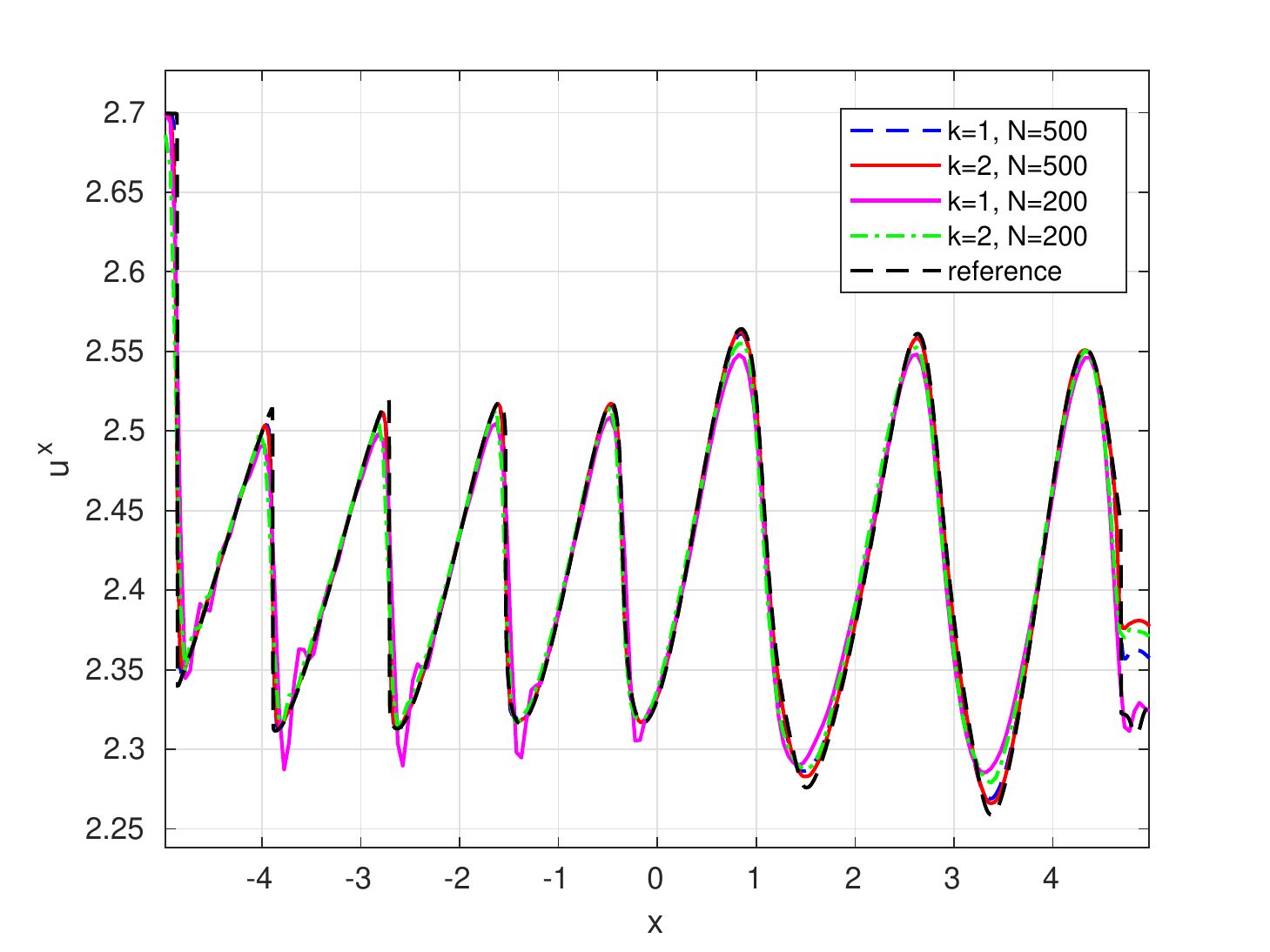}
		\caption{$v^x$}
	\end{subfigure}	
	\begin{subfigure}[b]{0.45\textwidth}
		\includegraphics[width=\textwidth]{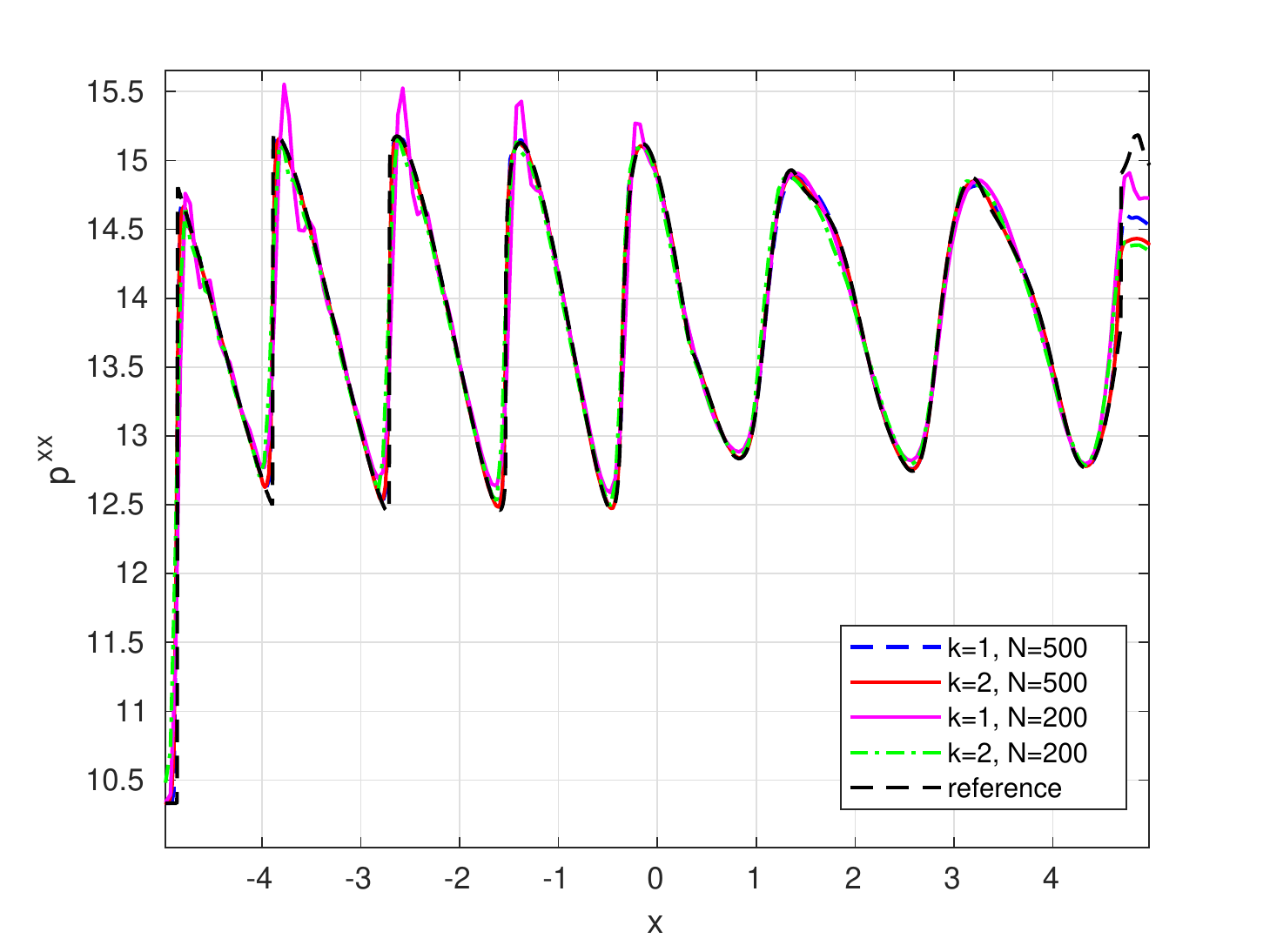}
		\caption{$p^{xx}$}
	\end{subfigure}	
	\begin{subfigure}[b]{0.45\textwidth}
		\includegraphics[width=\textwidth]{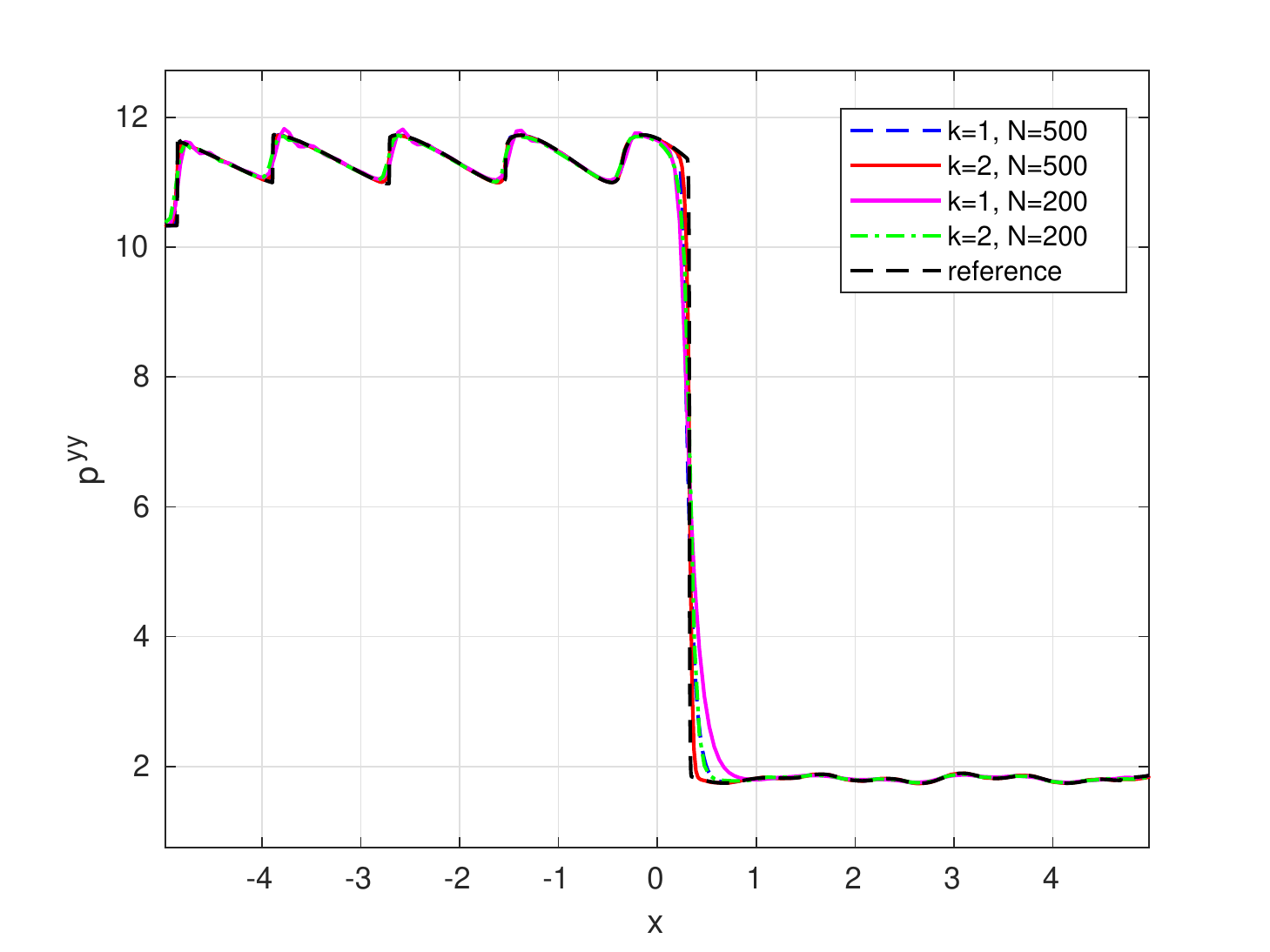}
		\caption{$p^{yy}$}
	\end{subfigure}	
	\caption{Test Problem 8 (Shu-Osher test problem): Plot of density, velocity and pressure components for ESDG-O2(k=1) and ESDG-O3(k=2) using 100 and 500 cells.}
	\label{fig:pb9}
\end{figure}

\begin{example}[Two rarefaction waves problem with Gaussian source]
	This is another test case, where we test the robustness of the scheme in low density and pressure areas. However, in this case the low density and pressure areas are generated via source terms. We consider the  computational domain of $[0,\,4]$ with outflow boundary conditions and the initial discontinuity is placed at $x=2$. The left and right states are given in Table \ref{tab:1dsource}. To test the effects of source terms, following \cite{meena2017positivity,meena_positivity-preserving_2020}, we consider the Gaussian source term corresponding to 
	\begin{equation*}
	W(x,\,t)=25 \exp(-200(x-2)^2).
	\end{equation*} 
	\begin{table}[htb!]
		\centering
		\begin{tabular}{l|cccccc}
			\hline
			States & $\rho$ & $v^x$ & $v^y$ & $p^{xx}$ & $p^{xy}$ & $p^{yy}$ \\ 
			\hline
			Left & 1 & -4  & 0 & 9 & 7 & 9 \\
			Right& 1 &  4  & 0 & 9 & 7 & 9\\
			\hline
		\end{tabular}
		\caption{Test Problem 7: Initial conditions for two rarefaction waves with Gaussian source term problem}
		\label{tab:1dsource}
	\end{table}
	
	Computational results for both schemes are plotted in Figure \ref{fig:pb5} at $T=0.1$. In this test, we use bound preserving limiter for both the schemes. We again observe that both of these schemes are able the capture all the features of the solution, and both are stable. Furthermore, the solution is highly accurate at the finer mesh of $500$ cells.
\end{example}

\begin{example}[Shu-Osher test problem]
	This test problem is obtained by modifying the Shu-Osher test case problem for Euler equations (\cite{shu1988,meena_positivity-preserving_2020}). We consider the computational domain of $[-5,5]$ with initial discontinuity at $x=-4$, which separates states given in Table \ref{tab:shuosher}. The computations are performed using $200$ and $500$ cells till time $T=1.8$.  Results show that both ESDG-O2 and ESDG-O3 schemes are capable of resolving the small-scale features of the flow. We note that the results using the fine grid ($N=500$) are much more accurate than the results on the coarse grid ($N=200$).
	\begin{table}
		\centering
		\begin{tabular}{l|cccccc}
			\hline
			States & $\rho$ & $v^x$ & $v^y$ & $p^{xx}$ & $p^{xy}$ & $p^{yy}$ \\ 
			\hline 
			Left & 3.857143 & 2.699369  & 0 & 10.33333 & 0 & 10.33333 \\
			Right& $1+0.2 \sin(5x)$ &  0  & 0 & 1 & 0 & 1\\
			\hline
		\end{tabular}
		\caption{Test Problem 8: Initial conditions for Shu-Osher test problem}
		\label{tab:shuosher}
	\end{table}
\end{example}
	
\subsection{Two dimensional numerical tests}
We will now present two dimensional test cases. 
\begin{example} [Two dimensional near vacuum test problem]
	In this test case form \cite{meena2017positivity,meena_positivity-preserving_2020}, we have a low density and low-pressure zone, hence we will test the robustness of the algorithms. We consider the computational domain $[-2,2]\times[-2,2]$ with outflow boundary conditions. Density is taken to be unity throughout the domain, whereas the pressure components are, $p^{xx}=p^{yy}=2, p^{xy}=0$. The velocity is set to be $8(x/r,\,y/r)$, where $r=\sqrt{x^2+y^2}$. Simulations are performed using $100\times100$ mesh until $T=0.05$. Bound preserving limiter is used during the simulations. 
	\begin{figure}[htb!]
		\centering
		\begin{subfigure}[b]{0.45\textwidth}
			\includegraphics[width=\textwidth]{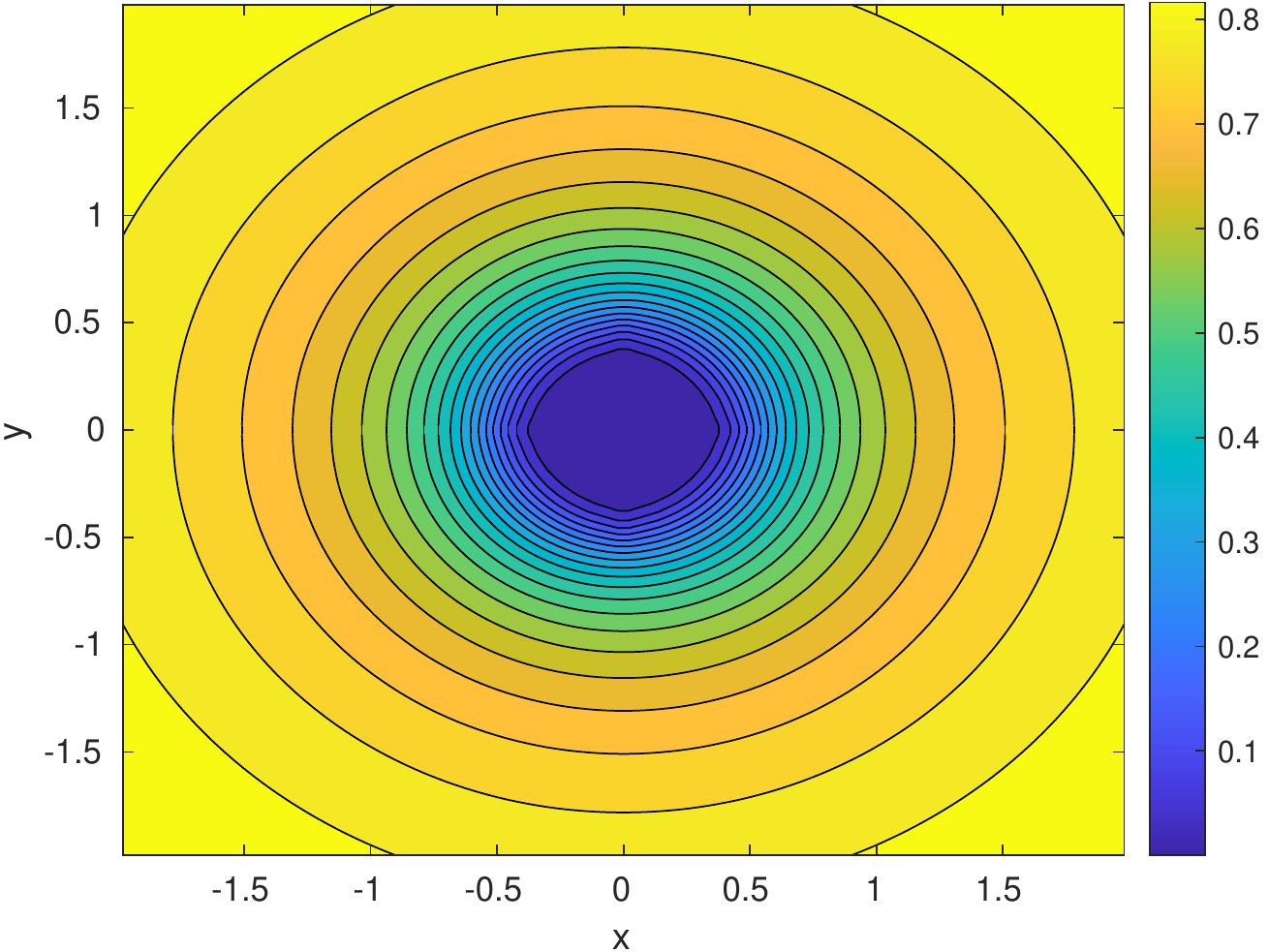}
			\caption{$\rho$}
		\end{subfigure}	
		\begin{subfigure}[b]{0.45\textwidth}
			\includegraphics[width=\textwidth]{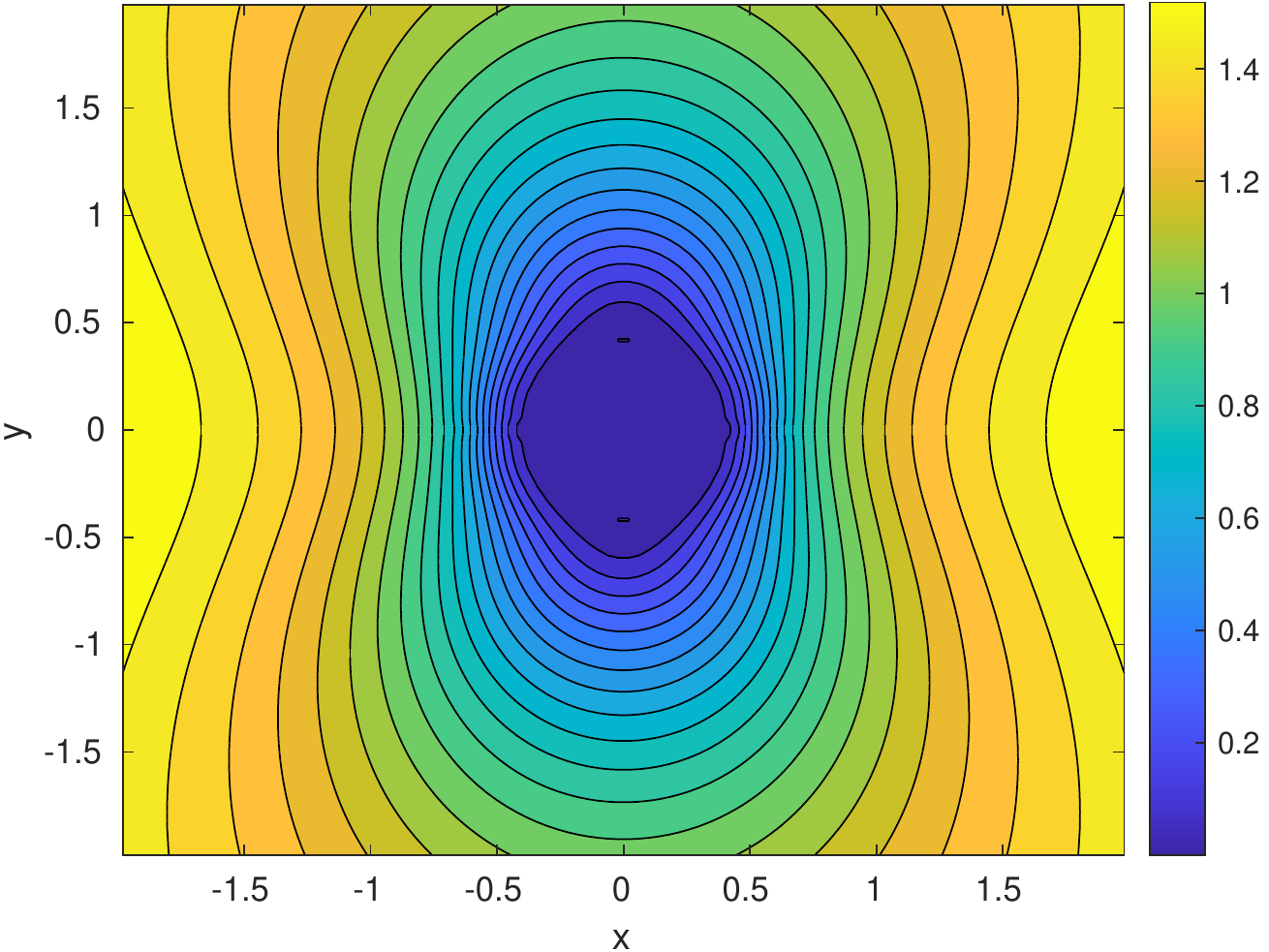}
			\caption{$p^{xx}$}
		\end{subfigure}	
		\begin{subfigure}[b]{0.45\textwidth}
			\includegraphics[width=\textwidth]{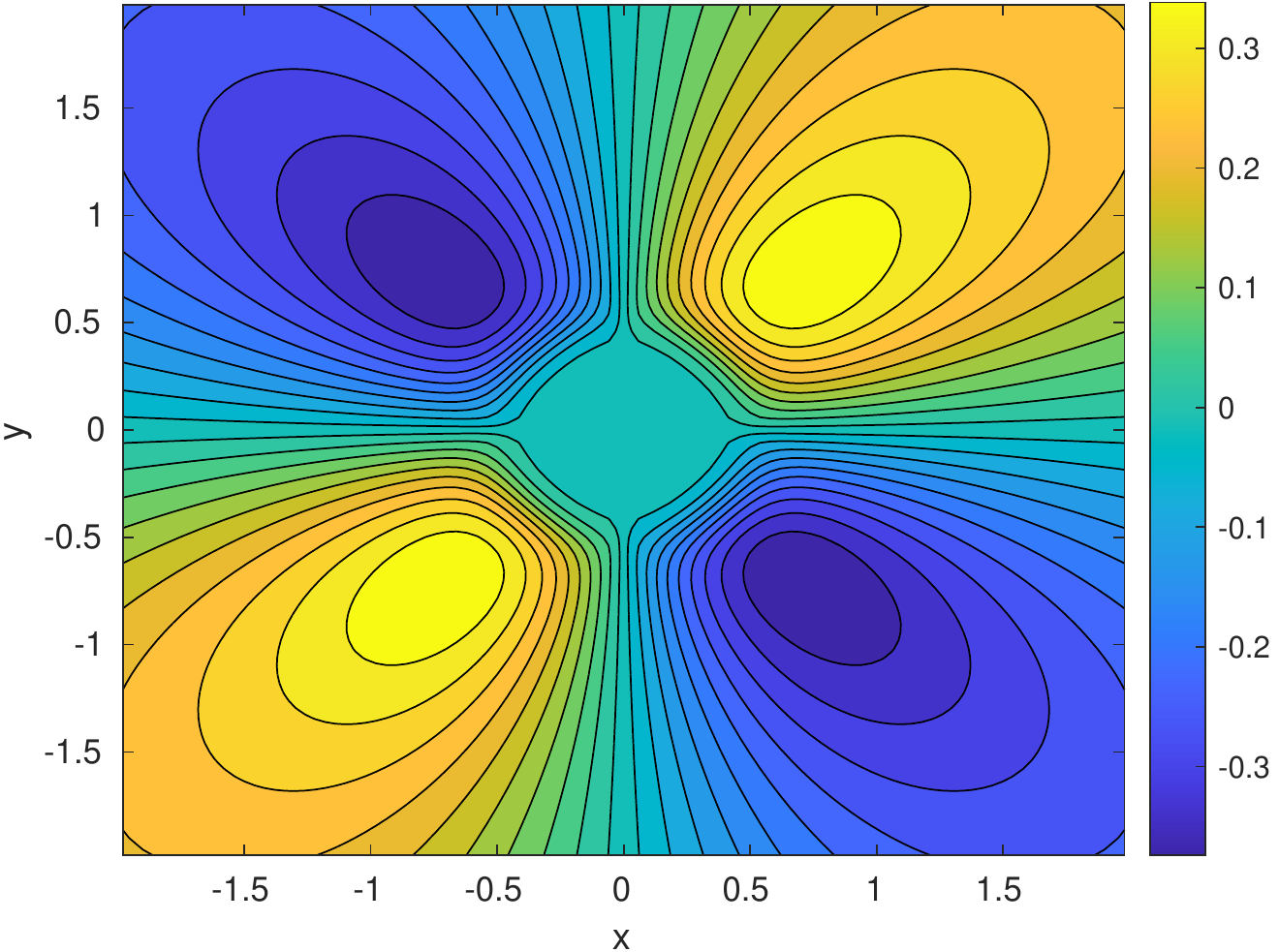}
			\caption{$p^{xy}$}
		\end{subfigure}
		\begin{subfigure}[b]{0.45\textwidth}
			\includegraphics[width=\textwidth]{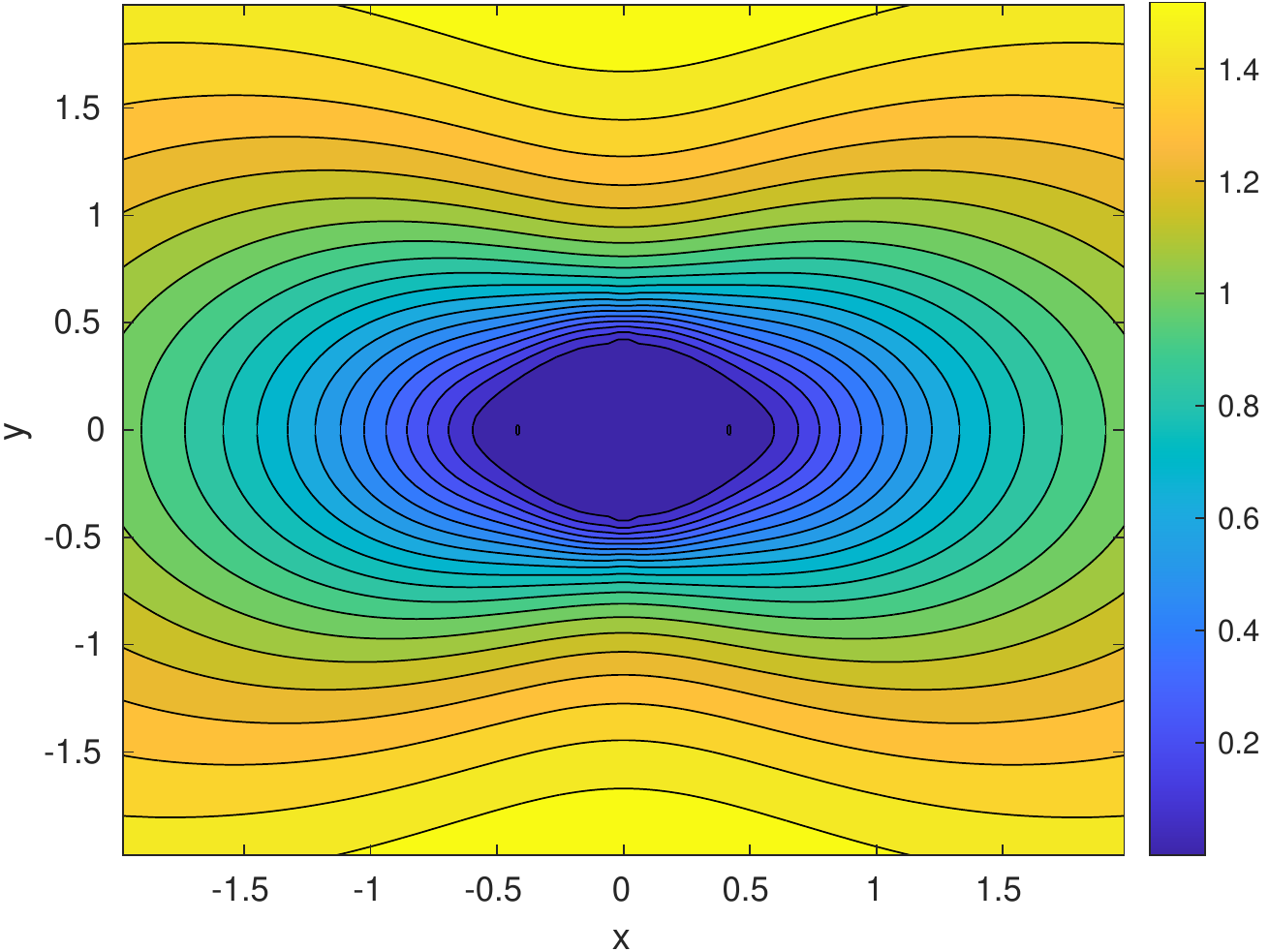}
			\caption{$p^{yy}$}
		\end{subfigure}
		\caption{Test Problem 9 (Two dimensional near vacuum test problem): Plot of density and pressure components for ESDG-O2 at time $t=0.05$ using $100\times100$ mesh.}
		\label{fig:pb11a}
	\end{figure}
	\begin{figure}[htb!]
		\centering
		\begin{subfigure}[b]{0.45\textwidth}
			\includegraphics[width=\textwidth]{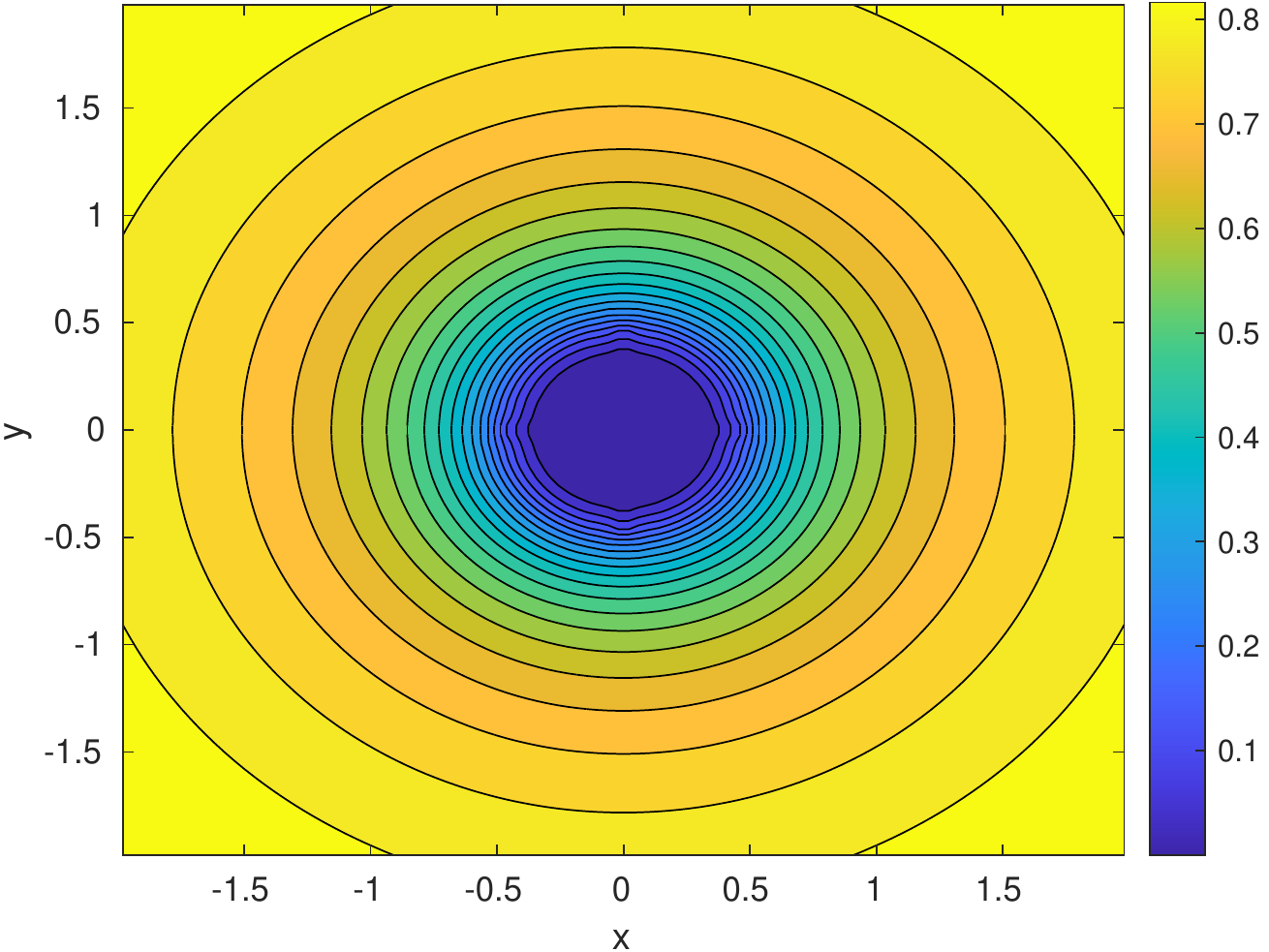}
			\caption{$\rho$}
		\end{subfigure}	
		\begin{subfigure}[b]{0.45\textwidth}
			\includegraphics[width=\textwidth]{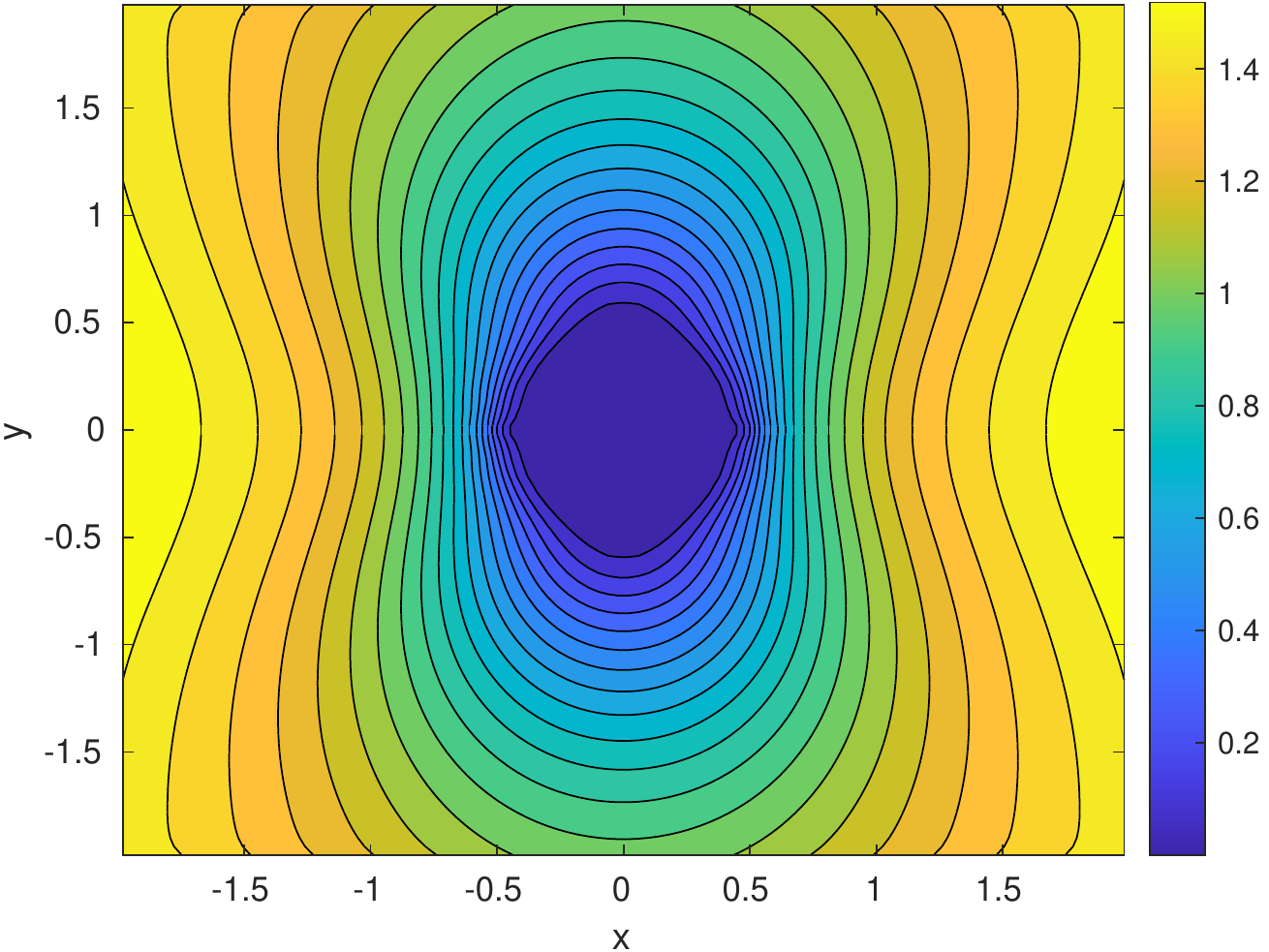}
			\caption{$p^{xx}$}
		\end{subfigure}	
		\begin{subfigure}[b]{0.45\textwidth}
			\includegraphics[width=\textwidth]{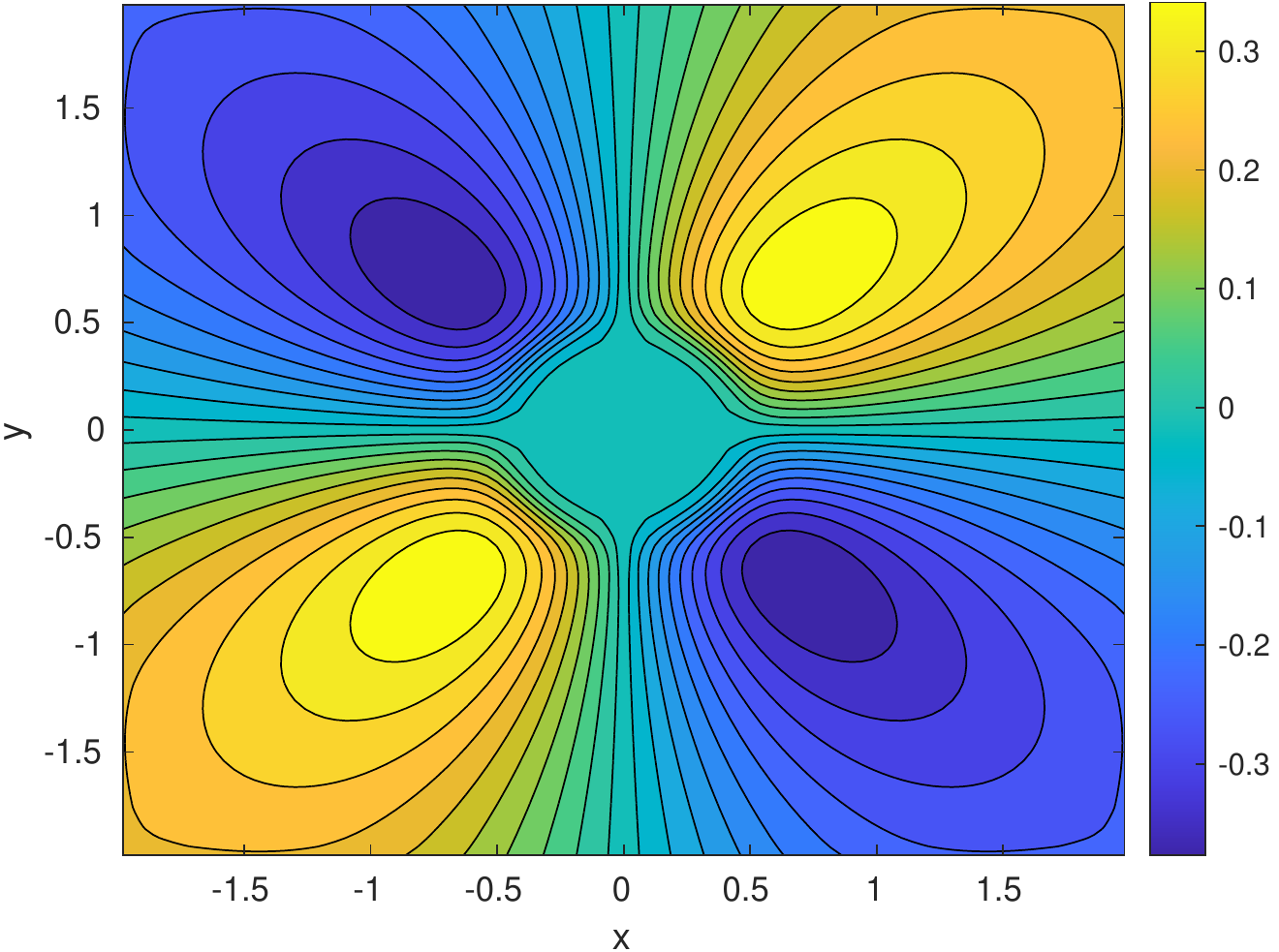}
			\caption{$p^{xy}$}
		\end{subfigure}
		\begin{subfigure}[b]{0.45\textwidth}
			\includegraphics[width=\textwidth]{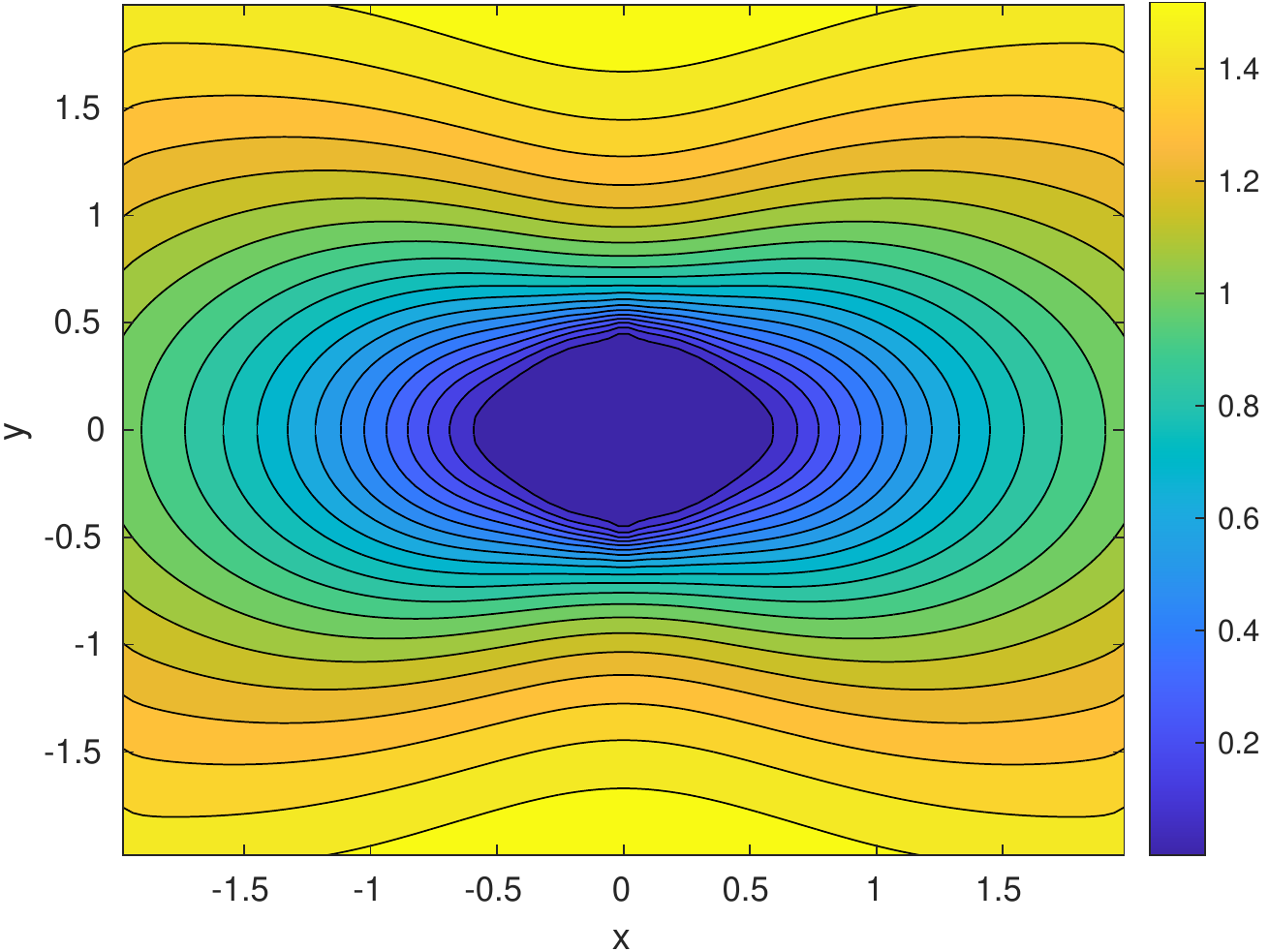}
			\caption{$p^{yy}$}
		\end{subfigure}
		\caption{Test Problem 9 (Two dimensional near vacuum test problem): Plot of density and pressure components for ESDG-O3 at time $t=0.05$ using $100\times100$ mesh.}
		\label{fig:pb11b}
	\end{figure}
The numerical results for ESDG-O2 are presented in Figure \ref{fig:pb11a}, and for ESDG-O3, they are presented in Figure \ref{fig:pb11b}. Both schemes are able to resolve the solutions with similar details and give results consistent with  \cite{meena_positivity-preserving_2020}. 
\end{example}

\begin{example}[Uniform plasma state with Gaussian source]
	We consider another test case from \cite{meena2017positivity,meena_positivity-preserving_2020}. The  computational domain $[0,4]\times[0,4]$ is assumed to be contain uniform initial state, 
	\begin{equation*}
	\rho=0.1,\,v^x=v^y=0.0,\,p^{xx}=p^{yy}=9.0,\,p^{xy}=7.0.
	\end{equation*}
	The source terms are used using a Gaussian profile given by,
	\begin{equation*}
	W(x,y ,t ) = 25 \exp\left( -200\left(( x-2 )^ 2+(y-2)^2\right) \right).
	\end{equation*}
	We compute the solution using $100\times100$ mesh, until $T=0.1$. Bound preserving limiter is not used here. Numerical results are plotted in Figure \ref{pb6a}, and \ref{pb6b} shows an an-isotropic change in the density due to the source. Figure \ref{pb6c} shows the one dimensional cut of the results along $x+y=4$.  We note that both ESDG-O2 and ESDG-O3 have similar performance.
	\begin{figure}[htb!]
		\centering
		\begin{subfigure}[b]{0.45\textwidth}
			\includegraphics[width=\textwidth]{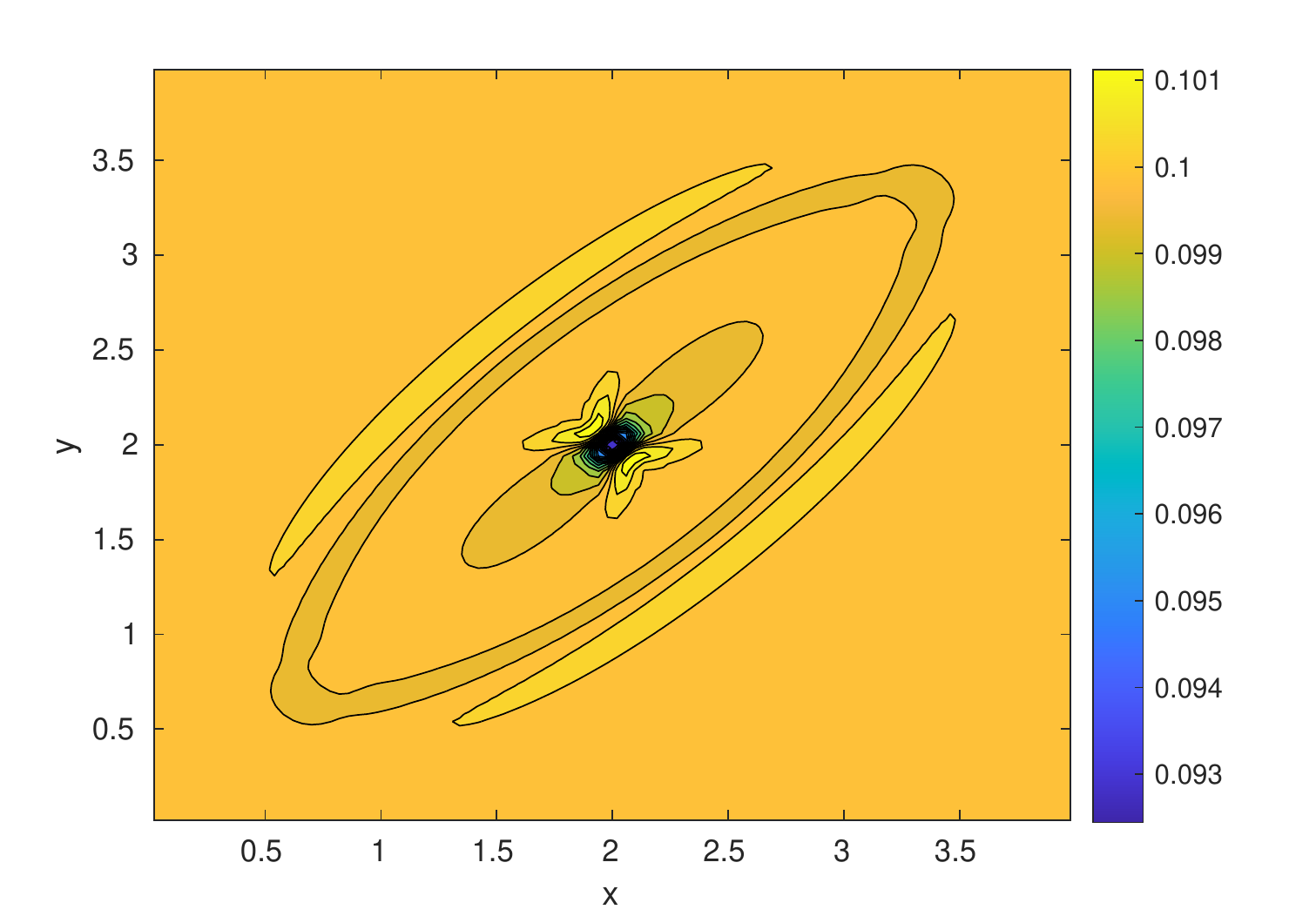}
			\caption{ESDG-O2}
			\label{pb6a}
		\end{subfigure}	
		\begin{subfigure}[b]{0.45\textwidth}
			\includegraphics[width=\textwidth]{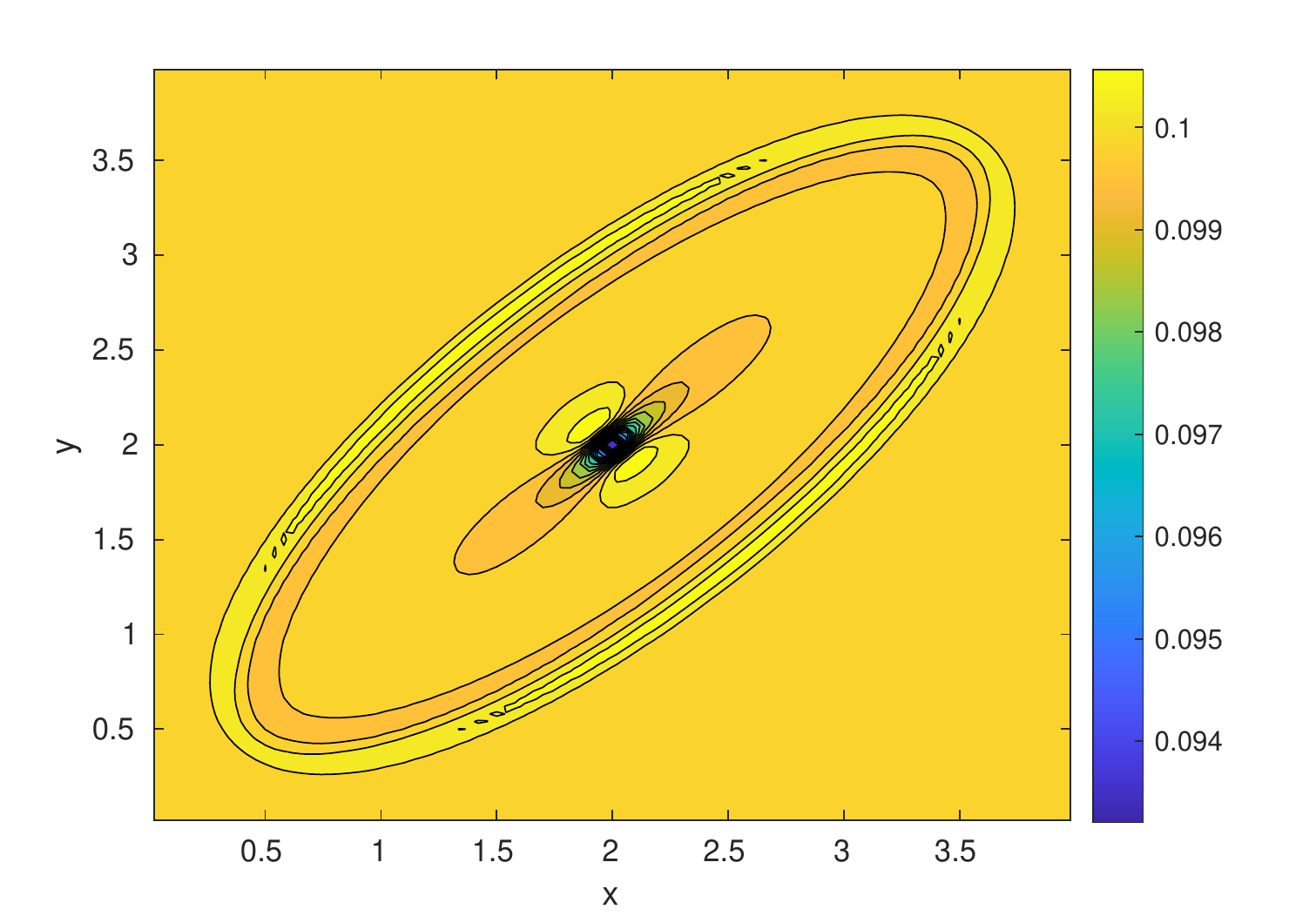}
			\caption{ESDG-O3}
			\label{pb6b}
		\end{subfigure}	
		\begin{subfigure}[b]{0.45\textwidth}
			\includegraphics[width=\textwidth]{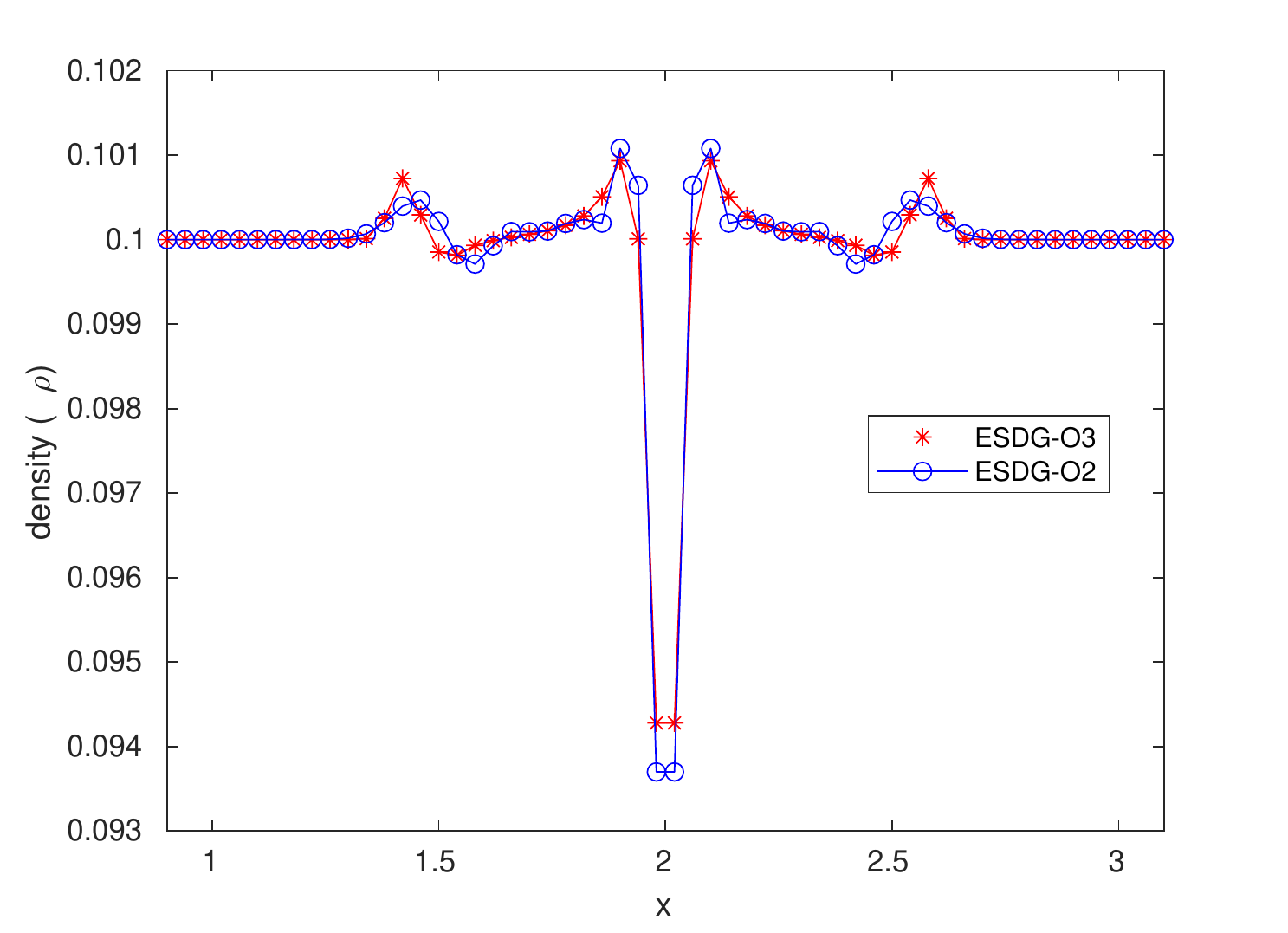}
			\caption{One dimensional cut of the computed solution (density) along the diagonal x+y=4.}
			\label{pb6c}
		\end{subfigure}
		\caption{Test Problem 10 (Uniform plasma state with Gaussian source): Plot of density using ESDG-O2 and ESDG-O3 schemes at time $t=0.1$ on $100\times100$ mesh.}
		\label{fig:pb6}
	\end{figure}
\end{example}

\begin{example}[Realistic simulation in two dimensions]
	In this test problem from \cite{berthon2015entropy}, we consider the domain  $[0,100]\times[0,100]$ filled with initial state,
	\begin{equation*}
	\rho=0.109885,\,v^x=v^y=0.0,\,p^{xx}=p^{yy}=1.0,\,p^{xy}=0.0.
	\end{equation*}
	We than consider a Gaussian source corresponds to
	\begin{equation*}
	W(x,y ,t ) =  exp\Bigg(-\Bigg(\frac{ x-50 }{10}\Bigg)^ 2-\Bigg(\frac{y-50}{10}\Bigg)^2\Bigg),
	\end{equation*}
	only in $x$-direction i.e. source in $y$-direction is set  be zero. Furthermore, an additional source $2 v_T \rho W$  ($v_T\in [0,\,1]$) is added to the energy equations (see \cite{berthon2015entropy,meena2017positivity}). We present the solutions  with $v_T=0$ and $v_T=1$ at $T=0.5$. Figure \ref{fig:pb7k1} and Figure \ref{fig:pb7k2} show the results by ESDG-O2 and ESDG-O3 schemes, respectively. We clearly observe that both schemes have similar performance at this resolution and obtained results are consistent with the results in \cite{meena2017positivity}. One dimensional cut of the solution along the diagonal $x+y=4$ in Figure \ref{pb6c} and Figure \ref{pb7c} shows that the computed density is lower when $v_T=1$.
	
	\begin{figure}[htb!]
		\centering
		\begin{subfigure}[b]{0.45\textwidth}
			\includegraphics[width=\textwidth]{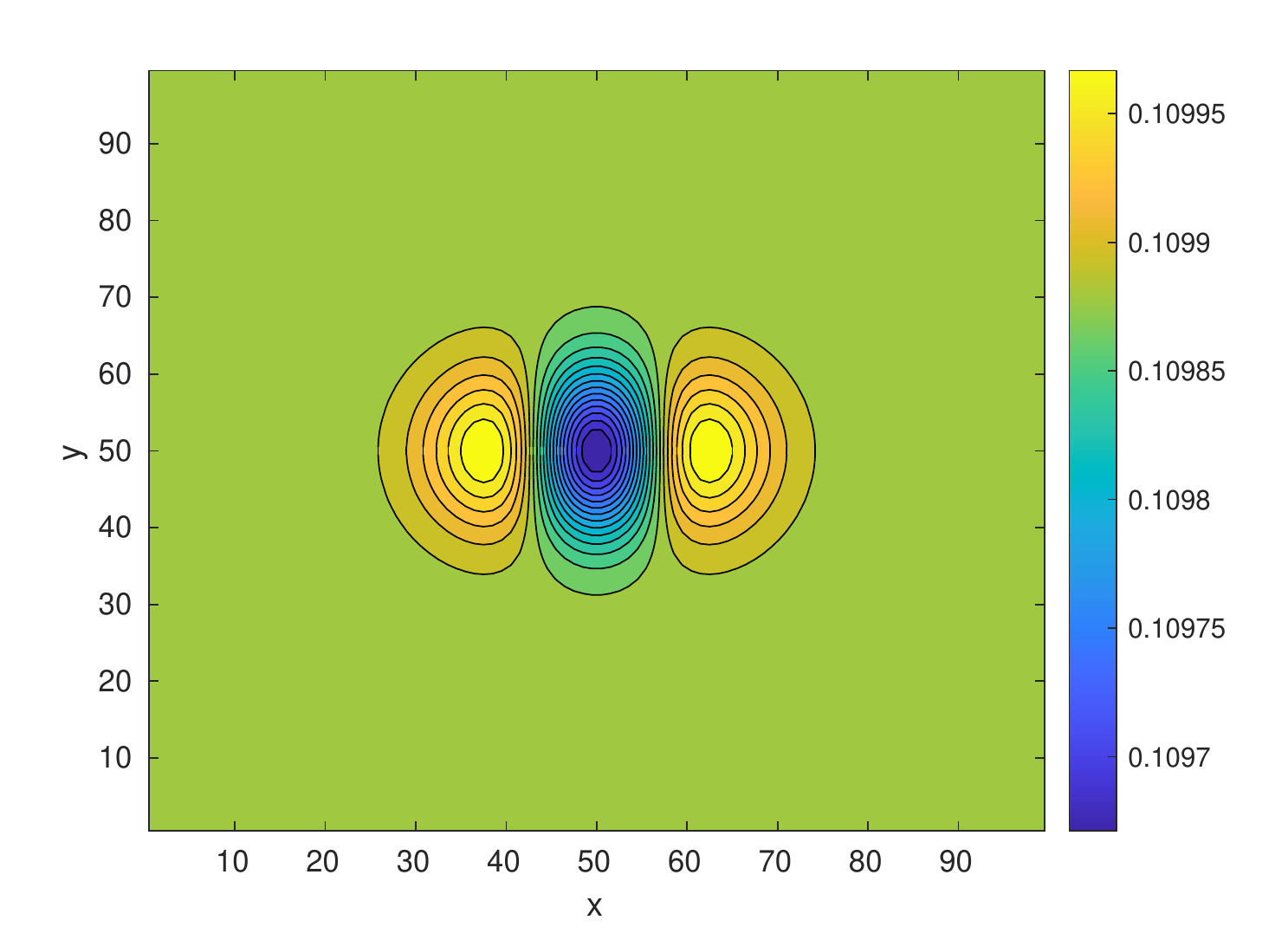}
			\caption{$v_T=0$}
			\label{pb7a}
		\end{subfigure}	
		\begin{subfigure}[b]{0.45\textwidth}
			\includegraphics[width=\textwidth]{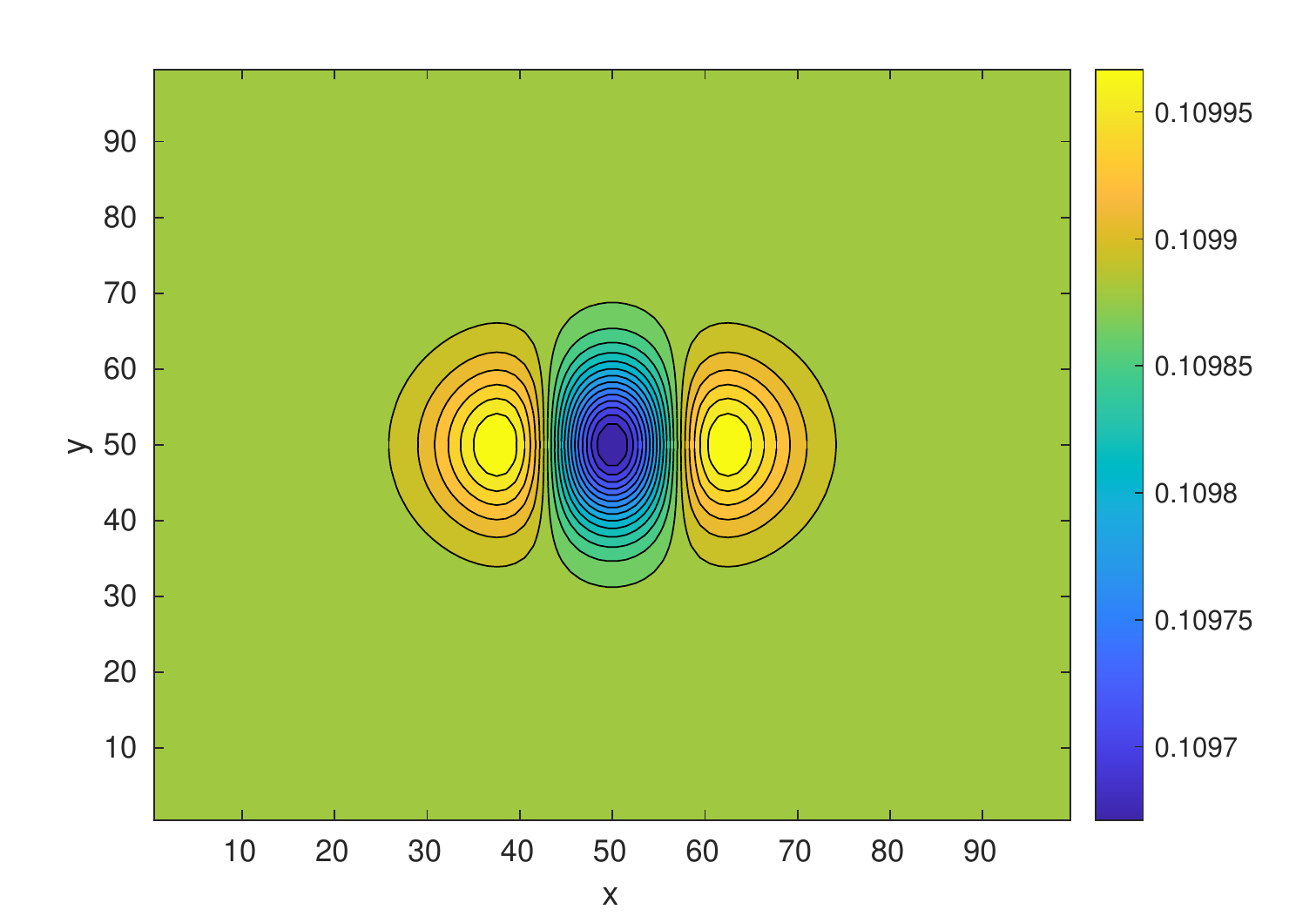}
			\caption{$v_T=1$}
			\label{pb7b}
		\end{subfigure}	
		\begin{subfigure}[b]{0.45\textwidth}
			\includegraphics[width=\textwidth]{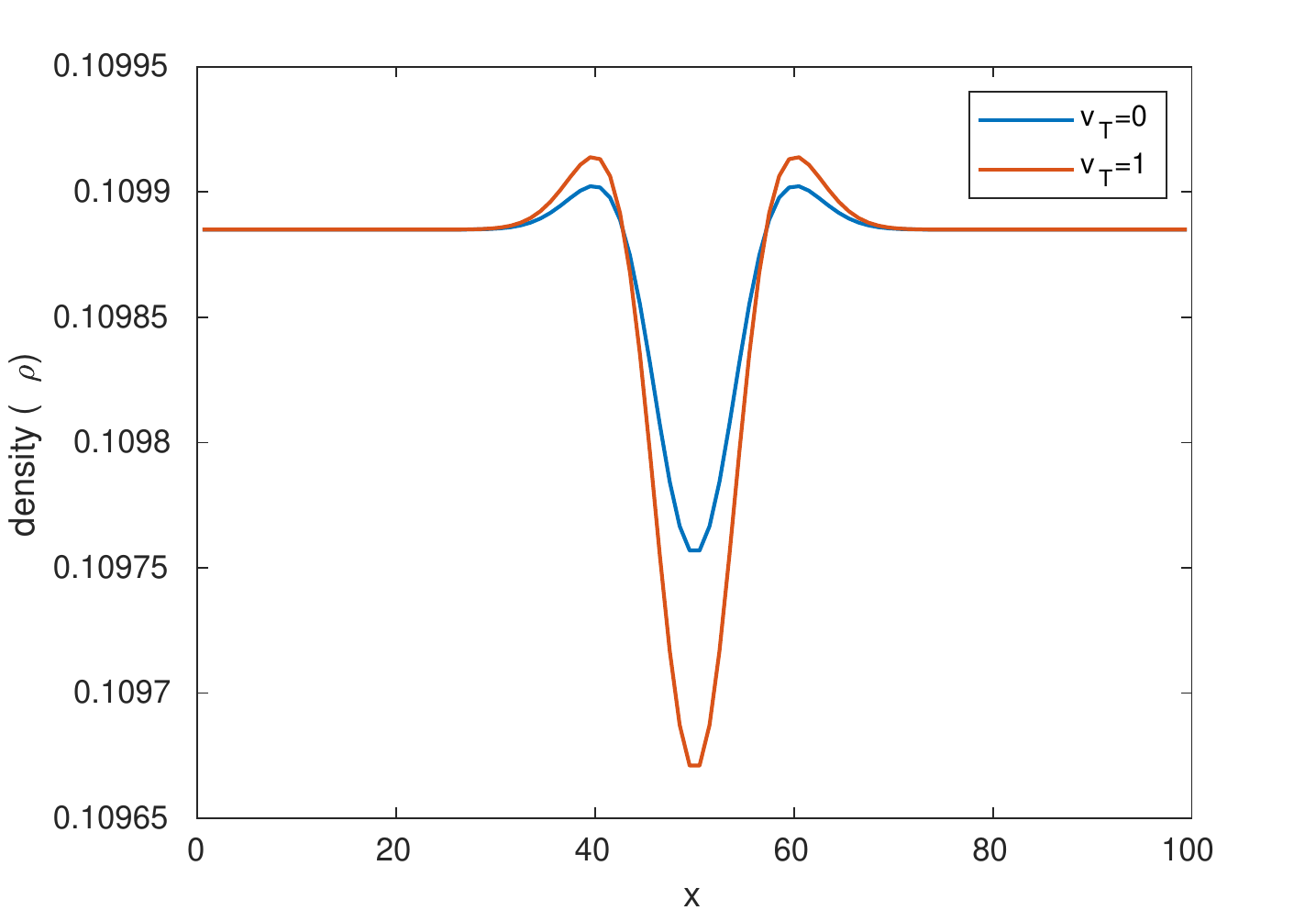}
			\caption{One dimensional cut of the computed solution (density) along the diagonal x+y=4.}
			\label{pb7c}
		\end{subfigure}
		\caption{Test Problem 11 (Realistic simulation in two dimensions): Plot of density approximated using ESDG-O2 scheme at time $t=0.5$ with $100\times100$ mesh.}
		\label{fig:pb7k1}
	\end{figure}
	\begin{figure}[htb!]
		\centering
		\begin{subfigure}[b]{0.45\textwidth}
			\includegraphics[width=\textwidth]{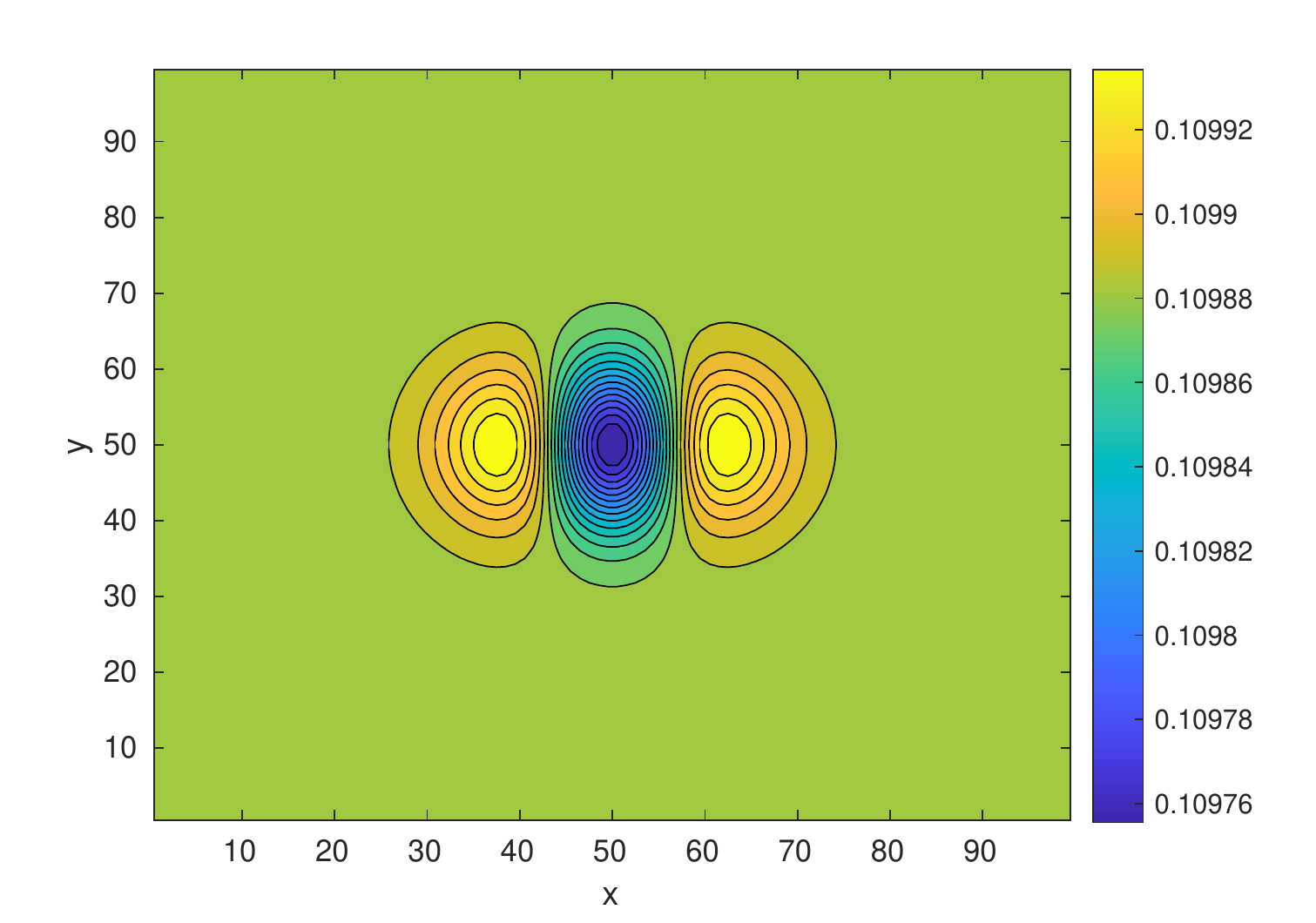}
			\caption{$v_T=0$}
			\label{pb7k2a}
		\end{subfigure}	
		\begin{subfigure}[b]{0.45\textwidth}
			\includegraphics[width=\textwidth]{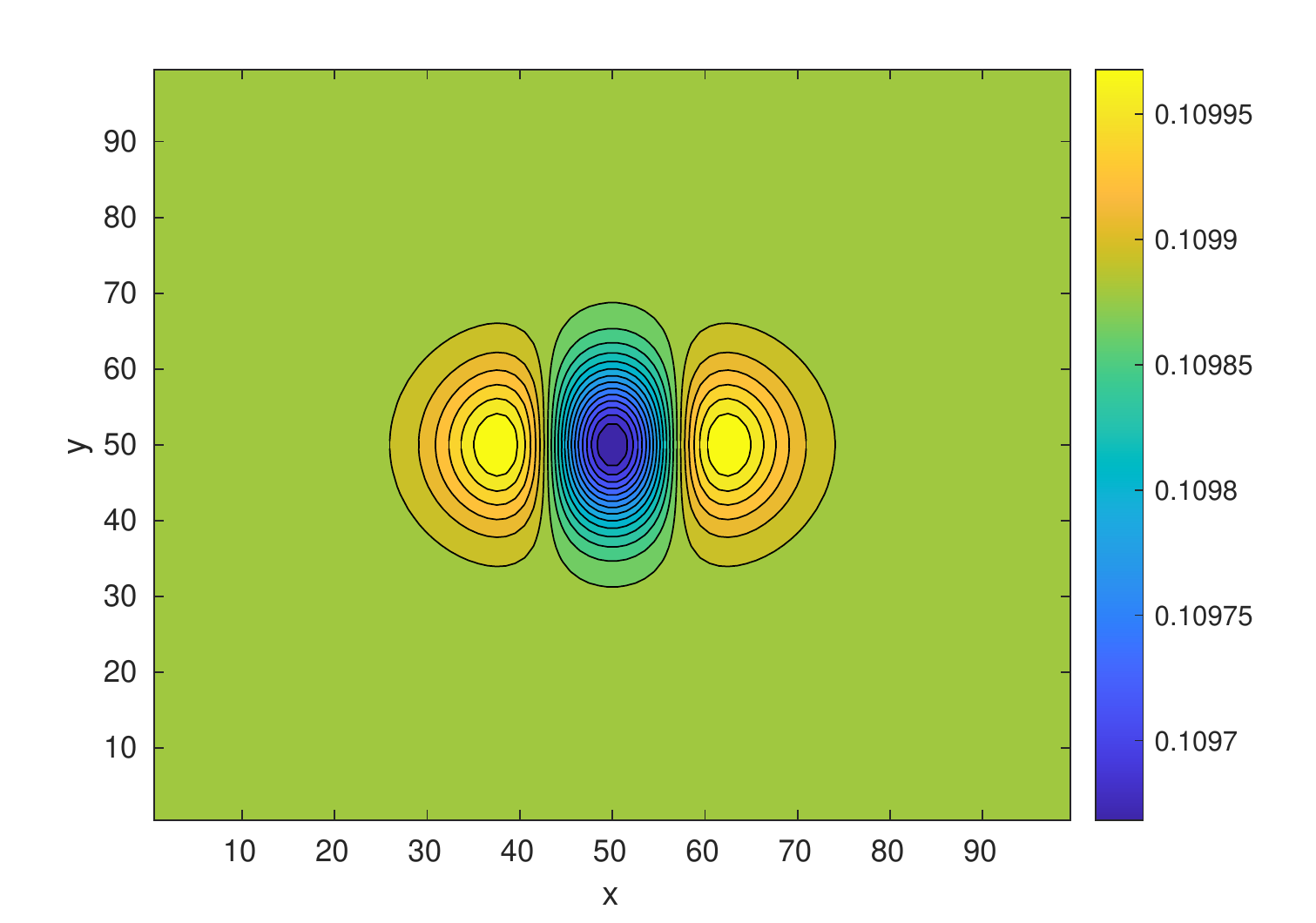}
			\caption{$v_T=1$}
			\label{pb7k2b}
		\end{subfigure}	
		\begin{subfigure}[b]{0.45\textwidth}
			\includegraphics[width=\textwidth]{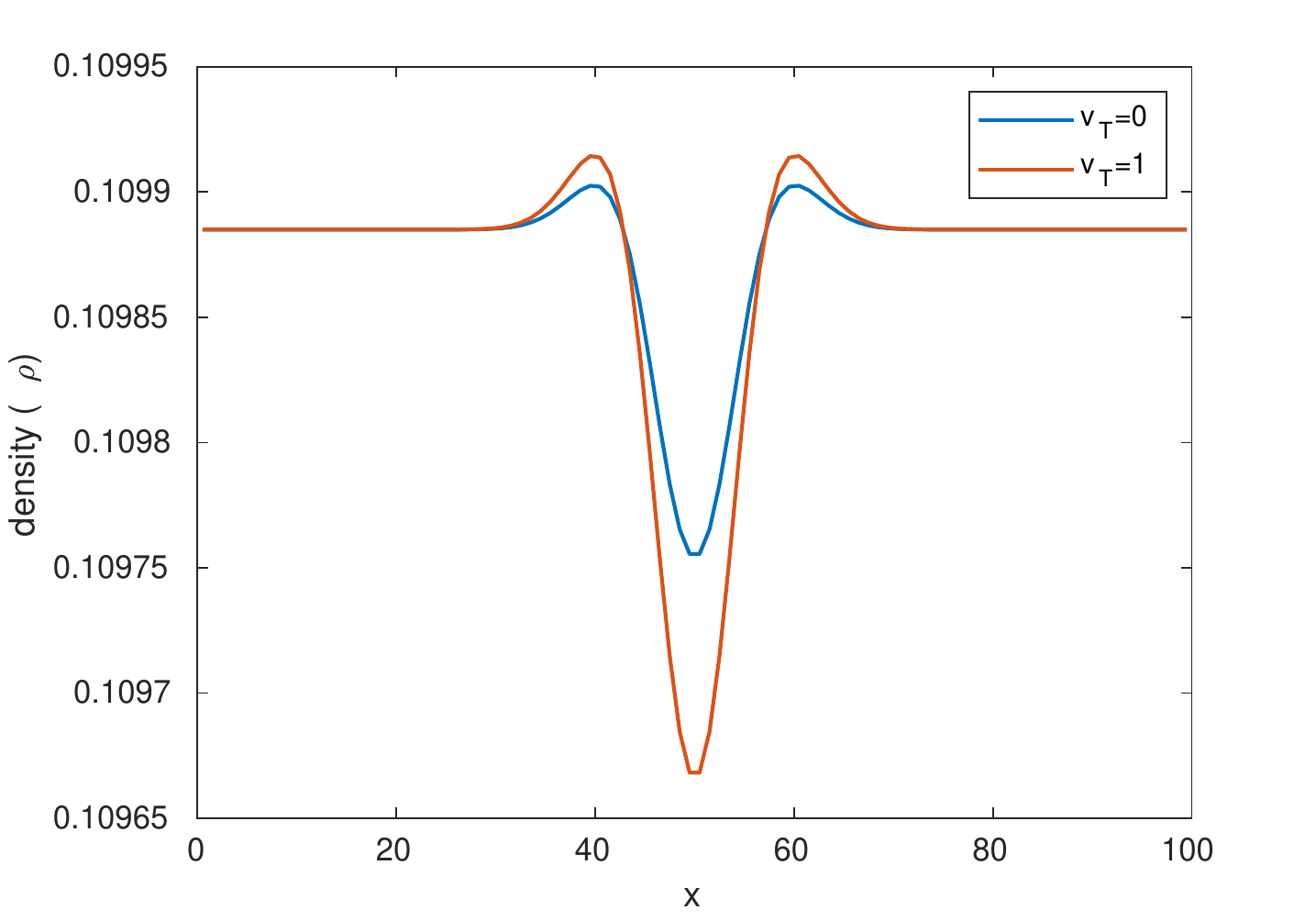}
			\caption{One dimensional cut of the computed solution (density) along the diagonal x+y=4.}
			\label{pb7k2c}
		\end{subfigure}
		\caption{Test Problem 11 (Realistic simulation in two dimensions): Plot of density approximated using ESDG-O3 scheme at time $t=0.5$ with $100\times100$ mesh.}
		\label{fig:pb7k2}
	\end{figure}
\end{example}

\section*{Conclusion}
In this article, we have considered the ten-moment equations. We first present the entropy framework for the system.  We then design the entropy stable discontinuous Galerkin scheme for the system in one and two dimensions. Two achieve the entropy stability; we use entropy conservative numerical flux in cells and entropy stable numerical flux at the cell interfaces. This is achieved by following the quadrature rules in ]\cite{chen2017entropy}. The resulting schemes are shown to be entropy stable at the semi-discrete level, with source terms. For the time discretization, we have used SSP Runge Kutta methods. These schemes are then tested on a variety of test cases in one and two dimensions. Furthermore, schemes are demonstrated to be accurate in capturing various waves and entropy stable, in both one and two dimensions.
\section*{Acknowledgment} Harish Kumar has been funded in part by SERB, DST MATRICS  grant with file No. MTR/2019/000380. 
\bibliographystyle{acm}
\bibliography{main}

\begin{appendices}
	\section{Entropy conservative flux}\label{ecfluxes}
	We denote,
	\begin{equation*}
	D=\dfrac{det(\textbf{p})}{\rho},\, \beta_{xx}=\dfrac{p^{xx}}{D},\,
	\beta_{xy}=\dfrac{p^{xy}}{D},\,\beta_{yy}=\dfrac{p^{yy}}{D},\,D_{\beta}=\beta_{xx}\beta_{yy}-\beta_{xy}^2.	\end{equation*}
	Then the entropy conservative fluxes
	$\mathbf{f}^*=[f^*_1,\,f^*_2,\,f^*_3,\,f^*_4,\,f^*_5,\,f^*_6]^T$ and $\mathbf{g}^*=[g^*_1,\,g^*_2,\,g^*_3,\,g^*_4,\,g^*_5,\,g^*_6]^T$ are given by (see \cite{sen2018entropy}.),
	\begin{equation*}
	\mathbf{f}^*=
	\begin{pmatrix}
	\rho^{ln} \bar{v}^x\\
	f^*_1\bar{v}^x+\dfrac{\bar{\rho}\bar{\beta}_{xx}}{\bar{\beta}_{xx}\bar{\beta}_{yy}-\left(\bar{\beta}_{xy}\right)^2}\\
	f^*_1\bar{v}^y+\dfrac{\bar{\rho}\bar{\beta}_{xy}}{\bar{\beta}_{xx}\bar{\beta}_{yy}-\left(\bar{\beta}_{xy}\right)^2}\\
	\left(\dfrac{\bar{\beta}_{xx}}{D_\beta^{ln}}-\overline{\left(v^x\right)^2}\right)f^*_1+2\bar{v}^xf^*_2\\
	\left(\dfrac{\bar{\beta}_{xy}}{D_\beta^{ln}}-\overline{v^x v^y}\right)f^*_1+\bar{v}^xf^*_3+\bar{v}^yf^*_2\\
	\left(\dfrac{\bar{\beta}_{yy}}{D_\beta^{ln}}-\overline{\left(v^y\right)^2}\right)f^*_1+2\bar{v}^yf^*_3
	\end{pmatrix}
	\end{equation*}
	\begin{equation*}
	\mathbf{g}^*
	=
	\begin{pmatrix}
	\rho^{ln} \bar{v}^y\\
	g^*_1\bar{v}^x+\dfrac{\bar{\rho}\bar{\beta}_{xy}}{\bar{\beta}_{xx}\bar{\beta}_{yy}-{\left(\bar{\beta}_{xy}\right)^2}}\\
	g^*_1\bar{v}^y+\dfrac{\bar{\rho}\bar{\beta}_{yy}}{\bar{\beta}_{xx}\bar{\beta}_{yy}-{\left(\bar{\beta}_{xy}\right)^2}}\\
	\left(\dfrac{\bar{\beta}_{xx}}{D_\beta^{ln}}-\overline{\left(v^x\right)^2}\right)g^*_1+2\bar{v}^xg^*_2\\
	\left(\dfrac{\bar{\beta}_{xy}}{D_\beta^{ln}}-\overline{v^x v^y}\right)g^*_1+\bar{v}^xg^*_3+\bar{v}^yg^*_2\\
	\left(\dfrac{\bar{\beta}_{yy}}{D_\beta^{ln}}-\overline{\left(v^y\right)^2}\right)g^*_1+2\bar{v}^yg^*_3
	\end{pmatrix}
	\end{equation*}

\end{appendices}
\end{document}